\documentclass{amsart}
\usepackage[all]{xy}
\usepackage{latexsym}
\usepackage{amsmath}
\usepackage{amssymb,amscd}
\usepackage{setspace}
\usepackage{mathdots}
\usepackage[mathscr]{eucal}
\usepackage{bbm}
\usepackage{stmaryrd}
\usepackage{color}

\newtheorem{theorem}{Theorem}[section] 
\newtheorem{lemma}[theorem]{Lemma}     
\newtheorem{corollary}[theorem]{Corollary}
\newtheorem{proposition}[theorem]{Proposition}
\newtheorem{remark}[theorem]{Remark}
\newtheorem{definition}[theorem]{Definition}

\numberwithin{equation}{section}
\newcommand{\fa}{\mathfrak{a}}
\newcommand{\fb}{\mathfrak{b}}
\newcommand{\p}{\mathfrak{p}}
\newcommand{\fm}{\mathfrak{m}}

\newcommand{\fC}{\mathfrak{C}}
\newcommand{\cO}{\mathcal{O}}

\newcommand{\F}{\mathbb F}

\newcommand{\N}{\mathbb N}
\newcommand{\Q}{\mathbb Q}
\newcommand{\R}{\mathbb R}
\newcommand{\C}{\mathbb C}
\newcommand{\Z}{\mathbb Z}

\newcommand{\rinj}{\mathrm{inj}}
\newcommand{\rInj}{\mathrm{Inj}}
\newcommand{\rsoc}{\mathrm{soc}}

\newcommand{\Fil}{\mathrm{Fil}}

\newcommand{\ra}{\rightarrow}
\newcommand{\lra}{\longrightarrow}

\newcommand{\GL}{\mathrm{GL}}
\newcommand{\SL}{\mathrm{SL}}

\newcommand{\bFp}{\overline{\F}_p}
\newcommand{\bQp}{\overline{\Q}_p}

\newcommand{\Sym}{\mathrm{Sym}}

\providecommand{\cInd}{\mathrm{c}\textrm{-}\mathrm{Ind}}

\newcommand{\ide}{\mathbf{1}}

\newcommand{\xto}[1][]{\xrightarrow{#1}}
\newcommand{\simto}{
\xto[\sim]} 

\newcommand{\matr}[4]{\begin{pmatrix}{#1}&{#2}\\ {#3}&{#4}\end{pmatrix}}
\newcommand{\smatr}[4]{\bigl(\begin{smallmatrix} {#1}& {#2}\\ {#3}&{#4}\end{smallmatrix}\bigl)}

\DeclareMathOperator{\Ann}{{\mathrm{Ann}}}

\DeclareMathOperator{\End}{{\mathrm{End}}}
\DeclareMathOperator{\Ext}{{\mathrm{Ext}}}

\DeclareMathOperator{\Gal}{{\mathrm{Gal}}}
\DeclareMathOperator{\Hom}{{\mathrm{Hom}}}

\DeclareMathOperator{\im}{{\mathrm{Im}}}

\DeclareMathOperator{\Ind}{{\mathrm{Ind}}}

\DeclareMathOperator{\Ker}{{\mathrm{Ker}}}

\DeclareMathOperator{\Mod}{\mathrm{Mod}}

\DeclareMathOperator{\Rep}{{\mathrm{Rep}}}

\DeclareMathOperator{\soc}{{\mathrm{soc}}}
\DeclareMathOperator{\Sp}{{\mathrm{Sp}}}

\DeclareMathOperator{\Tor}{{\mathrm{Tor}}}

\def\VV{\check{\mathbb{V}}}
\def\bn{\mathbf{n}}
\def\CG{\mathfrak{C}(\mathscr{G})}
\def\CGtor{\mathfrak{C}^{\rm fg,tor}(\mathscr{G})}
\def\Tn{\mathscr{T}_1(p^{\mathbf{n}})}

\def\fC{\mathfrak{C}}

 \newcommand{\GKdim}{\mathrm{GK}\textrm{-}\dim}

\newcommand{\quash}[1]{}

\begin{document}

\title[]
{Multiplicities of cohomological automorphic forms on $\GL_2$ and mod $p$ representations of $\GL_2(\Q_p)$}

\author{Yongquan HU}\thanks{Morningside Center of
Mathematics, Academy of Mathematics and Systems Science, Chinese Academy of
Sciences, Beijing 100190, China; University of the Chinese Academy of
Sciences, Beijing 100049, China. \\
Email: \texttt{yhu@amss.ac.cn}\\
Partially supported by National Natural Science Foundation of China Grants
 11688101; China's Recruitement Program of Global Experts,  National Center for Mathematics and Inter disciplinary Sciences and Hua Loo-Keng Key Laboratory of Mathematics, Chinese Academy of Sciences}
\thanks{
 (Secondary).}
\date{}

\maketitle

\begin{abstract}
 We prove a new upper bound for the dimension of  the space of cohomological  automorphic forms of fixed level and growing parallel weight on $\GL_2$ over a number field which is not totally real, improving the one obtained in \cite{Mar-Annals}. 
 The main tool of the proof is the mod $p$ representation theory of $\GL_2(\Q_p)$ as started by Barthel-Livn\'e and Breuil,  and  developed by Pa\v{s}k\=unas.
\end{abstract}
  \tableofcontents

\section{Introduction}

Let $F$ be a finite extension of $\Q$ of degree $r$, and $r_1$ (resp. $2r_2$) be the number of real (resp. complex) embeddings. 
Let $F_{\infty}=F\otimes_{\Q}\R$, so that $\GL_2(F_{\infty})=\GL_2(\R)^{r_1}\times\GL_2(\C)^{r_2}$.
Let $Z_{\infty}$ be the centre of $\GL_2(F_{\infty})$, $K_f$ be a compact open subgroup of $\GL_2(\mathbb{A}_f)$ and let 
\[X=\GL_2(F)\backslash \GL_2(\mathbb{A})/K_fZ_{\infty}.\] 
If ${\bf d}=(d_1,...,d_{r_1+r_2})$ is an $(r_1+r_2)$-tuple of positive even integers, we let $S_{\bf d}(K_f)$ denote the space of cusp forms on $X$ which are of cohomological type with weight ${\bf d}$.

In this paper, we are interested in understanding the asymptotic behavior of the dimension of $S_{\bf d}(K_f)$ when ${\bf d}$ varies and $K_f$ is fixed.
Define 
\[\Delta({\bf d})=\prod_{1\leq i\leq r_1}d_i\times \prod_{r_1<i\leq r_1+r_2}d_i^2.\]
When $F$ is totally real and $K_f$ is fixed, Shimizu \cite{Shi} proved that\footnote{Given $r\geq 1$ and two functions $f,g: \N^{r}\ra \N$, we write $f\ll g$ for the usual notation $f=O(g)$, meaning that there exist $M, C>0$ such that for all ${\bf d}\in\N^{r}$ with $\max\{d_i\}>M$, $f(\mathbf{d})\leq Cg(\mathbf{d})$. We write $f\sim g$ if both $f\ll g$ and $g\ll f$ hold. In case  the constants $M, C$ depend on other inputs $*$, we write $f\sim_{*}g$ or $f\ll_{*}g$ to indicate this.} 
\[\dim_{\C} S_{\bf d}(K_f)\sim C\cdot \Delta({\bf d})\]
for some constant $C$ independent of ${\bf d}$. However, if $F$ is not totally real, the actual growth rate of $\dim_{\C}S_{\bf d}(K_f)$ is still a mystery; see the discussion below when $F$ is imaginary quadratic. 
 
The main result of this paper is the following (see Theorem \ref{theorem-global-A} for a slightly more general statement). 
\begin{theorem}\label{theorem-intro-A}
If $F$ is not totally real and ${\bf d}=(d,...,d)$ is a parallel weight with $d\geq 2$ even, then for any fixed $K_f$, we have 
\[\dim_{\C} S_{\bf d}(K_f)\ll_{\epsilon,K_f}d^{r-1/2+\epsilon}.\]
\end{theorem}

To compare our result with the previous ones, let us restrict to the case when $F$ is imaginary quadratic. In \cite{FGT}, Finis, Grunewald and Tirao  proved the bounds
\[d\ll \dim_{\C} S_{\bf d}(K_f)\ll \frac{d^2}{\ln d},\ \ {\bf d}=(d,d)\]
using base change and the trace formula respectively (the lower bound is conditional on $K_f$, see \cite{FGT}). In \cite{Mar-Annals}, Marshall has improved the upper bound  to be 
\begin{equation}\label{equation-intro-Marshall}\dim_{\C} S_{\bf d}(K_f)\ll_{\epsilon,K_f} d^{5/3+\epsilon}\end{equation}
while our Theorem \ref{theorem-intro-A} gives 
\[\dim_{\C} S_{\bf d}(K_f)\ll_{\epsilon,K_f} d^{3/2+\epsilon},\]
hence  a saving by a power $d^{1/6}$.
It is worth to point out that such a power saving is quite rare for tempered automorphic forms. Indeed, purely analytic methods, such as the trace formula, only allow to strengthen the trivial bound by a power of $\ln d$, see \cite{FGT, S-Xue}. We refer to the introduction of \cite{Mar-Annals} for a discussion on this point and a collection of known results.

 Finally, let us mention that    the experimental data of \cite{FGT} 
(when $F$ is imaginary quadratic) suggests that the actual growth rate of $\dim_{\C}S_{\bf d}(K_f)$ is probably $d$. We hope to return to this problem in future work. \medskip

Let us first explain Marshall's proof of the bound \eqref{equation-intro-Marshall}. 
It consists of two main steps, the first of which is to   convert the problem to bounding the dimension of  certain group cohomology of  Emerton's completed cohomology spaces $H^j$ (with mod $p$ coefficients) and the second one is to establish this bound. For the first step, he used the (generalized) Eichler-Shimura isomorphism, Shapiro's lemma and a fundamental spectral sequence due to Emerton. For the second, he actually  proved a bound  in a more general setting  which applies typically to  $H^j$. To make this precise, let us mention a key intermediate result in this step   (stated in the simplest version). Let $p$ be a prime number and define \[K_1=\matr{1+p\Z_p}{p\Z_p}{p\Z_p}{1+p\Z_p}, \ \ \ T_1(p^n)=\matr{1+p\Z_p}{p^n\Z_p}{p^n\Z_p}{1+p\Z_p}.\]
Let  $Z_1\cong 1+p\Z_p$ be the centre of $K_1$. Also let $\F$ be a sufficiently large finite extension of $\F_p$.  By a careful and involved analysis of the structure of finitely generated torsion modules  over the Iwasawa algebra $\Lambda:=\F[[K_1/Z_1]]$,  Marshall proved the following (\cite[Prop. 5]{Mar-Annals}): if $\Pi$ is a smooth admissible $\F$-representation of $K_1/Z_1$ which is cotorsion\footnote{That is, the Pontryagin dual $\Pi^{\vee}:=\Hom_{\F}(\Pi,\F)$ is torsion  as an $\F[[K_1/Z_1]]$-module.}, then for any $i\geq 0$, 
\begin{equation}\label{equation-intro-Marshall-2}\dim_{\F}H^i(T_1(p^n)/Z_1, \Pi)\ll p^{4n/3}.\end{equation}

Our proof of Theorem \ref{theorem-intro-A} follows closely the above strategy. Indeed, the first step is identical to Marshall's.  Our main innovation is in the second step  by improving the bound \eqref{equation-intro-Marshall-2}. 
The key observation is that Emerton's completed cohomology is not just a  representation of $K_1$, but also a representation of $\GL_2(\Q_p)$, which largely narrows the possible shape of ${H}^j$. This fact was already observed in \cite{Mar-Annals} and used \emph{once}\footnote{we mean the trick of `change of groups', see \S\ref{subsection-changegroups}} when deriving \eqref{equation-intro-Marshall} from \eqref{equation-intro-Marshall-2}. However, the mod $p$ representation theory of $\GL_2(\Q_p)$ developed by Barthel-Livn\'e \cite{BL}, Breuil \cite{Br03} and Pa\v{s}k\=unas \cite{Pa13,Pa15}, allows us to make the most of the action of $\GL_2(\Q_p)$ and prove the following result (see Theorem \ref{theorem-higher-r}).  
\begin{theorem}\label{theorem-intro-B}
Let $\Pi$ be a smooth admissible $\F$-representation of $\GL_2(\Q_p)$ on which $Z_1$ acts trivially. Assume that $\Pi$ is  cotorsion  as a $\Lambda$-module. Then for any $i\geq 0$,
\[\dim_{\F}H^i({T_1(p^n)}/Z_1,\Pi)\ll np^{n}.\]
\end{theorem}
We obtain the above bound by using numerous results of the mod $p$ representation theory of $\GL_2(\Q_p)$.  First, the classification theorems of \cite{BL} and \cite{Br03} allow us to control the dimension of invariants for irreducible $\pi$ (i.e. when $i=0$), in which case we prove 
\begin{equation}\label{equation-Hu-irr}\dim_{\F}H^0(T_1(p^n)/Z_1,\pi)\ll n.\end{equation}
In fact, to do this we also need  more refined structure theorems due to Morra \cite{Mor11,Mor}. Second, the theory of Pa\v{s}k\=unas \cite{Pa13} allows us to pass to general admissible cotorsion representations. To explain this, let us assume moreover that all the Jordan-H\"older factors of $\Pi$ are isomorphic to a given supersingular irreducible representation $\pi$.  In \cite{Pa13} Pa\v{s}k\=unas  studied the universal deformation of $\pi^{\vee}$ and showed that the universal deformation space (with mod $p$ coefficients) is three dimensional. 
We show that the admissibility  and cotorsion conditions imposed on $\Pi$   force that $\Pi^{\vee}$ is a deformation of $\pi^{\vee}$ over a one-dimensional space. Knowing this, the case $i=0$ of Theorem \ref{theorem-intro-B} follows easily from \eqref{equation-Hu-irr}.

To  prove Theorem \ref{theorem-intro-B}  for higher cohomology degrees and to generalize it to a finite product of $\GL_2(\Q_p)$ (which is essential for our application), we need to solve several complications caused by the additional requirement of carrying an action of $\GL_2(\Q_p)$. 
 In \cite{CE-Annals,Mar-Annals} the higher cohomology degree case is treated by the standard dimension-shifting argument,  for which one needs to consider admissible representations $\Pi$ which are  not necessarily cotorsion, that is, the Pontryagin dual $\Pi^{\vee}$  has a positive rank  over $\Lambda$. Using the bound in the torsion case, one is reduced to consider torsion-free $\Pi^{\vee}$. The usual argument (as in \cite[\S3.2]{Mar-Annals}) uses the existence of morphisms $\Lambda^s\ra \Pi^{\vee}$ and $\Pi^{\vee}\ra \Lambda^s$ with torsion cokernels, where $s$ is the $\Lambda$-rank of $\Pi^{\vee}$. However, these are only  morphisms of $\Lambda$-modules, so the bound for torsion modules does not apply to these cokernels. This issue makes the cohomology of general torsion-free modules difficult to control. To solve this, we introduce a special class of (coadmissible) torsion modules,  called \emph{elementary}, whose higher degree cohomologies are zero except in degree $i=1$ and can be determined from its degree $0$ cohomology. We show that $\Pi^{\vee}$ has a resolution by elementary torsion modules. The proof uses a generalization of  an important construction of   Breuil-Pa\v{s}k\=unas \cite{BP} for $\GL_2(\Q_p)$ to a finite product of $\GL_2(\Q_p)$, which we carry out in the appendix \S\ref{section-appendix}.

\medskip


 \subsubsection*{Notation} Throughout the paper, we fix a prime $p$ and a finite extension $\F$ over $\F_p$ taken to be sufficiently large.

\subsubsection*{Acknowledgement} 
Our debt to the work of Vytautas Pa\v{s}k\=unas and Simon Marshall will be obvious to the reader.  We also thank Marshall for his comments on an earlier draft. We thank Pa\v{s}k\=unas for pointing out a mistake in a previous version and for many discussions around it. We also thank Pa\v{s}k\={u}nas and Lue Pan for discussions around the material in \S5.1.3. Finally we thank the referee whose  careful reading of the manuscript and insightful remarks helped  to make the paper substantially more readable.  

\section{Non-commutative Iwasawa algebras}
\label{section-Iwasawa}

Let $G$ be a $p$-adic analytic group of dimension $d$ and  $G_0$ be  an open compact subgroup of $G$. We assume that $G_0$ is uniform and pro-$p$.  Let $\Lambda$  be the \emph{Iwasawa algebra} of $G_0$ over $\F$, namely
\[\Lambda :=\F[[G_0]]=\varprojlim \F[G_0/N]\]
where the inverse limit is taken over  the open normal subgroups $N$ of $G_0$. It is a Noetherian local integral domain (\cite[\S3.6]{AB}).
A  finitely generated (left) $\Lambda$-module is said to have \emph{codimension} $c$ if $\Ext^i_{\Lambda}(M,\Lambda)=0$ for all $i<c$ and is non-zero for $i=c$; the codimension of the zero module is defined to be $\infty$. We denote the codimension by $j_{\Lambda}(M)$. If $M$ is non-zero, then $j_{\Lambda}(M)\leq d$. For our purpose, it is more convenient to use the   \emph{canonical dimension} of $M$ defined by
\[\delta_{\Lambda}(M)=d-j_{\Lambda}(M).\]
If $0\ra M'\ra M\ra M''\ra0$ is a short exact sequence of finitely generated $\Lambda$-modules, then 
\begin{equation}\label{equation-delta-short}\delta_{\Lambda}(M)=\max\{\delta_{\Lambda}(M'),\delta_{\Lambda}(M'')\}.\end{equation}

  If $M$ is a finitely generated $\Lambda$-module, we have the notion of \emph{Gelfand-Kirillov} dimension of $M$, defined to be the growth rate of the function $\dim_{\F}M/J^nM$, where $J$ denotes the maximal ideal of $\Lambda$.  We have the following important fact (\cite[\S5.4]{AB}).
\begin{theorem}\label{theorem-canonical=GK}
  For all finitely generated $\Lambda$-modules $M$, the canonical dimension and the Gelfand-Kirillov dimension of $M$ coincide.
\end{theorem}

For $n\geq 0$, define inductively $G_{n+1}:=\overline{G_{n}^p[G_{n},G_0]}$ which are normal subgroups of $G_0$; the decreasing chain   $G_0\supseteq G_1\supseteq \cdots$ is called  the \emph{lower $p$-series} of $G_0$, see \cite[\S2.4]{AB}. Since $G_0$ is uniform, we have $|G_n:G_{n+1}|=p^d$. With this notation, the utility of   Theorem \ref{theorem-canonical=GK} is the following result (see \cite[Prop. 2.17]{EP18} and its proof).  
\begin{corollary}\label{corollary-CE}
Let $M$ be a finitely generated $\Lambda$-module with $\delta_{\Lambda}(M)=c$. Then there are real numbers $a\geq b>0$ such that
\begin{equation}\label{equation-CE}
bp^{cn}+O(p^{(c-1)n})\leq \dim_{\F}H_0(G_n,M)\leq ap^{cn}+O(p^{(c-1)n}).\end{equation}
Moreover, we have a uniform lower bound $b\geq 1/c!$.
\end{corollary}

\begin{proposition}\label{proposition-GKdim}
Let $M$ be a finitely generated $\Lambda$-module and  $\phi:M\ra M$ be an endomorphism of $\Lambda$-modules. Assume that $\bigcap_{k\geq 1}\phi^k(M)=0$. Then one of the following holds:
\begin{enumerate}
\item[(i)]  $\phi$ is nilpotent and $\delta_{\Lambda}(M)=\delta_{\Lambda}(M/\phi(M))$;
\item[(ii)] $\phi$ is not nilpotent and for $k_0\gg1$, 
\begin{equation}\label{equation-delta-phi-M}
\delta_{\Lambda}(M)=\max\big\{\delta_{\Lambda}(M/\phi(M)), \delta_{\Lambda}(\phi^{k_0}(M)/\phi^{k_0+1}(M))+1\big\}.\end{equation}
\end{enumerate}
In any case, $\delta_{\Lambda}(M)\leq \delta_{\Lambda}(M/\phi(M))+1$.
\end{proposition}
\begin{remark}\label{remark-condition-phi}
It would be more natural to impose the condition $\phi(M)\subset JM$. We consider the present one for the following reasons. On the one hand,  in practice we do need to deal with $\phi$ such that $\bigcap_{k\geq1}\phi^k(M)=0$ but $\phi(M)\nsubseteq JM$. On the other hand, since $M$ is finitely generated, $M/JM$ is finite dimensional over $\F$, hence the condition $\bigcap_{k\geq 1}\phi^k(M)=0$ implies $\phi^k(M)\subset JM$ for $k\gg1$.
\end{remark}

\begin{proof}
We assume first that $\phi$ is nilpotent, say $\phi^{k_0}=0$ for some $k_0\geq 1$.  Then $M$ admits a finite filtration by $\phi^k(M)$ (for $0\leq k\leq k_0$). Since each of the graded pieces is a quotient of $M/\phi(M)$, the assertion follows from \eqref{equation-delta-short}.
 
Now assume that $\phi$ is not nilpotent, so by Lemma \ref{lemma-phi} below $\phi$ induces an injection $\phi^{k_0}(M)\ra \phi^{k_0}(M)$ for some $k_0\gg1$. It is clear that the RHS of \eqref{equation-delta-phi-M} does not depend on the choice of $k_0$. The above argument shows that \[\delta_{\Lambda}(M/\phi^{k_0}(M))=\delta_{\Lambda}(M/\phi(M)).\] Hence, by \eqref{equation-delta-short} applied to the short exact sequence $0\ra \phi^{k_0}(M)\ra M\ra M/\phi^{k_0}(M)\ra0$, it suffices to show 
\[\delta_{\Lambda}(\phi^{k_0}(M))=\delta_{\Lambda}\big(\phi^{k_0}(M)/\phi^{k_0+1}(M)\big)+1.\]
That is, by replacing $M$ by $\phi^{k_0}(M)$, we may assume $\phi$ is injective and need to show $\delta_{\Lambda}(M)=\delta_{\Lambda}(M/\phi(M))+1$. Under the assumption $\bigcap_{k\geq 1}\phi^k(M)=0$, this follows from \cite[Lem. A.15]{GN} using Remark \ref{remark-condition-phi}.   
\end{proof}

\begin{lemma}\label{lemma-phi}
Let $M$ be a finitely generated $\Lambda$-module and $\phi:M\ra M$ be an endomorphism of $\Lambda$-modules.  Then one of the following holds:
\begin{enumerate}
\item[(i)] $\phi$ is nilpotent; 
\item[(ii)] $\phi$ is not nilpotent  and for $k_0\gg0$, $\phi$ induces an injection $\phi^{k_0}(M)\ra \phi^{k_0}(M)$.
\end{enumerate}
\end{lemma}
\begin{proof}
Let $M[\phi^{\infty}]\subset M$ denote the submodule $\cup_{k\geq 1}\ker(\phi^k)$. Since $\Lambda$ is Noetherian, $M[\phi^{\infty}]$ is finitely generated, so there exists $k_0\gg 1$ such that $M[\phi^{\infty}]=M[\phi^{k_0}]$. If $M=M[\phi^{k_0}]$, then $\phi$ is nilpotent; otherwise, $\phi$ is not nilpotent, and $\phi:\phi^{k_0}(M)\ra \phi^{k_0}(M)$ is injective. 
\end{proof}

\subsection{Torsion vs torsion-free}  As recalled above, $\Lambda$ is a Noetherian local integral domain. Let $\mathcal{L}$ be the skew field of fractions of $\Lambda$ (\cite[\S3.6]{AB}). If $M$ is a finitely generated $\Lambda$-module, then $\mathcal{L}\otimes_{\Lambda} M$ is a finite dimensional $\mathcal{L}$-vector space, and we define the rank of $M$ to be the dimension  of this vector space.  We see that rank is additive in short exact sequences and that $M$ has rank $0$ if and only if $M$ is  torsion as a $\Lambda$-module. 

Let $\cO=W(\F)$ be the ring of Witt vectors with coefficients in $\F$. 
Similar to $\Lambda=\F[[G_0]]$, we may form the Iwasawa algebras 
\[\widetilde{\Lambda}:=\cO[[G_0]]=\varprojlim_{N\triangleleft G_0\ \mathrm{open}}\cO[G_0/N],\ \ \ \widetilde{\Lambda}_{\Q_p}=\widetilde{\Lambda}\otimes_{\Z_p}\Q_p.\]
They are both integral domains. 
 Let  $\mathcal{L}_{\Q_p}$ be the skew field of fractions of $\widetilde{\Lambda}_{\Q_p}$. 
If $M$ is a finitely generated module over $\widetilde{\Lambda}_{\Q_p}$, we define its rank as above and the analogous facts hold.
 
Recall the following simple fact, see \cite[Lem. 1.17]{CE-survey}.
\begin{lemma}\label{lemma-p-free}
Let $M$ be a finite generated $\widetilde{\Lambda}$-module which is furthermore $p$-torsion free. Then $M\otimes_{\Z_p}\Q_p$ is  a torsion $\widetilde{\Lambda}_{\Q_p}$-module if and only if $M\otimes_{\Z_p}\F_p$ is  a torsion $\Lambda$-module. 
\end{lemma}

\section{Mod $p$ representations of $\GL_2(\Q_p)$} \label{section-Paskunas}

\textbf{Notation.}  Let $p$ be a prime\footnote{The assumption $p\geq 5$ is not always necessary,   but for convenience we make this assumption throughout the paper.} $\geq 5$, $G=\GL_2(\Q_p)$, $K=\GL_2(\Z_p)$, $Z$ be the centre of $G$, $T$ be the diagonal torus, and $B=\matr{*}*0*$ the upper Borel subgroup.

 Let $\Rep_{\F}(G)$  denote the category of smooth $\F$-representations of $G$ with a (fixed) central character, say $\zeta:Z\ra \F^{\times}$. 
Let $\Rep^{\rm l,fin}_{\F}(G)$ denote the subcategory of $\Rep_{\F}(G)$ consisting of locally finite objects.  Here an object $\Pi\in\Rep_{\F}(G)$  is said  to be \emph{locally finite} if for all $v\in \Pi$ the $\F[G]$-submodule generated by $v$ is of finite length.  

\medskip

Let $\Mod_{\F}^{\rm pro}(G)$ be the category of compact  left $\F[[K]]$-modules with an action of $\F[G]$ such that the two actions coincide when restricted to $\F[K]$ and that $Z$ acts via  $\zeta^{-1}$. It is anti-equivalent to $\Rep_{\F}(G)$ under Pontryagin dual $\Pi\mapsto \Pi^{\vee}:=\Hom_{\F}(\Pi,\F)$. Here,  $\Pi^{\vee}$  is naturally a \emph{right} $\F[G]$-module, but for convenience we view it as a \emph{left} $\F[G]$-module using the canonical anti-automorphism $\F[G]\simto \F[G]$  induced by $(g\mapsto g^{-1}):G\ra G$.  Let $\mathfrak{C}=\mathfrak{C}(G)$ be the full subcategory of $\Mod^{\rm pro}_{\F}(G)$ anti-equivalent to $\Rep^{\rm l,fin}_{\F}(G)$.

An object $M\in \mathfrak{C}$ is called \emph{coadmissible} if $M^{\vee}$ is admissible in the usual sense, i.e. $(M^{\vee})^H$ is finite dimensional for any open subgroup $H$ of $G$. This is equivalent to requiring $M$ to be finitely generated over $\F[[K]]$ (or equivalently, finitely generated over $\F[[H]]$ for any open compact subgroup $H\subset K$).
 
If $H$ is a closed subgroup of $K$,  denote by $\Rep_{\F}(H)$ the category of smooth $\F$-representations of $H$ on which $H\cap Z$ acts via the restriction of $\zeta$. Let $\mathfrak{C}(H)$ be the dual category of $\Rep_{\F}(H)$.

\medskip

For $n\geq 1$, let $K_n=\smatr{1+p^n\Z_p}{p^n\Z_p}{p^n\Z_p}{1+p^n\Z_p}$. Also let  $Z_1:=K_1\cap Z$.  Since $Z_1$ is pro-$p$ and $\zeta$ is smooth,  the restriction of $\zeta$ to ${Z_1}$ is trivial, so any  $\F$-representation of $G$ (resp. $K$) with central character $\zeta$ can be viewed as a representation of $G/Z_1$ (resp. $K/Z_1$).  Set \[\Lambda:=\F[[K_1/Z_1]] . \]
Since $K_1/Z_1$ is uniform (as $p>2$) and pro-$p$, the results in \S\ref{section-Iwasawa} apply  to $\Lambda$. Note that $\dim(K_1/Z_1)=3$. To simplify, we write $j(\cdot)=j_{\Lambda}(\cdot)$ and $\delta(\cdot)=\delta_{\Lambda}(\cdot)$. 
 
\medskip

 If $H$ is a closed subgroup of $G$ and $\sigma$ is a smooth representation of $H$, we denote by $\Ind^G_H\sigma$ the usual smooth induction. When $H$ is moreover open, we let $\cInd^G_H\sigma$ denote the compact induction, meaning the subspace of $\Ind^G_H\sigma$ consisting of functions whose support is compact modulo $H$.
\medskip

Let $\omega:\Q_p^{\times}\ra \F^{\times}$ be the mod $p$ cyclotomic character. If $H$ is any group, we write $\ide_H$ for the trivial representation of $H$ (over $\F$).

\subsection{Irreducible representations}\label{subsection-Breuil}
The work of Barthel-Livn\'e \cite{BL}  shows that absolutely irreducible objects in $\Rep_{\F}(G)$ fall into four classes:
\begin{enumerate}
\item one dimensional representations $\chi\circ\det$, where $\chi:\Q_p^{\times}\ra \F^{\times}$ is a smooth character;
\item (irreducible) principal series $\Ind_B^G\chi_1\otimes\chi_2$ with $\chi_1\neq \chi_2$;
\item special series, i.e.  twists of the Steinberg representation  $\Sp:=(\Ind_B^G\ide_T)/\ide_G$;
\item supersingular representations, i.e.  irreducible representations which are not isomorphic to subquotients of any parabolic induction. \end{enumerate}
For $0\leq r\leq p-1$, let $\Sym^r\F^2$ denote the standard symmetric power representation of $\GL_2(\F_p)$. Up to a twist by $\det^{m}$ with $0\leq m\leq p-2$, any absolutely irreducible $\F$-representation of $\GL_2(\F_p)$ is isomorphic to $\Sym^r\F^2$. Inflating to $K$ and letting $\smatr{p}00p$ act trivially, we may view  $\Sym^r\F^2$ as a representation of $KZ$. Let $I(\Sym^{r}\F^2):=\cInd_{KZ}^G\Sym^{r}\F^2$ denote the compact induction to $G$. It is well-known that $\End_G(I(\Sym^r\F^2))$ is isomorphic to $\F[T]$ for a certain Hecke operator $T$ (\cite{BL}). For $\lambda\in \F$ we  define
\[\pi(r,\lambda):=I(\Sym^r\F^2)/(T-\lambda).\] 
If  $\chi:\Q_p^{\times}\ra \F^{\times}$ is a smooth character, then let $\pi(r,\lambda,\chi):=\pi(r,\lambda)\otimes\chi\circ\det$.
In \cite{BL}, Barthel and Livn\'e showed that  any supersingular representation of $G$ is a \emph{quotient} of $\pi(r,0,\chi)$ for suitable $(r,\chi)$. Later on, Breuil 
 proved that $\pi(r,0,\chi)$ is itself irreducible (\cite{Br03}), hence completed the classification of irreducible objects in $\Rep_{\F}(G)$.  We will refer to $(r,\lambda,\chi)$ as above as a \emph{parameter triple}.

 Recall the link between non-supersingular representations and compact inductions: if $\lambda\neq0$ and $(r,\lambda)\neq (0,\pm1)$, then 
 \begin{equation}\label{equation-PS-cInd}\pi(r,\lambda)\cong \Ind_B^G\mu_{\lambda^{-1}}\otimes\mu_{\lambda}\omega^r,\end{equation}
 where $\mu_{x}:\Q_p^{\times}\ra \F^{\times}$ denotes the unramified character sending $p$ to $x$.  If $(r,\lambda)\in\{(0,\pm1),(p-1,\pm1)\}$, then we have non-split exact sequences:
 \begin{equation}\label{equation-pi(0)}0\ra  \Sp\otimes\mu_{\pm1}\circ\det\ra \pi(0,\pm1)\ra \mu_{\pm1}\circ\det\ra0,\end{equation}
 \begin{equation}\label{equation-pi(p-1)}0\ra \mu_{\pm1}\circ\det\ra \pi(p-1,\pm1)\ra \Sp\otimes\mu_{\pm1}\circ\det\ra0.\end{equation}

It is clear  for non-supersingular   representations and follows from \cite{Br03} for supersingular representations that any absolutely irreducible $\pi\in\Rep_{\F}(G)$ is admissible. Therefore $\pi^{\vee}$ is coadmissible and it makes sense to talk about $\delta(\pi^{\vee})$.  
\begin{theorem}\label{theorem-Paskunas-Morra}
Let $\Pi\in\Rep_{\F}(G)$. 
If $\Pi$ is of finite length, then $\Pi$ is admissible and $\delta(\Pi^{\vee})\leq 1$. 
\end{theorem}  
\begin{proof}
 The first assertion is clear. For the second, we may assume $\Pi$ is absolutely irreducible. Corollary \ref{corollary-CE} allows us to translate the problem to computing the  growth of $\dim_{\F}\Pi^{K_n}$.  If $\Pi$ is non-supersingular, then it is easy, see \cite[Prop. 5.3]{Mor} for a proof. If $\Pi$ is supersingular, this  is first done in \cite[Thm. 1.2]{Pa10} and later in \cite[Cor. 4.15]{Mor}. 
 \end{proof}
\medskip 

 Recall that a \emph{block} in $\Rep_{\F}(G)$ is an equivalence class of  absolutely irreducible objects in $\Rep_{\F}(G)$, where $\tau\sim \pi$ if and only if there exists a series of  irreducible representations $\tau=\tau_0,\tau_1,\dots,\tau_n=\pi$ such that $\Ext^1_G(\tau_i,\tau_{i+1})\neq0$ or $\Ext^1_G(\tau_{i+1},\tau_i)\neq0$ for each $i$.  
 \begin{proposition}\label{proposition-Gabriel}
The category $\Rep_{\F}^{\rm l,fin}(G)$ decomposes into a direct product of subcategories 
\[\Rep_{\F}^{\rm l,fin}(G)=\bigoplus_{\mathfrak{B}}\Rep_{\F}^{\rm l,fin}(G)^{\mathfrak{B}}\]
where the direct sum is taken over all the blocks $\mathfrak{B}$ and the objects of $\Rep_{\F}^{\rm l,fin}(G)^{\mathfrak{B}}$ are representations with all the irreducible subquotients lying in $\mathfrak{B}$. Correspondingly, we have a decomposition
of categories $\mathfrak{C}=\prod_{\mathfrak{B}}\mathfrak{C}^{\mathfrak{B}}$, where $\mathfrak{C}^{\mathfrak{B}}$ denotes the dual category of $\Rep_{\F}(G)^{\mathfrak{B}}$.
\end{proposition}
\begin{proof}
See \cite[Prop. 5.34]{Pa13}.
\end{proof}

The following theorem describes the blocks (when $p\geq 5$ as we are assuming). 
 
\begin{theorem}  \label{theorem-block}
Let $\pi\in \Rep_{\F}(G)$ be absolutely  irreducible and let $\mathfrak{B}$ be the block in which $\pi$ lies.  Then one of the following holds:
\begin{enumerate}
\item[(I)] if $\pi$ is supersingular, then $\mathfrak{B}=\{\pi\}$;
\item[(II)] if $\pi\cong \Ind_{B}^G\chi_1\otimes\chi_2\omega^{-1}$ with $\chi_1\chi_2^{-1}\neq \ide,\omega^{\pm1}$, then
\[\mathfrak{B}=\big\{\Ind_B^G\chi_1\otimes\chi_2\omega^{-1},\ \Ind_B^G\chi_2\otimes\chi_1\omega^{-1}\big\};\]
\item[(III)] if $\pi=\Ind_B^G\chi\otimes\chi\omega^{-1}$, then $\mathfrak{B}=\{\pi\}$;
\item[(IV)] otherwise, $\mathfrak{B}=\{\chi\circ\det,\Sp\otimes\chi\circ\det, (\Ind_B^G\alpha)\otimes\chi\circ\det\}$, where $\alpha=\omega\otimes\omega^{-1}$. 
\end{enumerate}
\end{theorem}
\begin{proof}
See \cite[Prop. 5.42]{Pa13}.
\end{proof}

\textit{Convention}: By \cite[Lem. 5.10]{Pa13}, any smooth irreducible $\bFp$-representation of $G$ with a central character is defined over a finite extension of $\F_p$. Theorem \ref{theorem-block} then implies that for a given block $\mathfrak{B}$, there is a common finite field $\F$ such that irreducible objects in ${\mathfrak{B}}$ are absolutely irreducible.  Hereafter, given  a finite set of blocks, we take $\F$ to be sufficiently large such that irreducible objects in these blocks are absolutely irreducible.

 \subsection{Projective envelopes}
 
 Fix $\pi\in\Rep_{\F}(G)$ irreducible and let $\mathfrak{B}$ be the block in which $\pi$ lies.  Let $\rInj_G\pi$ be an injective envelope of $\pi$ in $\Rep_{\F}^{\rm l,fin}(G)$; the existence is guaranteed by \cite[Cor. 2.3]{Pa13}.
  Let $P=P_{\pi^{\vee}}:=(\rInj_G\pi)^{\vee}\in \mathfrak{C}$ and $E=E_{\pi^{\vee}}:=\End_{\mathfrak{C}}(P)$. Then $P$ is  a projective envelope of $\pi^{\vee}$ in $\mathfrak{C}$ and is  naturally a left $E$-module.  Since $P$ is indecomposable, Proposition \ref{proposition-Gabriel} implies that (the dual of) every irreducible subquotient of $P$ lies in $\mathfrak{B}$. Also, 
  $E$ is a local $\F$-algebra (with residue field $\F$). Pa\v{s}k\=unas has computed $E$ and showed in particular that $E$ is commutative, except when $\mathfrak{B}$ is of type (III) listed in Theorem \ref{theorem-block}; in any case, we denote by $R=Z(E)$ the centre of $E$. 

\begin{theorem}(Pa\v{s}k\=unas)\label{theorem-Paskunas}
Keep the above notation.

(i) $R$ is naturally isomorphic to the Bernstein centre of $\mathfrak{C}^{\mathfrak{B}}$. In particular, $R$ acts on any object in $\mathfrak{C}^{\mathfrak{B}}$ and any morphism in $\mathfrak{C}^{\mathfrak{B}}$ is $R$-equivariant.

(ii)  
If $\mathfrak{B}$ is not of type (IV), $R$ is a regular local $\F$-algebra of  Krull dimension $3$. If $\mathfrak{B}$ is of type (IV),  $R$ is  isomorphic to $\F[[x,y,z,w]]/(xw-yz)$.  In particular, $R$ is Cohen-Macaulay of Krull dimension $3$.

(iii) $E=R$ except for blocks of type (III) in which case $E$ is a  free $R$-module of rank $4$.

(iv) If $\mathfrak{B}$ is not of type (IV), then $P$ is flat over both $E$ and $R$.
\end{theorem}

\begin{proof}
(i) This is \cite[Thm. 1.5]{Pa13}.

(ii) (iii) These are proved  in \cite{Pa13}. Precisely, see  Prop. 6.3 for type (I),  Cor. 8.7  for type (II),  \S9  for type (III) and  Cor. 10.78, Lem. 10.93 for type (IV). 

(iv) The flatness of $P$ over $E$ follows from  \cite[Cor. 3.12]{Pa13}, because the setting in \emph{loc. cit.} is satisfied for blocks of type (I)-(III). The flatness of $P$ over $R$ for blocks of type (III) follows from this and (iii).
\end{proof}

 \begin{proposition}\label{proposition-residue=finitelength}
 $\F\otimes_EP$ (resp. $\F\otimes_RP$) has finite length in $\mathfrak{C}$ and $\delta(\F\otimes_EP)=\delta(\F\otimes_RP)=1$.
\end{proposition}

\begin{proof}
We note that $\F\otimes_EP$ is characterized as the maximal quotient of $P$ which contains  $\pi^{\vee}$ with multiplicity one, see \cite[Rem. 1.13]{Pa13}. This object is denoted by $Q$ in \cite[\S3]{Pa13} and can be described explicitly. If $\mathfrak{B}$ is of type (I) or (III), $Q$ is just $\pi^{\vee}$. If $\mathfrak{B}$ is of type (II), it has  length $2$ by \cite[Prop. 6.1, (34)]{Pa15}.  If $\mathfrak{B}$ is of type (IV), the assertion follows from Proposition \ref{proposition-Q-pi} below in \S\ref{subsection-type-IV} where the explicit structure of $\F\otimes_EP$ is determined. The assertion $\delta(\F\otimes_EP)=1$ follows from the explicit description together with Theorem \ref{theorem-Paskunas-Morra}.

To see that $\F\otimes_RP$  has finite length, we may assume $\mathfrak{B}$ is of type (III), in which case $E$ is a free $R$-module of rank $4$. 
Since $E$ is a local ring with residue field $\F$, any irreducible (right) $E$-module is isomorphic to $\F$, so $\F\otimes_RE$ has a  filtration of finite length with graded pieces isomorphic to $\F$. The result then follows from the isomorphism $\F\otimes_RP\cong(\F\otimes_RE)\otimes_EP$; cf.  the proof of \cite[Cor. 4.2]{Pa18}.
\end{proof}

\begin{remark}\label{remark-unnecessary}
Note that Pa\v{s}k\=unas gave a short proof of Proposition \ref{proposition-residue=finitelength} without explicitly determining $\F\otimes_EP$, see \cite[Lem. 5.8]{Pa18}. We keep this proof because knowing the explicit form  might be of independent interest.
\end{remark}

\subsection{Serre weights}\label{subsection-Serreweight}
We keep the notation in the previous subsection. 
Let $\pi\in\Rep_{\F}(G)$ be irreducible.
By a \emph{Serre weight} of $\pi$ we mean an isomorphism class of (absolutely) irreducible $\F$-representations  of $K$, say $\sigma$,
 such that $\Hom_{K}(\sigma,\pi)\neq0$. Denote by $\mathscr{D}(\pi)$ the set of Serre weights of $\pi$. The description of $\mathscr{D}(\pi)$ can be deduced from \cite{BL} and \cite{Br03}; see \cite[Rem. 6.2]{Pa15} for a summary (of most cases). 

\begin{lemma}\label{lemma-D(pi)neqD(pi')}
If $\pi\neq \pi'$ are two objects in a block $\mathfrak{B}$, then $\mathscr{D}(\pi)\cap \mathscr{D}(\pi')=\emptyset$. 
\end{lemma}
\begin{proof}
The assertion is trivial if $\mathfrak{B}$ is of type (I) or (III). For type (II) or type (IV), it is  a direct check (using the assumption $p\geq 5$), see \cite[Rem. 6.2]{Pa15}.
\end{proof}

As before, write $P=P_{\pi^{\vee}}$ and let $\mathfrak{B}$ be the block in which $\pi$ lies. 
 
\begin{lemma}\label{lemma-Serreweight}
Let $\pi\in\Rep_{\F}(G)$ be irreducible.
 If $\pi\notin\{\ide_G,\Sp\}$ up to twist, then $\Hom_{K}^{\rm cont}(P,\sigma^{\vee})\neq0$ if and only if $\sigma\in\mathscr{D}(\pi)$. If $\pi\in\{\ide_G,\Sp\}$, then $\Hom_{K}^{\rm cont}(P,\sigma^{\vee})\neq0$ if and only if $\sigma\in\{\Sym^{0}\F^2,\Sym^{p-1}\F^2\}$.  
\end{lemma}
\begin{proof}
The assertion is trivial if $\mathfrak{B}$ is of type (I) or (III), because the block contains only one irreducible object. If $\mathfrak{B}$ is of type (II) or $\pi=\pi_{\alpha}$ up to twist, it is proved in \cite[Thm. 6.6]{Pa15}. 
Assume $\pi\in\{\ide_G,\Sp\}$ in the rest. Since $\mathscr{D}(\ide_G)=\{\Sym^0\F^2\}$, resp. $\mathscr{D}(\Sp)=\{\Sym^{p-1}\F^2\}$, we only need to prove (after taking dual)  
\[\Hom_{K}(\sigma', \rInj_{G}\pi)=0,\] where $\sigma'=\Sym^{p-3}\F^2\otimes\det$ is the only Serre weight for $\pi_{\alpha}$ (see \cite[Rem. 6.2]{Pa15}). By Frobenius reciprocity, it is equivalent to show 
$\Hom_G(I(\sigma'),\rInj_{G}\pi)=0$. If $\phi$ is such a morphism, then since $\rInj_G\pi$ is locally finite, $\phi$ must factor through a certain quotient $I(\sigma')/f(T)$ for some non-zero polynomial $f(T)\in \F[T]$ (\cite[Thm. 19]{BL}). If $\phi$ were non-zero, then its image must contain $\pi$ as its $G$-socle. But it follows from \cite[Thm. 33(2)]{BL} that $I(\sigma')/f(T)$ can not have $\ide_G$ or $\Sp$ as a subquotient. This contradiction allows to conclude. 
\end{proof}

\begin{proposition}\label{proposition-HT}
(i) Let $\sigma\in\Rep_{\F}(K)$ be irreducible.  Whenever $\Hom_K^{\rm cont}(P,\sigma^{\vee})^{\vee}$ is non-zero,  it is a cyclic $E$-module and if $J_{\sigma}$ denotes its annihilator then $E/J_{\sigma}\cong\F[[S]]$ (where $S$ is a formal variable). Moreover, if $\mathfrak{B}$ is not of type (I), then  $J_{\sigma}\subset E$ is  independent of $\sigma$.

(ii) Let $\widetilde{\sigma}=\oplus_{\sigma}\sigma$ where the sum is taken over all irreducible $\sigma$ such that $\Hom_K^{\rm cont}(P,\sigma^{\vee})\neq0$. Then $\Hom_{K}^{\rm cont}(P,\widetilde{\sigma}^{\vee})^{\vee}$ is a Cohen-Macaulay $R$-module of Krull dimension $1$.
\end{proposition}
\begin{proof}
(i) If $\mathfrak{B}$ is of type (I), it is proved in \cite[Thm. 6.6, (38)]{Pa15}.
Assume $\mathfrak{B}$ is not of type (I) and let $(r,\lambda,\chi)$ be a parameter triple of $\pi$.  If $(r,\lambda)\neq (p-1,\pm1)$, this is proved in \cite[Prop. 2.9]{HT} via \cite[Cor. 2.5]{HT}, where an (unfortunate) assumption (\textbf{H}) is imposed.   
If $(r,\lambda)=(p-1,\pm1)$, i.e. the assumption (\textbf{H}) is not satisfied, the statements are  still true and the proof can be adapted from the case of $(r,\lambda)=(0,\pm1)$, see \cite[Rem. 2.6]{HT}. 

(ii)  If $\mathfrak{B}$ is of type (I), it is a special case of \cite[Lem. 2.33]{Pa15} via \cite[Thm. 5.2]{Pa15}. 
If $\mathfrak{B}$ is of type (II) or (IV), then $E=R$ and the result is a weaker form of (i). If $\mathfrak{B}$ is of type (III),  we need to show that the image  of $R\hookrightarrow E\twoheadrightarrow \F[[S]]$ contains a regular element,  for which it is enough to show that the image contains a non-zero element in the maximal ideal of $\F[[S]]$. But this is clear because $\F[[S]]$ is  finite over (the image of) $R$. Alternatively, one can apply Proposition \ref{proposition-III-image} below where the image of $R$ in $\F[[S]]$ is determined.
\end{proof}

The following general result is extracted from \cite[Thm. 3.5]{HP}.
\begin{proposition}\label{proposition-projective}
Let $\widetilde{P}\in \mathfrak{C}(K)$ and $f\in \End_{\mathfrak{C}(K)}(\widetilde{P})$. Assume 
\begin{enumerate}
\item[(i)] $\widetilde{P}$ is projective in $\mathfrak{C}(K)$;
\item[(ii)] for any irreducible $\sigma\in\Rep_{\F}(K)$, the induced morphism 
\[f_*:\Hom_{K}^{\rm cont}(\widetilde{P},\sigma^{\vee})^{\vee}\ra \Hom_{K}^{\rm cont}(\widetilde{P},\sigma^{\vee})^{\vee}\] is injective. 
\end{enumerate}
Then $f$ is injective and $\widetilde{P}/f\widetilde{P}$ is projective in $\mathfrak{C}(K)$.  
\end{proposition}

\begin{proof}
Consider the complex $P_{\bullet}$ of projective modules $0\ra P_1\overset{f}{\ra} P_0\ra0$, where $P_0=P_1=\widetilde{P}$. Applying $\Hom_{K}^{\rm cont}(-,\sigma^{\vee})$ to it, where $\sigma\in \Rep_{\F}(K)$ is irreducible, we obtain a convergent spectral sequence 
\[E_2^{ij}:=\Ext^{i}_{K/Z_1}(H_j(P_{\bullet}),\sigma^{\vee})\Rightarrow H^{i+j}(\Hom_K^{\rm cont}(P_{\bullet},\sigma^{\vee}))\]
which gives the following  long exact sequence:
\[ \Ext^1_{K/Z_1}(\mathrm{coker}(f),\sigma^{\vee})\hookrightarrow H^1(\Hom_K^{\rm cont}(P_{\bullet},\sigma^{\vee}))\ra \Hom_K^{\rm cont}(\ker(f),\sigma^{\vee})\ra \Ext^2_{K/Z_1}(\mathrm{coker}(f),\sigma^{\vee}).\]
The assumption (after taking dual) says that the morphism
\[\Hom_K^{\rm cont}(P_0,\sigma^{\vee})\ra\Hom_{K}^{\rm cont}(P_1,\sigma^{\vee})\] 
is surjective, i.e. $H^1(\Hom_{K}^{\rm cont}(P_{\bullet},\sigma^{\vee}))=0$,  which forces $\Ext^1_{K/Z_1}(\mathrm{coker}(f),\sigma^{\vee}))=0$. This being true for any $\sigma$, we deduce that  $\mathrm{coker}(f)$ is projective in $\mathfrak{C}(K)$, proving the second assertion. As a consequence, $\Ext^2_{K/Z_1}(\mathrm{coker}(f),\sigma^{\vee})$ vanishes, and so does $\Hom_K^{\rm cont}(\ker(f),\sigma^{\vee})$. This being true for any $\sigma$, we deduce that $\ker(f)=0$, i.e. $f$ is injective.
\end{proof}

\subsection{Principal series and deformations}\label{subsection-PS}
 
Recall that $T$ denotes the diagonal torus of $G$. If $\eta:T\ra \F^{\times}$ is a smooth character, let $\pi_{\eta}:=\Ind_B^G\eta$ which is possibly reducible. Let $\rInj_T\eta$ be an injective envelope of ${\eta}$ in $\Rep_{\F}(T)$ (with central character) and set $\Pi_{\eta}:=\Ind_B^G\rInj_T\eta$. Then $\Pi_{\eta}$ is a locally finite smooth representation of $G$. It is easy to see that $\rsoc_{G}\Pi_{\eta}=\rsoc_G\pi_{\eta}$, which we denote by $\pi$.  So there is a $G$-equivariant embedding $\Pi_{\eta}\hookrightarrow \rInj_G\pi$ and by \cite[Prop. 7.1]{Pa13} the image does not depend on the choice of the embedding.

\begin{proposition}\label{proposition-not-admissible}
$\Pi_{\eta}$ is not admissible.
\end{proposition}
\begin{proof}
Let $(r,\lambda,\chi)$ be a parameter triple such that $\pi_{\eta}\cong \pi(r,\lambda,\chi)$. This is always possible by \cite[Thm. 30]{BL} and we have $(r,\lambda)\neq (0,\pm1)$. Moreover, we may assume $\chi=1$ up to twist. It suffices to prove that 
\begin{equation}\label{equation-dim=infty}\dim_{\F}\Hom_{K}(\Sym^r\F^2,\Pi_{\eta})=+\infty,\end{equation}
which follows from \cite{BL,BL2} as we explain below. 

Recall that $\F[T]$ denotes the Hecke algebra associated  to $I(\Sym^{r}\F^2)$. In \cite[\S6]{BL} is constructed an $\F[T,T^{-1}]$-linear morphism 
\[P:I(\Sym^r\F^2)\otimes_{\F[T]}\F[T,T^{-1}]\ra \Ind_B^GX_1\otimes X_2\]
where $X_i:\Q_p^{\times}\ra (\F[T,T^{-1}])^{\times}$ are  tamely ramified characters given by
 \[X_1\ \mathrm{unramified},\ \ \ X_1(p)=T^{-1},\ \ \ X_1X_2=\omega^r.\]  
Note that specializing to $T=\lambda$, we have $X_1\equiv\mu_{\lambda^{-1}}$ so that $\Ind_B^GX_1\otimes X_2\ (\mathrm{mod}\ T-\lambda)\cong\pi_{\eta}$ by \eqref{equation-PS-cInd}. 
  
 By \cite[Thm. 25]{BL}, $P$ is an isomorphism except for $r=0$, in which case $P$ is injective  and we have an exact sequence (\cite[Thm. 20]{BL2})
 \[0\ra I(\Sym^0\F^2)\otimes_{\F[T]}\F[T,T^{-1}]\overset{P}{\ra} \Ind_B^GX_1\otimes X_2\ra \Sp\otimes \F[T,T^{-1}]/(T^{-2}-1)\ra0. \]
But since $\lambda\neq \pm1$, $T-\lambda$ acts invertibly on the last term, hence  $P$ modulo  $(T-\lambda)^n$ induces an isomorphism  for any $n\geq1$:
\[P_n: I(\Sym^{0}\F^2)/(T-\lambda)^n \cong \Ind_{B}^GX_1\otimes X_2 \mod (T-\lambda)^n .\]
As remarked above,  the RHS is the parabolic induction of a deformation of $\eta$ to $\F[T]/(T-\lambda)^n$, hence embeds in $\Pi_{\eta}$. This implies
 \eqref{equation-dim=infty}.
\end{proof}
\begin{remark} \label{remark-not-admissible}
Keep the notation in the proof of Proposition \ref{proposition-not-admissible}. If $r=p-1$, then $\omega^{p-1}=\omega^0$, and the above proof  shows that $\Hom_{K}(\Sym^0\F^2,\Pi_{\eta})$ is also infinite dimensional.
\end{remark}


Let $M_{\eta^{\vee}}:=(\Pi_{\eta})^{\vee}$ and $E_{\eta^{\vee}}=\End_{\mathfrak{C}}(M_{\eta^{\vee}})$ .

\begin{lemma} \label{lemma-Meta-flat}
 $E_{\eta^{\vee}}$ is isomorphic to $\F[[x,y]]$
 and $M_{\eta^{\vee}}$ is  flat over $E_{\eta^{\vee}}$.
\end{lemma}
\begin{proof}
 By \cite[Prop. 7.1]{Pa13}, we have a natural isomorphism $E_{\eta^{\vee}}\cong \End_{\mathfrak{C}(T)}((\rInj_T\eta)^{\vee})$ and the latter ring is isomorphic to $\F[[x,y]]$ by \cite[Cor. 7.2]{Pa13}.
 
 By \cite[\S3.2]{Pa13}, $(\rInj_T\eta)^{\vee}$ is isomorphic to the universal deformation of the $T$-representation $\eta^{\vee}$ (with fixed central character), with $E_{\eta^{\vee}}$ being  the universal deformation ring. In particular, it is flat over $E_{\eta^{\vee}}$. The second assertion follows from this and the definition of $M_{\eta^{\vee}}$.  
\end{proof}

Recall that $\pi$ denotes the $G$-socle of $\pi_{\eta}$. Let $P=P_{\pi^{\vee}}$. 

\begin{proposition}\label{prop-Ind-specialcase} 
Let $M\in\mathfrak{C}$ be a coadmissible quotient of $M_{\eta^{\vee}}$.  Then $\delta(M)\leq 2$.
\end{proposition}
\begin{proof}
Since $M$ is  coadmissible while $M_{\eta^{\vee}}$ is not by Proposition \ref{proposition-not-admissible}, the kernel of $M_{\eta^{\vee}}\twoheadrightarrow M$ is non-zero and not coadmissible; denote it by $M'$.
We claim that $\Hom_{\mathfrak{C}}(M_{\eta^{\vee}},M')\neq 0$. For this it suffices to prove $\Hom_{\mathfrak{C}}(P,M')\neq0$, because any morphism $P\ra M'$ must factor through $P\twoheadrightarrow M_{\eta^{\vee}}\ra M'$, see  \cite[Prop. 7.1(iii)]{Pa13}. Assume $\Hom_{\mathfrak{C}}(P,M')=0$ for a contradiction. Then $\pi^{\vee}$  does not  occur as a subquotient in $M'$. This is impossible unless $\pi_{\eta}$ is reducible, i.e. $\pi_{\eta}\cong\pi(p-1,1)$ up to twist.  
Assuming it is the case,  we have $\pi=\ide_G$ and all irreducible subquotients of $M'$ are isomorphic to $\Sp^{\vee}$, see \eqref{equation-pi(p-1)}. In particular, we obtain $\Hom_{K}^{\rm cont}(M',(\Sym^0\F^2)^{\vee})=0$ (as the $K$-socle of $\Sp$ is isomorphic to $\Sym^{p-1}\F^2$). However, this would imply an isomorphism 
\[\Hom_K^{\rm cont}(M_{\eta^{\vee}},(\Sym^0\F^2)^{\vee})\cong \Hom_K^{\rm cont}(M,(\Sym^{0}\F^2)^{\vee}).\]
Together with Remark \ref{remark-not-admissible} this contradicts the coadmissibility of $M$.
 
 Via the embedding $\Hom_{\mathfrak{C}}(M_{\eta^{\vee}},M')\hookrightarrow \Hom_{\mathfrak{C}}(M_{\eta^{\vee}},M_{\eta^{\vee}})=E_{\eta^{\vee}}$, the claim implies the existence of a non-zero element
 $f\in E_{\eta^{\vee}}$ which annihilates $M$. Since $E_{\eta^{\vee}}\cong \F[[x,y]]$ is a regular ring of dimension $2$, we may find $g\in E_{\eta^{\vee}}$ such that $f,g$ is a system of parameters of $E_{\eta^{\vee}}$.  Then $E_{\eta^{\vee}}/(f,g)$ is finite dimensional over $\F$, and consequently  $M_{\eta^{\vee}}/(f,g)$ has finite length in $\mathfrak{C}$. Hence, $M/(f,g)M=M/gM$  also has finite length and  Theorem \ref{theorem-Paskunas-Morra} implies that $\delta(M/gM)\leq 1$. We then conclude by Proposition \ref{proposition-GKdim}. 
\end{proof}

The inclusion $\Pi_{\eta}\hookrightarrow \rInj_G\pi$ induces a surjection $P\twoheadrightarrow M_{\eta^{\vee}}$. By \cite[Prop. 7.1]{Pa13} this  induces a surjective ring morphism $E\twoheadrightarrow E_{\eta^{\vee}}$, via which $P\twoheadrightarrow M_{\eta^{\vee}}$ is a morphism of  (left) $E$-modules.
\begin{corollary}\label{corollary-P=Meta}
For any irreducible $\sigma\in \Rep_{\F}(K)$, the natural surjective morphism
\begin{equation}\label{equation-P=Meta}\Hom_{K}^{\rm cont}(P,\sigma^{\vee})^{\vee}\ra \Hom_K^{\rm cont}(M_{\eta^{\vee}},\sigma^{\vee})^{\vee}\end{equation}
is an isomorphism of $E$-modules.
\end{corollary}
\begin{proof}
The quotient  $P\twoheadrightarrow M_{\eta^{\vee}}$ is a morphism of $E$-modules, hence so is  \eqref{equation-P=Meta}. 

We may assume $\Hom_K^{\rm cont}(P,\sigma^{\vee})$ is non-zero, so that it is isomorphic to $\F[[S]]$ as an $E$-module by Proposition \ref{proposition-HT}(i). Hence, to prove the injectivity of \eqref{equation-P=Meta}, it suffices to prove that   $\Hom_K^{\rm cont}(M_{\eta^{\vee}},\sigma^{\vee})$ is infinite dimensional. This is already established in the proof of Proposition \ref{proposition-not-admissible}, together with Remark \ref{remark-not-admissible} and Lemma \ref{lemma-Serreweight} if $\sigma\in\{\Sym^0\F^2,\Sym^{p-1}\F^2\}$ up to twist.
\end{proof}

\subsection{Coadmissible quotients}
\label{subsection-coadmissible}
Keep the notation in the previous subsection.  Let $M\in\mathfrak{C}$ be a coadmissible quotient of $P=P_{\pi^{\vee}}$.  We set ${\rm m}(M):=\Hom_{\mathfrak{C}}(P,M)$ which is  a finitely generated right $E$-module. There is a natural morphism 
\begin{equation}\label{equation-evaluation}
{\rm ev}: {\rm m}(M)\otimes_EP\ra M
\end{equation} 
which is surjective  by \cite[Lem. 2.10]{Pa13}.   Remark that we should have written $\mathrm{m}(M)\widehat{\otimes}_EP$ in \eqref{equation-evaluation}, where $\widehat{\otimes}$ means taking completed tensor product. But since $\mathrm{m}(M)$ is finitely generated over $E$, the completed and usual tensor product coincide, see the discussion before \cite[Lem. 2.1]{Pa15}.

Let $\Ker$ be the kernel of \eqref{equation-evaluation}. By \cite[Lem. 2.9]{Pa13} we have  $$\Hom_{\mathfrak{C}}(P,{\rm m}(M)\otimes_EP)\cong {\rm m}(M),$$ so $\Hom_{\mathfrak{C}}(P,\Ker)=0$ because $P$ is projective in $\mathfrak{C}$. This implies that $\Ker$ does not admit $\pi^{\vee}$ as a subquotient, i.e. if $\pi'$ is an irreducible subquotient of $\Ker$, then $\pi'\in\mathfrak{B}$ and $\pi'\ncong \pi$. In particular, if $\mathfrak{B}$ is of type (I) or (III) of Theorem \ref{theorem-block}, then $\Ker=0$ and \eqref{equation-evaluation} is an isomorphism. In any case, we have the following fact. 

\begin{corollary}\label{corollary-serreweight-ok}
If $\sigma\in\mathscr{D}(\pi)$, then \eqref{equation-evaluation} induces an isomorphism
 \[\Hom_{K}^{\rm cont}({\rm m}(M)\otimes_EP,\sigma^{\vee})^{\vee}\cong \Hom_K^{\rm cont}(M,\sigma^{\vee})^{\vee}.\]
\end{corollary} 
\begin{proof}
Using Lemma \ref{lemma-D(pi)neqD(pi')}, the above argument shows that  $\Hom_{K}^{\rm cont}(\Ker,\sigma^{\vee})^{\vee}=0$ for any $\sigma\in\mathscr{D}(\pi)$, giving  the result.
\end{proof}

\begin{proposition}\label{proposition-M-to-m(M)}
Let $M\in\mathfrak{C}$ be a coadmissible quotient of $P=P_{\pi^{\vee}}$.   The following statements hold.

(i)   ${\rm m}(M)\otimes_EP$ is coadmissible. 

(ii) If $M$ is torsion as a  $\Lambda$-module, then so is ${\rm m}(M)\otimes_EP$.
\end{proposition}
\begin{proof}
Since $\Ker=0$ if $\mathfrak{B}$ is of type (I) or (III),  both the assertions are trivial in these cases. Thus, we assume that $\mathfrak{B}$ is of type (II) or (IV) in the rest. 

(i)  
The coadmissibility of $M$ is equivalent to that $\Hom_K^{\rm cont}(M,\sigma^{\vee})$ is finite dimensional over $\F$ for any irreducible $\sigma\in \Rep_{\F}(K)$, and we need to check this property for $\mathrm{m}(M)\otimes_EP$. 
By \cite[Prop. 2.4]{Pa15}\footnote{This result is stated for a commutative subring of $E$ rather than for $E$ itself which can be non-commutative, but the same proof goes through without change.} we have a natural isomorphism of finitely generated $E$-modules:
\begin{equation}\label{equation-Pas-Prop2.4}\Hom_{K}^{\rm cont}(\mathrm{m}(M)\otimes_EP,\sigma^{\vee})^{\vee}\cong \mathrm{m}(M)\otimes_E\Hom_{K}^{\rm cont}(P,\sigma^{\vee})^{\vee},\end{equation}
hence we only need to consider those (irreducible) $\sigma$ such that $\Hom_K^{\rm cont}(P,\sigma^{\vee})^{\vee}\neq0$.  Choose any weight $\sigma'\in \mathscr{D}(\pi)$, which clearly implies $\Hom_K^{\rm cont}(P,\sigma'^{\vee})\neq0$. Then using \eqref{equation-Pas-Prop2.4}, Proposition \ref{proposition-HT}(i) and Corollary \ref{corollary-serreweight-ok}, we obtain isomorphisms
\[\Hom_K^{\rm cont}(\mathrm{m}(M)\otimes_EP,\sigma^{\vee})^{\vee}  \cong \Hom_K^{\rm cont}(\mathrm{m}(M)\otimes_EP,\sigma'^{\vee})^{\vee}\cong \Hom_{K}^{\rm cont}(M,\sigma'^{\vee})^{\vee}. \]
Since $M$ is coadmissible by assumption,  all these  spaces are finite dimensional over $\F$.

(ii)  It is equivalent to show that $\Ker$ is a torsion  $\Lambda$-module (it is coadmissible  as a consequence of (i)).  Since the case of type (II) is similar and simpler, we assume in the rest that $\mathfrak{B}$ is of type (IV), so that $\mathfrak{B}$ consists of three irreducible objects and we let $\pi_1,\pi_2$ be the two other  than $\pi$. Since $\Ker$ is coadmissible by (i) and does not admit $\pi^{\vee}$ as a subquotient, we can find  $s_1,s_2\geq 0$ and a surjection \[P_{\pi_1^{\vee}}^{\oplus s_1}\oplus P_{\pi_2^{\vee}}^{\oplus s_2}\twoheadrightarrow \Ker.\]
Let $Q_1$ (resp. $Q_2$) be  the maximal quotient of $P_{\pi_1^{\vee}}$ (resp. $P_{\pi_2^{\vee}}$) none of whose irreducible subquotients is isomorphic to  $\pi^{\vee}$. Then the above surjection must factor through $Q_1^{\oplus s_1}\oplus Q_2^{\oplus s_2}\twoheadrightarrow \Ker$.  Hence, it is enough to show that any  coadmissible quotient of $Q_1$ (resp. $Q_2$)   is torsion. This follows from the results in \cite[\S10]{Pa13} as we explain below. Up to twist we may assume $\mathfrak{B}=\{\ide_G,\Sp,\pi_{\alpha}\}$.
 
Let us first assume $\pi=\pi_{\alpha}$, so that up to order $\pi_1=\ide_G$ and $\pi_2=\Sp$. We have the following exact sequences  
\[0\ra P_{\pi_{\alpha}^{\vee}}\ra P_{\ide_G^{\vee}}\ra M_{\ide_T^{\vee}}\ra0,\]
\[P_{\pi_{\alpha}^{\vee}}^{\oplus 2}\ra P_{\Sp^{\vee}}\ra M_{\ide_T^{\vee},0}\ra0,\]
see \cite[(234),(236)]{Pa13}, where $M_{\ide_T^{\vee},0}$ is a submodule of $M_{\ide_T^{\vee}}$ defined by \cite[(233)]{Pa13}, namely defined by the  exact sequence
\begin{equation}\label{equation-Pas-233}
0\ra M_{\ide_T^{\vee},0}\ra M_{\ide_T^{\vee}}\ra \F\ra0.\end{equation} It is easy to see that $M_{\ide_T^{\vee}}$ (resp. $M_{\ide_T^{\vee},0}$) does not admit $\pi^{\vee}$ as a subquotient, hence  $Q_1$ (resp. $Q_2$) is equal to $M_{\ide_T^{\vee}}$ (resp. $M_{\ide_T^{\vee},0}$).  Proposition \ref{prop-Ind-specialcase} then implies that any coadmissible quotient of $Q_i$ is a torsion $\Lambda$-module; remark that  $M_{\ide_T^{\vee},0}$ is \emph{not} a quotient of $M_{\ide_T^{\vee}}$ so that Proposition \ref{prop-Ind-specialcase} does not apply directly to coadmissible quotients of $M_{\ide_T^{\vee},0}$, but we can conclude using \eqref{equation-Pas-233}.
Finally, a similar argument works in the case $\pi\in\{\ide_G,\Sp\}$.
\end{proof}

\begin{remark}
Given a block $\mathfrak{B}$, let $P_{\mathfrak{B}}=\oplus_{\pi\in\mathfrak{B}}P_{\pi^{\vee}}$ and $E_{\mathfrak{B}}=\End_{\mathfrak{C}}(P_{\mathfrak{B}})$. It follows from \cite[Lem. 2.9, Lem. 2.10]{Pa13} that for any $M\in \mathfrak{C}^{\mathfrak{B}}$, the evaluation morphism \[\Hom_{\mathfrak{C}}(P_{\mathfrak{B}},M)\otimes_{E_{\mathfrak{B}}}P_{\mathfrak{B}}\ra M\] is always an isomorphism. This may suggest that, even for a quotient of $P_{\pi^{\vee}}$ for some fixed $\pi\in\mathfrak{B}$, we should consider the $E_{\mathfrak{B}}$-module $\Hom_{\mathfrak{C}}(P_{\mathfrak{B}},M)$ rather than the $E_{\pi^{\vee}}$-module $\Hom_{\mathfrak{C}}(P_{\pi^{\vee}},M)$ (they are both finitely generated modules over $Z(E_{\mathfrak{B}})=Z(E_{\pi^{\vee}})$).  However, for our main application, we need to translate the $\Lambda$-torsionness of $M$ into a commutative algebra statement, and Proposition \ref{proposition-dim-equality} below shows that $\Hom_{\mathfrak{C}}(P_{\pi^{\vee}},M)$ fits in with the need.  
The analogue of Proposition \ref{proposition-dim-equality} for $\Hom_{\mathfrak{C}}(P_{\mathfrak{B}},M)$ need not hold:  take $\mathfrak{B}=\{\ide_G,\Sp,\pi_{\alpha}\}$ and $M=\ide_G^{\vee}$ so that $\delta(M)=0$. (But see \cite[Prop. 3.6.12, Rem. 3.6.13]{Pan} for one inequality.) 
\end{remark}

\subsection{Blocks of type (III)} \label{subsection-III}

In this subsection, we assume $\mathfrak{B}$ is of type (III), that is, $\mathfrak{B}=\{\pi\}$ with $\pi\cong \Ind_B^G\eta$, where $\eta=\chi\otimes\chi\omega^{-1}$. After twisting, we assume $\chi=\ide$ is trivial. Let $P=P_{\pi^{\vee}}$ and $E=\End_{\mathfrak{C}}(P)$. Then $E$ is non-commutative, and is  free over its centre $R$ of rank $4$, see Theorem \ref{theorem-Paskunas}.   
 Let $M_{\eta^{\vee}}$ and $E_{\eta^{\vee}}$ be as in \S\ref{subsection-PS}; recall that $E_{\eta^{\vee}}$ is isomorphic to $\F[[x,y]]$ and is identified with $E^{ab}$ (the maximal abelian quotient of $E$)  by the discussion before \cite[Lem. 9.2]{Pa13}. On the other hand, it is known that $\mathscr{D}(\pi)$ consists of one single weight, i.e. $\sigma=\Sym^{p-2}\F^2$, see \cite[Rem. 6.2]{Pa15}. Hence we have surjective morphisms via \eqref{equation-P=Meta} and Proposition \ref{proposition-HT}:
\[E  \twoheadrightarrow E^{ab}\twoheadrightarrow \F[[S]].\]
The goal of this subsection is to  prove the following result.  
\begin{proposition}\label{proposition-III-image}
(i) We may choose the variables $x,y$ in such a way that the image of the composite morphism $R\hookrightarrow  E\twoheadrightarrow E^{ab}\cong \F[[x,y]]$ is equal to $\F[[x^2,xy,y^2]]$.

(ii) We may choose the variable $S$ in such a way that the image of  the composite morphism $R\hookrightarrow E\twoheadrightarrow  \F[[S]]$ is equal to $\F[[S^2]]$.  

(iii) Let $I=\Ker(R\ra \F[[S]])$ and $J=\Ker(E\twoheadrightarrow \F[[S]])$. Then $J^4\subset IE\subset J$.
\end{proposition}

The proof of Proposition \ref{proposition-III-image} relies on another explicitly constructed ring defined in  \cite[(145)]{Pa13}, which we denote by $E'$ (instead of $R$ in \emph{loc. cit.}).  
We briefly recall the construction.
 Let $\mathcal{G}$ be the maximal pro-$p$ quotient of $G_{\Q_p}:=\Gal(\bQp/\Q_p)$ and $\mathcal{G}^{ab}$ the maximal abelian quotient of $\mathcal{G}$. By local class field theory, we have 
 \[G_{\Q_p}^{ab}\cong \Gal(\Q_p(\mu_{p^{\infty}})/\Q_p)\times\Gal(\Q_p^{\rm ur}/\Q_p)\cong \Z_p^{\times}\times\widehat{\Z},\]
where $\mu_{p^{\infty}}$ is the group of $p$-power order roots of unity in $\bQp$ and $\Q_p^{\rm ur}$ is the maximal unramified extension of $\Q_p$. Since $\mathcal{G}^{ab}$ is equal to the maximal pro-$p$ quotient of $ G_{\Q_p}^{ab}$, we obtain 
$$\mathcal{G}^{ab}\cong (1+p\Z_p)\times\Z_p.$$ 
We choose a pair of generators $\overline{\gamma},\overline{\delta}$ of $\mathcal{G}^{ab}$ such that $\bar{\gamma}\mapsto (1+p,0)$ and $\bar{\delta}\mapsto (1,1)$. Then $\mathcal{G}$ is a free pro-$p$ group generated by $2$ elements $\gamma,\delta$ which lift respectively $\overline{\gamma},\overline{\delta}$. See \cite[\S2]{San13} for details. Following \cite[(145)]{Pa13}, we let (note that  in \emph{loc. cit.} the ring is defined over $\cO$ and is denoted by $R$),  $E'$ is defined as:
\[E':=\frac{\F[[t_1,t_2,t_3]]\widehat{\otimes}_{\F}\F[[\mathcal{G}]]}{J}\]
for a certain closed two-sided ideal $J$ generated by the relations listed in \cite[(146),(147)]{Pa13}. 

 With the notation in \emph{loc. cit.}, we have the following facts:
\begin{enumerate}
\item[(a)] The natural morphism $\F[[t_1,t_2,t_3]]\ra E'$ is injective and identifies $\F[[t_1,t_2,t_3]]$ with the centre of $E'$, denoted by $R'$. $E'$ is a free $R'$-module of rank $4$, and  $E'$ contains two elements 
\[u:=\gamma-1-t_1,\ \ v:=\delta-1-t_2\]
 such that  $\{1,u,v,t'\}$ is an $R'$-basis,  where $t':=uv-vu$.  See \cite[Cor. 9.24, Cor. 9.25]{Pa13}.
\item[(b)] $E'^{ab}$  is isomorphic to  $\F[[\overline{u},\overline{v}]]$, where  $\overline{u}$ (resp. $\overline{v}$) denotes the image of $u$ (resp. $v$). The kernel of $E'\ra E'^{ab}$ is equal to $t'E'$. See Lemma \cite[Lem. 9.3]{Pa13} and the proof of \cite[Cor. 9.27]{Pa13}. 
\item[(c)]   $E'$ is equipped with an involution  $*$ which satisfies $u^*=-u$, $v^*=-v$, $t'^*=-t'$ and
\[R'=\{\phi\in E': \ \phi=\phi^{*}\}.\] 
See   \cite[(161), Lem. 9.14]{Pa13}. 
\item[(d)] There exists a ring isomorphism $\varphi: E\cong E'^{op}$ by \cite[Cor. 9.27]{Pa13}. Moreover, it is compatible with Colmez's functor $\VV$ (modified as  in \cite[\S5.7]{Pa13}) in the following sense. As explained in \cite[\S9.1]{Pa13},  $\VV$ induces a natural transformation $\mathrm{Def}_{\pi^{\vee}}\ra \mathrm{Def}_{\VV(\pi^{\vee})}$ between certain deformation functors of $\pi^{\vee}$ and of $\VV(\pi^{\vee})$,  which are respectively pro-represented by $E$ and $\F[[\mathcal{G}]]^{op}$, hence induces a ring morphism by Yoneda's Lemma
\[\varphi_{\VV}: \F[[\mathcal{G}]]^{op}\ra E\]
which is uniquely determined up to conjugation by $E^{\times}$. Here,  we consider  deformation problems with coefficients in finite local (possibly non-commutative) Artinian $\F$-algebras with residue field $\F$.
Then the following diagram is commutative
\begin{equation}\label{equation-III-varphi-VV}\xymatrix{\F[[\mathcal{G}]]^{op}\ar[r]\ar_{\varphi_{\VV}}[dr]&E'^{op}\ar_{\cong}^{\varphi^{-1}}[d]\\ 
&E}\end{equation}
where the upper horizontal morphism is the natural one.
\end{enumerate} 
In summary, we have a commutative diagram
\begin{equation}\label{equation-III-diagram}\xymatrix{R'\ar[r]\ar_{\cong}[d]&E'^{op}\ar@{->>}[r]\ar_{\cong}^{\varphi^{-1}}[d]&E'^{ab}\ar_{\cong}[d]&\\
R\ar[r]&E\ar@{->>}[r]&E^{ab}\ar@{->>}[r]&\F[[S]].}\end{equation}
Thus, to prove Proposition \ref{proposition-III-image}, we may work with $R'$, $E'^{op}$, $E'^{ab}$ instead of $R$, $E$, $E^{ab}$, via the isomorphism $\varphi^{-1}$.  
\begin{lemma}\label{lemma-III-mod-v}
(i) The element $\gamma\in\mathcal{G}$ is sent to $1$ under the composite map
\begin{equation}\label{equation-III-compositemap}
\F[[\mathcal{G}]]^{op}\overset{\varphi_{\VV}}{\lra} E\twoheadrightarrow \F[[S]].\end{equation}
(ii) The element $u\in E'$ is sent to $0$ under the composite map  $E'^{op}\simto E\twoheadrightarrow \F[[S]]$.
\end{lemma}
\begin{proof}
(i) Recall that $P$ is flat over $E$ and can be viewed as a deformation of $\pi^{\vee}$ over $E$, in the sense of \cite[\S3.1]{Pa13}. Consider $\F[[S]]\otimes_EP$ and view it as a deformation of $\pi^{\vee}$ to $\F[[S]]$. 
It is proved in \cite[1.5.9]{ki09} and reformulated in \cite[Prop. 2.9, Prop. 2.11]{HT}  that 
\[\VV(\F[[S]]\otimes_EP)\cong \mu_{S+1}^{-1}\]
where $\mu_{S+1}:G_{\Q_p}\ra \F[[S]]^{\times}$ is the unramified character sending geometric Frobenii to $S+1$. Here we have used the isomorphism $\pi\cong\pi(p-2,1)$, see \eqref{equation-PS-cInd}.  It is clear that $\mu_{S+1}$ factors through $G_{\Q_p}\twoheadrightarrow \mathcal{G}^{ab}$ and   $\mu_{S+1}(\overline{\gamma})=1$  by our choice of $\gamma$.   The result follows from this because, when viewed as a deformation of $\VV(\pi^{\vee})$ to $\F[[S]]$, $\VV(\F[[S]]\otimes_EP)$ is obtained from the universal deformation $\F[[\mathcal{G}]]$ via  the map \eqref{equation-III-compositemap}. Here, $\F[[\mathcal{G}]]^{op}$ is viewed as the universal deformation ring via $\F[[\mathcal{G}]]^{op}\cong\End_{\F[[\mathcal{G}]]}(\F[[\mathcal{G}]])$), see \cite[\S3.2]{Pa13}.

(ii)  
Since  $u=\frac{\gamma-\gamma^{-1}}{2}$ (as is shown after \cite[(160)]{Pa13}), the result follows from (i).  
\end{proof}

Denote by $\overline{R}'$ the image of $R'\hookrightarrow  E'\twoheadrightarrow E'^{ab}\cong \F[[\overline{u},\overline{v}]]$, and let $\fm_{\overline{R}'}$ be its maximal ideal.
\begin{lemma}\label{lemma-III-1}
$\overline{u}^2, \overline{v}^2, \overline{u}\overline{v}\in\overline{R}'$.
\end{lemma}
\begin{proof}
It is proved in \cite[(159)]{Pa13}\footnote{There is a typo in the formula, namely we should get  $u^2=2t_1+t_1^2$ and $v^2=2t_2+t_2^2$ from \cite[(148)]{Pa13}. This does not affect the rest of \cite[\S9]{Pa13}, because we still deduce that $u^2,v^2$ are central elements and only this fact  is used later (see \cite[Lem 9.18]{Pa13}). } that ${u}^2, v^2\in R'$,  hence  $\overline{u}^2,\overline{v}^2\in \overline{R}'$. On the other hand, we know that $uv+vu\in R'$ by (c), hence $\overline{u} \overline{v}\in \overline{R}'$ because $2$ is invertible in $\overline{R}'$ (recall $p\geq 5$).
\end{proof}
 
\begin{lemma}\label{lemma-III-2}
For any $(a,b)\in \F^2\backslash\{(0,0)\}$, $a\overline{u}+b\overline{v}\notin \fm_{\overline{R}'}\F[[\overline{u},\overline{v}]]$.
\end{lemma}
 
\begin{proof}
The condition $a\overline{u}+b\overline{v}\in \fm_{\overline{R}'}\cdot\F[[\overline{u},\overline{v}]]$ is equivalent to the existence of $\lambda_1, \lambda_2,\lambda_3,\lambda_4\in R'$ and $\phi\in \fm_{R'}E'$ such that (in $E'$)
\[au+bv=\phi+t'(\lambda_1+\lambda_2u+\lambda_3v+\lambda_4t').\]
Taking the involution $*$, we obtain 
\[\begin{array}{rll}-au-bv&=&\phi^*+(\lambda_1-\lambda_2u-\lambda_3v-\lambda_4t')(-t')\\
&=&\phi^{*}+t'(-\lambda_1-\lambda_2u-\lambda_3v+\lambda_4t'),\end{array}\]
where the first equality follows from (c) and the second from \cite[(160)]{Pa13}. 
This implies, as $2$ is invertible, 
\begin{equation}\label{equation-III-au+bv} au+bv=t'(\lambda_1+\lambda_2u+\lambda_3v)+\frac{\phi-\phi^{*}}{2}.\end{equation}
A computation using the relations established in the proof of \cite[Lem. 9.18]{Pa13} gives 
\[t'(\lambda_1+\lambda_2u+\lambda_3v)=\mu_2u+\mu_3v+\lambda_1t'\]
where $\mu_2,\mu_3\in R'$ are given by 
\[\mu_2=\lambda_2(uv+vu)+2\lambda_3v^2, \ \ \ \mu_3=-\lambda_3(uv+vu)-2\lambda_2u^2.\] 
In particular, $\mu_2,\mu_3\in \fm_{R'}$. On the other hand, since $\frac{\phi-\phi^{*}}{2}\in\fm_{R'}E'$, it can be written as $\lambda_1'+\lambda_2'u+\lambda_3'v+\lambda_4't'$ with $\lambda_i'\in\fm_{R'}$. Since $E'$ is free over $R'$ with basis $\{1,u,v,t'\}$, \eqref{equation-III-au+bv} forces $a,b\in\fm_{R'}$, hence  $a=b=0$.
\end{proof}

 \begin{proof}[Proof of Proposition   \ref{proposition-III-image}] 
The diagram \eqref{equation-III-diagram} shows that we may work with $R'$, $E'^{op}$, $E'^{ab}$ instead of $R$, $E$, $E^{ab}$, via  the isomorphism $\varphi^{-1}$.    
 
 (i) By Lemmas \ref{lemma-III-1} and \ref{lemma-III-2},  the ideal $\fm_{\overline{R}'} \F[[\overline{u},\overline{v}]]$ is equal to $(\overline{u}^2, \overline{v}^2, \overline{u}\overline{v})$, which is an ideal minimally generated by $3$ elements.  As a consequence, the embedding dimension of $\overline{R}'$ is greater or equal to $3$. But this embedding dimension is $\leq 3$, because $\overline{R}'$ a quotient of $R'$ whose embedding dimension is $3$.  Therefore,   $\overline{R}'$ is exactly the subring of $\F[[\overline{u},\overline{v}]]$ topologically generated by $\overline{u}^2,  \overline{v}^2,  \overline{u} \overline{v}$. This finishes the proof of (i).
 
(ii) It follows from (i) using Lemma \ref{lemma-III-mod-v}; for example we may take $S$ to be the image of $v$.
  
(iii) Let $I'=\mathrm{Ker}(R'\simto R\ra\F[[S]])$ and $J'=\mathrm{Ker}(E'^{op}\simto E\twoheadrightarrow \F[[S]])$. We need to show $J'^4\subset I'E'^{op}$ (note that $I'$ is contained in the centre of $E'$). On the one hand, we have   $t'\in J'$ by (b)  and $u\in J'$  by Lemma \ref{lemma-III-mod-v}. Since $E'^{op}/(t',u)\cong\F[[\overline{v}]]$, and the morphism $E'^{op}\twoheadrightarrow \F[[S]]$ is surjective, we must have $J'=(t',u)$. On the other hand, it is easy to see that $t'^2, u^2\in I'$. Using  the relation \cite[(160)]{Pa13}, one easily checks the desired inclusion.
\end{proof}

We note the following consequence of Proposition \ref{proposition-III-image}. 
\begin{corollary}\label{corollary-III-gh}
The kernel of $R \ra \F[[S]]$ is minimally generated by two elements.
\end{corollary}
\begin{proof}
By Proposition \ref{proposition-III-image}(ii), the image of $R \ra \F[[S]]$ is isomorphic to $\F[[S^2]]$, which is a regular local ring. It is a standard fact  that the kernel of a  surjective local morphism between two regular local rings can be generated by a regular sequence, see \cite[Thm. 21.2(ii)]{Mat}.
\end{proof}

\begin{remark}
In general, given a height two prime ideal $\mathfrak{p}$ in a $3$-dimensional regular local ring, e.g. $\F[[t_1,t_2,t_3]]$, it is not clear whether $\mathfrak{p}$ is the radical of an ideal generated by  $2$ elements, or equivalently, whether there exists a reduction of $\mathfrak{p}$ with  $2$ generators (\cite[Chap. 8]{HS}). 
\end{remark}

\subsection{Blocks of type (IV)}
\label{subsection-type-IV}

In this subsection, we complement some results in the work of Pa\v{s}k\={u}nas \cite{Pa13,Pa15} when $\mathfrak{B}$ is of type (IV).  Proposition \ref{proposition-Q-pi} in \S\ref{subsubsection-1}   was used in the proof of Proposition \ref{proposition-residue=finitelength}, but can be avoided as explained in Remark \ref{remark-unnecessary}. The results in \S\ref{subsubsection-2}, except Lemma \ref{lemma-IV-variables},  will not be used in this paper, but might be found useful elsewhere.

   The notation here is the same as in the previous subsections.  In particular,  $\pi\in \Rep_{\F}(G)$ is irreducible  of type (IV), and $P_{\pi^{\vee}}$ is a projective envelope of $\pi^{\vee}$ in $\mathfrak{C}$ and $ E_{\pi^{\vee}}=\End_{\mathfrak{C}}(P_{\pi^{\vee}})$.  Note that  the rings $E_{\pi^{\vee}}$ are naturally isomorphic (to $\F[[x,y,z,w]]/(xw-yz)$) for any $\pi\in\mathfrak{B}$ (see \cite[\S10]{Pa13}), so the subscript will be omitted in the rest (while the one of $P_{\pi^{\vee}}$ will be kept). Up to twist, we may assume $\mathfrak{B}=\{\ide_G,\Sp,\pi_{\alpha}\}$.

 \subsubsection{ $\F\otimes_E P_{\pi^{\vee}}$}\label{subsubsection-1}
 Our first aim is to determine $\F\otimes_EP_{\pi^{\vee}}$ for $\pi\in\mathfrak{B}$. 
For $\pi_1,\pi_2\in\Rep_{\F}^{\rm l, fin}(G)^{\mathfrak{B}}$ (in particular, $Z$ acts trivially on them), we write following \cite[\S10]{Pa13}:
\[e^1(\pi_1,\pi_2):=\dim_{\F}\Ext^1_{G/Z}(\pi_1,\pi_2).\]
 For convenience of the reader, we recall the list of $e^1(\pi_1,\pi_2)$ for $\pi_1,\pi_2\in\mathfrak{B}$, see   \cite[\S10.1]{Pa13}: 
\[e^1(\ide_G,\ide_G)=0,\ \ \ \ e^1(\Sp,\ide_G)=1,\ \ \ \ e^1(\pi_{\alpha},\ide_G)=1,\]
\[e^1(\ide_G,\Sp)=2,\ \ \ \ e^1(\Sp,\Sp)=0,\ \ \ \ e^1(\pi_{\alpha},\Sp)=0,\]
\[e^1(\ide_G,\pi_{\alpha})=0,\ \ \ \ e^1(\Sp,\pi_{\alpha})=1, \ \ \ \ e^1(\pi_{\alpha},\pi_{\alpha})=2.\]

We deduce that there exists a unique (up to isomorphism) non-split extension
\begin{equation}\label{equation-def-kappa}
0\ra \ide_G\ra\kappa\ra \pi_{\alpha}\ra0.\end{equation}
Also,  let $\tau_1$ be the universal extension of $\ide_G^{\oplus 2}$ by $\Sp$, i.e. we have 
\[0\ra \Sp\ra \tau_1\ra \ide_G^{\oplus2}\ra0\]
with $\rsoc_G\tau_1=\Sp$. 
\begin{lemma}\label{lemma-extension-dimension}
We have
\[e^1(\Sp,\kappa)=2,\ \ e^1(\pi_{\alpha},\tau_1)=2,\ \ e^1(\tau_1,\pi_{\alpha})=1.\]
\end{lemma}
\begin{proof}
See \cite[Lem. 10.18]{Pa13} for the first equality, \cite[Lem. 10.12]{Pa13} for the second,  \cite[(187)]{Pa13} and the argument before it for the third. 
\end{proof}

\begin{proposition}\label{proposition-Q-pi}
Let $\pi\in\mathfrak{B}$ and set $Q_{\pi^{\vee}}=\F\otimes_EP_{\pi^{\vee}}.$ In the following statements, the existence of the extension  is guaranteed by Lemma \ref{lemma-extension-dimension}.
\begin{enumerate}
\item[(i)] If $\pi=\ide_G$,  $Q_{\ide_G^{\vee}}$ is isomorphic to the universal extension of $\kappa^{\vee}$ by $(\Sp^{\vee})^{\oplus 2}$: 
\begin{equation}\label{equation-Q-ide}0\ra (\Sp^{\vee})^{\oplus 2}\ra Q_{\ide_G^{\vee}}\ra \kappa^{\vee}\ra0.\end{equation}
\item[(ii)] If $\pi=\Sp$,  $Q_{\Sp^{\vee}}$ is isomorphic to the universal extension of  $\tau_1^{\vee}$ by $(\pi_{\alpha}^{\vee})^{\oplus 2}$:   
\begin{equation}\label{equation-Q-Sp}0\ra (\pi_{\alpha}^{\vee})^{\oplus 2}\ra Q_{\Sp^{\vee}}\ra \tau_1^{\vee}\ra0.\end{equation}
\item[(iii)] If $\pi=\pi_{\alpha}$,  $Q_{\pi_{\alpha}^{\vee}}$ is isomorphic to the unique non-split extension of  $\pi_{\alpha}^{\vee}$ by $\tau_1^{\vee}$: 
\begin{equation}\label{equation-Q-pialpha}0\ra \tau_1^{\vee} \ra Q_{\pi_{\alpha}^{\vee}}\ra \pi_{\alpha}^{\vee}\ra0.\end{equation}
\end{enumerate}
 \end{proposition}
\begin{proof}
Note that  $Q_{\pi^{\vee}}$ is characterized as the maximal quotient of $P_{\pi^{\vee}}$ which contains  $\pi^{\vee}$ with multiplicity one, see \cite[Rem. 1.13]{Pa13}. We need to check that  if $\pi'$ is irreducible such that $\Ext^1_{G/Z}(\pi',(Q_{\pi^{\vee}})^{\vee})\neq0$, then $\pi'\cong \pi$. Proposition \ref{proposition-Gabriel} implies that we may assume $\pi'\in\mathfrak{B}$.

(i)  Write (in this proof) $\tau$ for the dual of the extension \eqref{equation-Q-ide}; we need to check  
\[\Ext^1_{G/Z}(\Sp,\tau)=0=\Ext^1_{G/Z}(\pi_{\alpha},\tau).\]
Since $e^1(\Sp,\Sp)=0$, the first equality follows from the  construction of $\tau$. The second is clear since $e^1(\pi_{\alpha},\Sp)=0$ (see the formulae recalled above) and $e^1(\pi_{\alpha},\kappa)=0$ by \cite[(194)]{Pa13}.

(ii) is proved in \cite[Lem. 4.4, (19)]{HT}.

(iii) is proved in \cite[Prop. 6.1, (35)]{Pa15}.
\end{proof}

\subsubsection{  $\Tor_i^E(\F,P_{\pi^{\vee}})$}\label{subsubsection-2}

Recall that if $\eta: T\ra \F^{\times}$ is a smooth character, we let $\pi_{\eta}=\Ind_{B}^G\eta$ and
\[\Pi_{\eta}=\Ind_B^G\rInj_T\eta,\ \ \ M_{\eta^{\vee}}=(\Pi_{\eta})^{\vee},\ \ \ E_{\eta^{\vee}}=\End_{\mathfrak{C}}(M_{\eta^{\vee}})\]
where $\rInj_T\eta$ denotes an injective envelope of $\eta$ in $\Rep_{\F}(T)$. In the rest of this subsection, we only consider $\eta\in\{\ide_T,\alpha\}$. By \cite[Prop. 7.1]{Pa13} there is a natural surjective ring homomorphism $q: E\twoheadrightarrow E_{\eta^{\vee}}$ induced by $P_{\pi_{\eta}^{\vee}}\twoheadrightarrow M_{\eta^{\vee}}$. 
\begin{lemma}\label{lemma-IV-variables}
In the isomorphism $E\cong \F[[x,y,z,w]]/(xw-yz)$, we may choose the variables so that the kernel of $q:E\twoheadrightarrow E_{\eta^{\vee}}$ is equal to $(z,w)$.
\end{lemma}
\begin{proof}
First,  via Colmez's functor we may identity $E$ with the special fiber of a certain universal Galois pseudo-deformation ring over $\cO:=W(\F)$, see \cite[Thm. 10.71]{Pa13}.  This ring is denoted by $R^{\psi}$ in \emph{loc. cit.} and we write $\overline{R}^{\psi}$ for its special fiber. Let $\mathfrak{r}$ denote the reducible locus of $R^{\psi}$ (see \cite[Cor. B.6]{Pa13} for its definition) and $\overline{\mathfrak{r}}$ its image in $\overline{R}^{\psi}$. Then by \cite[Cor. B.5, B.6]{Pa13}, $\overline{R}^{\psi}$ is isomorphic to $\F[[c_0,c_1,d_0,d_1]]/(c_0d_1+c_1d_0)$ and $\overline{\mathfrak{r}}=(c_0,c_1)$.  
On the other hand, via the natural isomorphism $E\cong \overline{R}^{\psi}$, $\ker(q)$ is identified with  $\overline{\mathfrak{r}}$ and $E_{\eta^{\vee}}$ with $\overline{R}^{\psi}/\overline{\mathfrak{r}}$, see \cite[Lem. 10.80]{Pa13}. This gives the result up to a change of variables. Note that the choice we make is not the one in \cite[Lem. 10.93]{Pa13}.
 \end{proof}

\begin{lemma}\label{lemma-Tor-Eeta}
We have 
\[\Tor_i^E(\F,M_{\eta^{\vee}})\cong (\pi_{\eta}^{\vee})^{\oplus 2},\ \  \forall i\geq 1.\]
\end{lemma}
\begin{proof}
By Lemma \ref{lemma-IV-variables}, we have a periodic (infinite) resolution of $E_{\eta^{\vee}}$ by free $E$-modules: 
\[\cdots \lra E^{\oplus 2}\overset{d'}{\lra}  E^{\oplus 2}\overset{d}{\lra} E^{\oplus 2}\overset{d'}{\longrightarrow}  E^{\oplus 2}\overset{d}{\longrightarrow} E^{\oplus 2}\overset{(w,z)}\longrightarrow E\ra E_{\eta^{\vee}}\ra0,\]
where $d$ is represented by the matrix $\smatr{x}{-z}{-y}{w}$,  sending $\binom{e_1}{e_2}$  to $\smatr{x}{-z}{-y}{w}\binom{e_1}{e_2}$; similarly $d'$ is represented by the matrix $\smatr{w}{z}{y}{x}$.
We deduce that
\[\Tor_i^{E}(\F,E_{\eta^{\vee}})\cong \F^{\oplus 2},\ \ \forall i\geq 1.\]
Because $M_{\eta^{\vee}}$ is a flat $E_{\eta^{\vee}}$-module by Lemma \ref{lemma-Meta-flat},   by flat base change we obtain 
\[\Tor_i^E(\F,M_{\eta^{\vee}})\cong \Tor_i^E(\F,E_{\eta^{\vee}})\otimes_{E_{\eta^{\vee}}}M_{\eta^{\vee}}\cong (\pi_{\eta}^{\vee})^{\oplus 2},\]
as required.
\end{proof}

\begin{lemma}\label{lemma-Hom(P,Tor)=0}
For  $i\geq 1$, we have 
$\Hom_{\mathfrak{C}}(P,\Tor^E_i(\F,P))=0.$
\end{lemma}
\begin{proof}
Choose a  resolution of $\F$ by finite free $E$-modules:
$F_{\bullet}\ra \F\ra0$.
Then the homology of $F_{\bullet}\otimes_EP$ computes $\Tor_i^{E}(\F,P)$. It is clear that \[\Hom_{\mathfrak{C}}(P,F_{\bullet}\otimes_E P)\cong F_{\bullet}.\]
Since $\Hom_{\mathfrak{C}}(P,-)$ is exact, this implies 
\[\Hom_{\mathfrak{C}}(P,H_i(F_{\bullet}\otimes_EP))\cong H_i(F_{\bullet})\]
as required.
\end{proof}

\begin{proposition}\label{proposition-Tor-P}
For any $i\geq 1$, we have 
 \[\Tor_i^E(\F,P_{\ide_G^{\vee}})\cong (\Sp^{\vee})^{\oplus 2}, \ \ \  \Tor_i^E(\F,P_{\pi_{\alpha}^{\vee}})=(\ide_G^{\vee})^{\oplus 2},\]
 and there is a short exact sequence
 \[0\ra (\pi_{\alpha}^{\vee})^{\oplus2}\ra \Tor_i^E(\F,P_{\Sp^{\vee}})\ra (\ide_G^{\vee})^{\oplus2}\ra0.\]
\end{proposition}
\begin{remark}
It is natural to ask if $\Tor_i^E(\F,P_{\Sp^{\vee}})$ is actually isomorphic to  $(\kappa^{\vee})^{\oplus 2}$, where $\kappa$ is defined by \eqref{equation-def-kappa}. 
\end{remark}
\begin{proof}
We first observe the following facts:
\begin{enumerate}
\item[(a)] $\SL_2(\Q_p)$ acts trivially on $\Tor_i^E(\F,P_{\pi_{\alpha}^{\vee}})$ for $i\geq 1$. Indeed, \cite[Cor. 10.43]{Pa13} states this for $i=1$ but the proof works for all $i\geq 1$.  This implies that  $\Tor_i^E(\F,P_{\pi_{\alpha}^{\vee}})$ is isomorphic to a finite direct sum of $\ide_G^{\vee}$ since $e^1(\ide_G,\ide_G)=0$.
\item[(b)] $\ide_G^{\vee}$ does not occur in $\Tor_i^{E}(\F,P_{\ide_G^{\vee}})$ for $i\geq 1$; this is a special case of Lemma \ref{lemma-Hom(P,Tor)=0}.
\end{enumerate}

Recall the following exact sequences 
\begin{equation}\label{equation-different-P-1'}0\ra P_{\pi_{\alpha}^{\vee}}\ra P_{\ide_G^{\vee}}\ra M_{\ide_T^{\vee}}\ra0,\end{equation} 
\begin{equation}\label{equation-different-P-4'}
0\ra P_{\Sp^{\vee}}\ra P_{\pi_{\alpha}^{\vee}}\ra M_{\alpha^{\vee}}\ra0,\end{equation}
see \cite[(234), (235)]{Pa13}. From 
\eqref{equation-different-P-1'} we obtain a long exact sequence
\[ \cdots\ra \Tor_1^E(\F,P_{\pi_{\alpha}^{\vee}}) {\ra} \Tor_1^E(\F,P_{\ide_G^{\vee}})\ra \Tor_1^E(\F,M_{\ide_T^{\vee}}) \ra Q_{\pi_{\alpha}^{\vee}}  \ra Q_{\ide_G^{\vee}}\ra \pi_{\ide_T}^{\vee}\ra0.\]
From 
 (a), (b), we deduce that the morphisms $ \Tor_i^E(\F,P_{\pi_{\alpha}^{\vee}})\ra \Tor_i^E(\F,P_{\ide_G^{\vee}})$ are zero  for $i\geq 1$, hence we obtain a long exact sequence 
\begin{equation}\label{equation-longsequence-1} 0\ra\Tor_1^E(\F,P_{\ide_G^{\vee}})\ra \Tor_1^E(\F,M_{\ide_T^{\vee}})\ra Q_{\pi_{\alpha}^{\vee}}\ra Q_{\ide_G^{\vee}}\ra \pi_{\ide_T}^{\vee}\ra0,\end{equation}
and short exact sequences for $i\geq 2$
\begin{equation}\label{equation-longsequence-2}0\ra \Tor_i^E(\F,P_{\ide_G^{\vee}})\ra \Tor_i^E(\F,M_{\ide_T^{\vee}})\ra \Tor_{i-1}^E(\F,P_{\pi_{\alpha}^{\vee}})\ra0.\end{equation}
 Using Proposition  \ref{proposition-Q-pi} and Lemma \ref{lemma-Tor-Eeta}, \eqref{equation-longsequence-1} implies   
$\Tor_1^E(\F,P_{\ide_G^{\vee}})\cong(\Sp^{\vee})^{\oplus 2}$, 
while \eqref{equation-longsequence-2} and (a), (b) imply  for $i\geq 2$
\[\Tor_{i-1}^E(\F,P_{\pi_{\alpha}^{\vee}})\cong(\ide_G^{\vee})^{\oplus 2},\ \ \Tor_{i}^E(\F,P_{\ide_G^{\vee}})\cong(\Sp^{\vee})^{\oplus 2}.\]
 This proves the first two assertions.

Similarly, the sequence \eqref{equation-different-P-4'} induces 
\[\cdots\ra \Tor_1^E(\F,P_{\Sp^{\vee}})\ra \Tor_1^E(\F,P_{\pi_{\alpha}^{\vee}})\ra \Tor_1^E(\F,M_{\alpha^{\vee}})\ra Q_{\Sp^{\vee}}\ra Q_{\pi_{\alpha}^{\vee}}\ra \pi_{\alpha}^{\vee}\ra0.\]
By Lemma \ref{lemma-Tor-Eeta} and (a), we see that the morphisms  $\Tor_i^E(\F,P_{\pi_{\alpha}^{\vee}})\ra \Tor_i^E(\F,M_{\alpha^{\vee}})$  are zero for $i\geq 1$, hence  obtain short exact sequences 
\[0\ra \Tor_{i+1}^E(\F,M_{\alpha^{\vee}})\ra\Tor_i^E(\F,P_{\Sp^{\vee}})\ra \Tor_i^E(\F,P_{\pi_{\alpha}^{\vee}})\ra0.\] 
Using Lemma \ref{lemma-Tor-Eeta} and what has been proved,  we deduce  the result for $\Tor_i^E(\F,P_{\Sp^{\vee}})$. 
\end{proof}

\subsubsection{Miracle flatness}
 In this subsection we prove the following result, using the  ``miracle flatness'' criterion in \cite[Prop. A.30]{GN}.  Recall  that the block $\mathfrak{B}$ is of type (IV).

\begin{proposition}\label{proposition-flatness}
There exists a  subring $R'\subset E$ which is a  regular local $\F$-algebra of Krull dimension $3$ such that $E$ is finite free over $R'$ and $P$ is flat over $R'$. 
\end{proposition}
 
\begin{proof} 
Choose a weight $\sigma\in\mathscr{D}(\pi)$. With the notation in  Proposition \ref{proposition-HT} we have  isomorphisms 
\begin{equation}\label{equation-miracle-F[[S]]}
\Hom_{K}^{\rm cont}(P,\sigma^{\vee})^{\vee}\cong E/J_{\sigma}\cong \F[[S]].\end{equation} 
Let $x_1\in E$ be a lift of $S$. Then $x_1$ is a regular element in $E$ (which is a domain). Since $E$ is a Cohen-Macaulay ring of Krull dimension $3$, we may extend $x_1$ to  a regular sequence in $E$, say $\{x_1,x_2,x_3\}$. Then the subring $R':=\F[[x_1,x_2,x_3]]$  is a regular local ring of Krull dimension $3$ and $E$ is finite over $R'$.  Moreover, the Auslander-Buchsbaum formula implies that $E$ is free over $R'$. 

We are left to prove that $P$ is flat over $R'$.   As in \cite[A.14]{GN}, we consider  
\[A=\Lambda \widehat{\otimes}_{\F} \F[[x_1]],\ \ B=\Lambda\widehat{\otimes}_{\F}R'\]
which are both Auslander regular rings (\cite[Def. A.2]{GN}). The natural inclusion $\F[[x_1]]\subset R'$ induces an inclusion   $A\subset B$. We may view $P$ as a module over both $A$ and $B$. The exact sequence $ P\overset{\times x_1}{\ra} P\ra P/x_1P\ra0$ induces a  sequence \[0\ra\Hom_{K}^{\rm cont}(P,\sigma^{\vee})^{\vee}\overset{\times S}{\ra}\Hom_{K}^{\rm cont}(P,\sigma^{\vee})^{\vee}\ra\Hom_K^{\rm cont}(P/x_1P,\sigma^{\vee})^{\vee}\ra0\]
which is exact by \eqref{equation-miracle-F[[S]]}.
By Proposition \ref{proposition-HT}, the above exact sequence still exists if we replace $\sigma$ by any irreducible $\sigma'\in \Rep_{\F}(K)$ with $\Hom_{K}^{\rm cont}(P,\sigma'^{\vee})^{\vee}\neq0$, and $\Hom_K^{\rm cont}(P/x_1P,\sigma'^{\vee})^{\vee}$ is $1$-dimensional over $\F$.  
Therefore, $P/x_1P$ is coadmissible and $P$ is finitely generated as an $A$-module (resp. as a $B$-module) by Nakayama's lemma.  On the other hand,  Proposition \ref{proposition-projective}   implies that $x_1$ is $P$-regular and $P/x_1P$ is also projective in $\mathfrak{C}(K)$; here we have used the main result of  \cite{EP} stating that $P$ remains projective in $\mathfrak{C}(K)$. 

In particular, we see that $P/x_1P$ is a Cohen-Macaulay module over $\Lambda$. Since $x_1$ is both $A$- and $P$-regular, a standard argument (using \cite[Lem. A.15]{GN} for example) shows that $P$ is  a Cohen-Macaulay $A$-module and 
\[\delta_A(P)=\delta_{\Lambda}(P/x_1P)+1=3+1=4.\]  
Using \cite[Cor. A.29]{GN}, $P$ is also a Cohen-Macaulay $B$-module with $\delta_B(P)=4$.  Since $E$ is finite free over $R'$, we have $\delta_{\Lambda}(\F\otimes_{R'}P)=1$ by Proposition \ref{proposition-residue=finitelength}, so  
\[j_B(P)=\dim B-\delta_B(P)=2=j_{\Lambda}(\F\otimes_{R'}P).\]  
By \cite[Prop. A.30]{GN}, we deduce that $P$ is flat over $R'$. 
\end{proof}

\begin{remark}\label{remark-flatness}
We construct  an explicit subring $R'$ of $E$ as  in Proposition \ref{proposition-flatness} as follows. Choose the isomorphism $E\cong \F[[x,y,z,w]]/(xw-yz)$ in such a way that the kernel of $q: E\twoheadrightarrow E_{\eta^{\vee}}$ is equal to $(z,w)$, see Lemma \ref{lemma-IV-variables}. Clearly, the elements $x, y-z, w$ form a regular sequence in $R$; 
let 
\begin{equation}\label{equation-define-R'} 
R':=\F[[x,y-z,w]].
\end{equation}
Then the composite  morphism $R'\hookrightarrow E\twoheadrightarrow E_{\eta^{\vee}}
$ remains surjective. 
In particular, a suitable linear combination of $x,y-z,w$ serves as a lift of $S$. This proves that $R'$ is one of the rings considered there.


\end{remark}

\section{Key computation}

We keep the notation of Section \ref{section-Paskunas}. For $n\geq 1$, let 
\[K_n=\matr{1+p^n\Z_p}{p^n\Z_p}{p^n\Z_p}{1+p^n\Z_p}, \ \ T_1(p^n)=\matr{1+p\Z_p}{p^n\Z_p}{p^n\Z_p}{1+p\Z_p}.\]
Recall that $\Lambda:=\F[[K_1/Z_1]]$.

\begin{theorem}\label{theorem-A}
Let $\Pi\in\Rep_{\F}(G)$ be an object of finite length. Then
\[\dim_{\F}\Pi^{T_1(p^n)}\ll n.\]
\end{theorem}

It is clear that we may assume $\Pi$ is irreducible in Theorem \ref{theorem-A}. Further, by the recall in \S\ref{subsection-Breuil},  up to twist it is enough to prove the following.  

\begin{theorem}\label{theorem-B}
For any $0\leq r\leq p-1$ and $\lambda\in \F$, we have 
\[\dim_{\F}\pi(r,\lambda,1)^{T_1(p^n)}\ll n.\]
\end{theorem} 

\subsection{Preparation}
We need some preparation to prove Theorem \ref{theorem-B}. To begin with, we establish a double coset decomposition formula in $K$. Let \[K_0(p^n)=\matr{\Z_p^{\times}}{\Z_p}{p^n\Z_p}{\Z_p^{\times}},\ \ H=\matr100{1+p\Z_p}.\]

\begin{lemma}\label{lemma-coset}
For any $n\geq 1$, we have 
\[|K_0(p^n)\backslash K/H|= (2n-1)(p-1)+2.\]
\end{lemma}

\begin{proof}
Let $A=\smatr abcd\in K$. We have the following facts:
\begin{enumerate}
\item[(i)]  if $A\in K_0(p)$, i.e. $c\in p\Z_p$, we have two subcases: 
\begin{enumerate}
\item[$\bullet$] if $c\in p^n\Z_p$, then $\smatr{a}bcd\in K_0(p^n)$; 
\item[$\bullet$] if $c\in p\Z_p\backslash p^n\Z_p$ (so $n\geq 2$), write $c=up^k$ with $u\in\Z_p^{\times}$ and $1\leq k\leq n-1$, then 
\[\matr abcd=\matr{d^{-1}(ad-bc)}{ubd^{-1}}0{u}\matr10{p^k}{[\lambda]}\matr{1}{0}{0}{t}\]
where $\lambda:=\overline{u^{-1}d}\in\F_p^{\times}$ and $t:=\frac{u^{-1}d}{[\lambda]}\in 1+p\Z_p$.
\end{enumerate}
We deduce  that 
\[K_0(p)=K_0(p^n)\bigcup \bigg(\bigcup_{1\leq k\leq n-1, \lambda\in\F_p^{\times}}K_0(p^n)\matr{1}0{p^k}{[\lambda]}H\bigg).\]
It is easy to check that this is a disjoint union, hence the cardinality of $K_0(p^n)\backslash K_0(p)/H$ is $1+(n-1)(p-1)$. 
\item[(ii)] if $A\notin K_0(p)$, i.e.  $c\in\Z_p^{\times}$,  we still have two subcases: 
\begin{enumerate} 
\item[$\bullet$] if $d\in\Z_p^{\times}$, then 
\[\matr abcd=\matr{-[\lambda]d^{-1}(ad-bc)}a0c\matr{0}11{[\lambda]}\matr{1}00{t}\]
where $\lambda:=\overline{c^{-1}d}\in\F_p^{\times}$ and $t=\frac{c^{-1}d}{[\lambda]}\in 1+p\Z_p$;
\item[$\bullet$] if  $d\in p\Z_p$, then 
\[\matr{a}bcd\matr0110=\matr badc\in K_0(p).\]
\end{enumerate}
\end{enumerate}
Combining (i) and (ii),   the cardinality of $K_0(p^n)\backslash K/H$ is equal to 
\[[1+(n-1)(p-1)]+[(p-1)+1+(n-1)(p-1)]= (2n-1)(p-1)+2.\]
\end{proof}

\begin{proposition}\label{proposition-T1-invariant}
Let $n\geq 1$ and $\sigma$ be a smooth $\F$-representation of $K_0(p^n)$ of finite dimension $d$. Let $V$ be a quotient $K$-representation of $\Ind_{K_0(p^n)}^K\sigma$, then $\dim_{\F}V^{H}\leq 2dpn$.
\end{proposition}
\begin{proof}
Let $W$ be the corresponding kernel so that we have an exact sequence
\[0\ra W\ra \Ind_{K_0(p^n)}^K\sigma\ra V\ra0.\]
Taking $H$-invariants, it induces
\[0\ra W^{H}\ra \big(\Ind_{K_0(p^n)}^K\sigma\big)^{H}\ra V^{H}\overset{\partial}{\ra} H^1(H,W),\]
hence an equality of dimensions
\begin{equation}\label{equation-equalitydim}\dim_{\F} W^{H}+\dim_{\F}V^{H}=\dim_{\F}\big(\Ind_{K_0(p^n)}^K\sigma\big)^{H}+\dim_{\F}\im(\partial).\end{equation}
Now note that $H\cong 1+p\Z_p\cong \Z_p$ is a pro-$p$ group of cohomological dimension $1$, so by Lemma \ref{lemma-H1(Zp)} below we have \[\dim_{\F}W^{H}=\dim_{\F}H^1(H,W)\geq \dim_{\F}\im(\partial),\]
hence by \eqref{equation-equalitydim}
\[\dim_{\F}V^{H}\leq \dim_{\F}\big(\Ind_{K_0(p^n)}^K\sigma\big)^{H}.\] 
We are thus reduced to prove the proposition in the special case $V=\Ind_{K_0(p^n)}^K\sigma$. Using  \cite[Lem. 3]{BL}, it is easy to see that  any irreducible smooth $\F$-representation of $K_0(p^n)$ is one-dimensional, so there exists a  filtration of $\sigma$ by sub-representations, of length $d$,  such that all graded pieces are one-dimensional.   Hence, we  may assume  $d=1$, in which case the result follows from Lemma \ref{lemma-coset}.
\end{proof}

\begin{lemma}\label{lemma-H1(Zp)}
Let $W$ be a finite dimensional smooth $\F$-representation of $\Z_p$, then
\[
\dim_{\F} H^1(\Z_p, W)=\dim_{\F}H^0(\Z_p,W).\]
\end{lemma}
\begin{proof}
This is clear if $\dim_{\F} W=1$ because then $
W$ must be the trivial representation of
$\Z_p$ so that
$H^1(\Z_p, W)\cong\Hom(\Z_p,\F)$
is of dimension 1. The general case is proved by induction on $\dim_{\F} W$ using the fact that $H^2(\Z_p,*)=0$ and  that  $W$ always contains a one-dimensional sub-representation. 
\end{proof}
\begin{remark}
In the proof of Proposition \ref{proposition-T1-invariant}, we crucially used the fact that $H$  has cohomological dimension $1$. This fact, very special to the group $\GL_2(\Q_p)$, is also used  in \cite{Br03} and \cite{Pa10} (but for the unipotent subgroup of $B(\Z_p)$). 
\end{remark}

\subsection{Supersingular case}
We give the proof of Theorem \ref{theorem-B} when $\Pi$ is supersingular, i.e. $\Pi=\pi(r,0,1)$ for some $0\leq r\leq p-1$. Since we have a $G$-equivariant isomorphism (\cite[Thm. 1.3]{Br03})
\[\pi(r,0,1)\cong \pi(p-1-r,0,\omega^r)\]
we may assume $r>0$ in the following.\medskip

Set $\sigma:=\Sym^r\F^2$ and for $n\geq 1$ denote by $\sigma_n$ the following representation of $K_0(p^n)$:
\[\sigma_n\bigg(\matr ab{p^nc}d\bigg):=\sigma\bigg(\matr dc{p^nb}a\bigg).\]
Let $R_0:=\sigma$ and $R_n:=\Ind_{K_0(p^n)}^K\sigma_n$ for $n\geq 1$. It is easy to see that 
\begin{equation}\label{equation-dim-Rn}\dim_{\F}R_0=(r+1),\ \ \dim_{\F}R_n=(r+1)(p+1)p^{n-1},\ \forall n\geq 1.\end{equation}
Moreover, the following properties hold (see \cite[\S4]{Br07}):
\begin{enumerate}
\item[(i)] $\cInd_{KZ}^G\sigma|_{K}\cong \oplus_{n\geq0}R_n$;
\item[(ii)] the Hecke operator $T|_{R_n}:R_n\ra R_{n+1}\oplus R_{n-1}$ is the sum of a $K$-equivariant injection $T^+:R_n\hookrightarrow R_{n+1}$ and (for $n\geq 1$) a $K$-equivariant surjection $T^-:R_n\twoheadrightarrow R_{n-1}$.
\item[(iii)] we have an isomorphism of $K$-representations
\begin{equation}\label{equation-pi-twoparts}\pi(r,0,1)\cong \big(\varinjlim_{n\ \rm{even}}R_0\oplus_{R_1}\oplus R_2\oplus_{R_3}\oplus\cdots\oplus R_n\big)\oplus \big(\varinjlim_{n\ \rm{odd}}(R_1/R_0)\oplus_{R_2}R_3\oplus_{R_4}\oplus \cdots\oplus R_n\big).\end{equation}
\end{enumerate}
Denote by $\Pi_0=\varinjlim\limits_{n\ \rm{even}}$  and $\Pi_1=\varinjlim\limits_{n\ \rm{odd}}$ the two direct summands of $\Pi$ in \eqref{equation-pi-twoparts}. For all $n\geq 0$, we let $\overline{R}_n$ be the image of $R_n\ra \pi(r,0,1)$. Then $\overline{R}_n\subset \Pi_0$ if $n$ is even, and $\overline{R}_n\subset \Pi_1$ if $n$ is odd.   
\begin{lemma}\label{lemma-dim-Rn}
For all $n\geq 0$, we have $\overline{R}_{n}\subset \overline{R}_{n+2}$ and 
$\dim_{\F}\overline{R}_{n}=(r+1)p^n$.
\end{lemma}
\begin{proof}
The inclusion $\overline{R}_n\subset \overline{R}_{n+2}$ follows from (ii) and \eqref{equation-pi-twoparts}. The dimension formula follows from \eqref{equation-pi-twoparts} using \eqref{equation-dim-Rn}.
Precisely, \eqref{equation-pi-twoparts} implies that if $n$ is even  then by \eqref{equation-dim-Rn} \[\dim_{\F}\overline{R}_n=\sum_{k=0}^{n}(-1)^k\dim_{\F}R_k=(r+1)\Big(1+(p+1)\sum_{k=1}^{n}(-1)^kp^{k-1}\Big)=(r+1)p^n,
\]
 and if $n$ is odd  then similarly
\[\dim_{\F}\overline{R}_n=\sum_{k=0}^{n}(-1)^{k+1}\dim_{\F}R_k=(r+1)p
^n.\]
\end{proof}

At this point, we need the following result of Morra. Recall that $\Pi=\Pi_0\oplus \Pi_1$ as $K$-representations.
\begin{theorem}\label{theorem-Morra-SS}
Let $n\geq 1$. For $i\in\{0,1\}$, the dimension of $K_n$-invariants of $\Pi_i$ satisfies
\[\dim_{\F}\Pi_i^{K_n}\leq(p+1)p^{n-1}.\]
Moreover, $\Pi_i$ is nearly uniserial in the following sense: if $W_1,W_2$ are two $K$-stable subspaces of $\Pi_i$ such that 
\[\dim_{\F}W_2-\dim_{\F}W_1\geq p,\]
then $W_1\subset W_2$. 
\end{theorem}
\begin{proof}
See \cite[Cor. 4.14, 4.15]{Mor} for the dimension formula. Note that the formula in \emph{loc. cit.} is for the dimension of $\Pi_0^{K_n}\oplus \Pi_1^{K_n}$. The second statement follows from \cite[Thm. 1.1]{Mor11} which describes the $K$-socle filtration of $\Pi_i$. To explain this, fix $i\in\{0,1\}$. By \cite[Thm. 1.1]{Mor11}, $\Pi_i$ admits an increasing filtration $\Fil^k\Pi_i$, $k\geq 0$ such that
\[\Fil^0\Pi_i=0,\ \ \Fil^1\Pi_i=\rsoc_K\Pi_i,\ \ \Fil^{k+1}\Pi_i/\Fil^k\Pi_i\cong \Ind_{B(\F_p)}^{\GL_2(\F_p)}\chi_k,\ \forall k\geq 2,\]
for suitable characters $\chi_k:B(\F_p)\ra \F^{\times}$. In particular, the graded pieces have dimension $p+1$ except for the first.
Moreover, the filtration satisfies the property that for any $K$-stable subspace $W\subset\Pi_i$ and any $k
\leq k'$, the condition  $\dim_{\F}\Fil^k\Pi_i\leq \dim_{\F}W\leq \dim_{\F}\Fil^{k'}\Pi_i$ implies $\Fil^k\Pi_i\subset W\subset \Fil^{k'}\Pi_i$.  

Now, for the given $W_1$ let $k_1$ be the smallest index such that $W_1\subset \Fil^{k_1}\Pi_i$; then $\dim_{\F}\Fil^{k_1}\Pi_i-\dim_{\F}W_1\leq p$. The assumption then implies $\dim_{\F}\Fil^{k_1}\Pi_i\leq \dim_{\F}W_2$ and  that $W_2$ contains $\Fil^{k_1}\Pi_i$, proving the result.    
\end{proof}

\begin{corollary}\label{corollary-Morra-SS}
We have 
$\Pi^{K_n}\subset \overline{R}_n\oplus \overline{R}_{n+1}$.
\end{corollary} 

\begin{proof}
We have assumed $r\geq 1$, so by Lemma \ref{lemma-dim-Rn} we get for $n\geq 1$: \[\dim_{\F}\overline{R}_n\geq 2p^n\geq (p+1)p^{n-1}+p\geq \dim_{\F}\Pi_i^{K_n}+p.\] By the nearly uniserial property of $\Pi_i$, this implies   
$\Pi_0^{K_n}\subset \overline{R}_n$ if $n$ is even, 
while  $\Pi_1^{K_n}\subset \overline{R}_n$ if $n$ is odd. Putting them together, we obtain 
the result.
\end{proof}

\begin{proof}[Proof of Theorem \ref{theorem-B} when $\lambda=0$]
Since $T_1(p^n)$ contains $K_n$, we have an inclusion  $\Pi^{T_1(p^n)}\subset \Pi^{K_n}$, so   Corollary \ref{corollary-Morra-SS} implies $\Pi^{T_1(p^n)}\subset \overline{R}_n\oplus \overline{R}_{n+1}$, hence
\[\Pi^{T_1(p^n)}\subset (\overline{R}_n)^{T_1(p^n)}\oplus (\overline{R}_{n+1})^{T_1(p^n)}\subset (\overline{R}_n)^{H}\oplus (\overline{R}_{n+1})^{H}.\]
Noting that $\dim_{\F}\sigma\leq p$, we obtain by  Proposition \ref{proposition-T1-invariant}: 
\[\dim_{\F} \Pi^{T_1(p^n)}\leq \dim_{\F}(\overline{R}_n)^{H}+\dim_{\F}(\overline{R}_{n+1})^{H}\leq 4p^2n, \]
hence the result.
\end{proof}

\subsection{Non-supersingular case}

Assume from now on $\lambda\neq 0$. We define the subspaces  $R_n$ ($n\geq 0$) of $\cInd_{KZ}^G\sigma$ as above. We still have the properties (i) and (ii) recalled there. The only difference, also the key difference, with the supersingular case is that the induced morphisms 
$R_n\ra \pi(r,\lambda,1)$
are all injective (because $\lambda\neq0$). Moreover, if we write $\overline{R}_n$ for the image of $R_n$ in $\pi(r,\lambda,1)$, then  $\overline{R}_n\subset \overline{R}_{n+1}$
 and
 \[\pi(r,\lambda,1)=\varinjlim_{n\geq 0} \overline{R}_n. \]  
 
\begin{proposition}\label{prop-Morra-PS}
Let $n\geq 1$, we have an inclusion
$\pi(r,\lambda,1)^{K_n}\subset \overline{R}_n$.
\end{proposition} 
\begin{proof}
By \cite[Thm. 1.2]{Mor11}, $\pi(r,\lambda,1)$ satisfies a (nearly) uniserial property as in the supersingular case.  Moreover, we have (see \cite[\S5]{Mor})
$\dim_{\F}\pi(r,\lambda,1)^{K_n}=(p+1)p^{n-1}$ 
while \[\dim_{\F}\overline{R}_n=\dim_{\F}R_n=(r+1)(p+1)p^{n-1}.\]
We then conclude as in the supersingular case.
\end{proof}
\begin{proof}[Proof of Theorem \ref{theorem-B} when $\lambda\neq0$] Since $T_1(p^n)$ contains $K_n$, we obtain by Proposition \ref{prop-Morra-PS}
\[\pi(r,\lambda,1)^{T_1(p^n)}\subset  (\overline{R}_n)^{T_1(p^n)}\subset (\overline{R}_n)^{H}.\]
The result then follows from Proposition \ref{proposition-T1-invariant}.
\end{proof}

\section{Main results}\label{section-generalization}
 For application in \S\ref{section-global}, we need to generalize Theorem \ref{theorem-A}  to higher cohomological degrees and to representations of a finite product of   $\GL_2(\Q_p)$.  


We let $G=\GL_2(\Q_p)$, $K=\GL_2(\Z_p)$ and other subgroups of $G$ are defined as in the previous section. Given $r\geq 1$, we let 
\[\mathscr{G}=\prod_{i=1}^rG,\ \ \ \mathscr{K}=\prod_{i=1}^rK,\ \  \ \mathscr{K}_1=\prod_{i=1}^r K_1,\]
\[\mathscr{Z}_1=\prod_{i=1}^rZ_1,\ \ \  \Lambda=\F[[\mathscr{K}_1/\mathscr{Z}_1]]\cong\widehat{\otimes}_{i=1}^r\F[[K_1/Z_1]].\]
That is, $\mathscr{G}$ is a product of $r$ copies of $G$, and so on. If $\bn=(n_1,...,n_r)\in(\Z_{\geq 1})^r$,  let
\[\mathscr{T}_1(p^{\bf n})=\prod_{i=1}^rT_1(p^{n_i}),\ \ \ \mathscr{K}_{\bn}=\prod_{i=1}^rK_{n_i}.\]
As in the case $r=1$, let $\Rep_{\F}(\mathscr{G})$ denote the category of smooth $\F$-representations of $\mathscr{G}$ with a central character and $\Rep_{\F}^{\rm l,fin}(\mathscr{G})$ the subcategory consisting of locally finite objects. 
Let $\mathfrak{C}(\mathscr{G})$ be the dual category of $\Rep_{\F}^{\rm l,fin}(\mathscr{G})$ under Pontryagin dual and $\mathfrak{C}^{\rm fg,tor}(\mathscr{G})$   the subcategory of $\mathfrak{C}(\mathscr{G})$ consisting of coadmissible torsion objects. 

 \subsection{Generalisations}\label{subsection-lemmas}
\subsubsection{Blocks} 
For $1\leq i\leq r$, let $\pi_i\in \mathfrak{C}(G)$ be  (absolutely) irreducible and $\mathfrak{B}_i$ be the block in which $\pi_i$ lies. Recall that $\pi_i$ is admissible (see \S\ref{subsection-Breuil}).

\begin{lemma}\label{lemma-tensor-pi}
(i) The tensor product $\pi_1\otimes \cdots\otimes  \pi_r$ is an   irreducible admissible representation of $\mathscr{G}$.
Conversely, up to enlarge $\F$, each  irreducible representation $\pi$ in $\Rep_{\F}(\mathscr{G})$ is of this form; in particular, $\pi$ is admissible.

(ii) Let $\pi=\pi_1\otimes\cdots\otimes\pi_r$ be as in (i). Then $\delta_{\Lambda}(\pi^{\vee})$ is equal to the cardinality of $i\in\{1,...,r\}$ such that $\pi_i$ is infinite dimensional. 

(iii) Let $\pi=\otimes_{i=1}^r\pi_i$ and $\pi'=\otimes_{i=1}^{r}\pi_i'$ be irreducible representations in $\Rep_{\F}(\mathscr{G})$. Then $\pi\sim \pi'$ (i.e. in the same block) if and only if $\pi_i\sim\pi_i'$ for all $i$.
  \end{lemma}
\begin{proof}
(i) The first part is standard. For the second, see \cite[Lem. B.7]{GN} for a proof. Note that we only need to assume $\pi$ carries a central character, then $\pi$ is automatically admissible; the proof uses the classification of irreducible objects in $\Rep_{\F}(G)$   (cf. \cite{BL},\cite{Br03}). 

(ii) It is a direct consequence of Theorems \ref{theorem-Paskunas-Morra}, using \cite[Lem. A.11]{GN}.

(iii) It follows from the fact that $\Ext^1_{\mathscr{G}}(\pi,\pi')\neq0$ if and only if there exists $1\leq i\leq r$ such that  $\Ext^1_{G}(\pi_i,\pi_i')\neq0$ and $\pi_j\cong \pi_j'$ for all $j\neq i$; see \cite[Lem. 3.4.10]{Pan} for a proof.
\end{proof}

Let $\mathfrak{B}$ be the block in which $\pi=\otimes_{i=1}^r\pi_i$ lies. As a consequence of  Lemma \ref{lemma-tensor-pi},  $\mathfrak{B}$ is equal to $\mathfrak{B}_1\otimes\cdots\otimes \mathfrak{B}_r:=\{\otimes_{i=1}^r\pi_i': \pi_i'\in\mathfrak{B}_i\}$.

Let $P_{\pi_i^{\vee}}$ be a projective envelope of $\pi_{i}^{\vee}$ in $\mathfrak{C}(G)$ and set $E_{\pi_i^{\vee}}=\End_{\mathfrak{C}(G)}(P_{\pi_i^{\vee}})$. Write
\[P:=\widehat{\otimes}_{i=1}^rP_{\pi_i^{\vee}}, \ \ \ E:=\widehat{\otimes}_{i=1}^rE_{\pi_i^{\vee}},\]
where $\widehat{\otimes}$ denotes the completed tensor product   over $\F$ (see \cite[\S B.1]{GN} for the definition and basic properties). For each $i$, let $R_{\pi_i^{\vee}}$ be the centre of $E_{\pi_i^{\vee}}$ and set 
\[R:=\widehat{\otimes}_{i=1}^rR_{\pi_i^{\vee}}\]
Then  $R$ is 
contained in  the centre of $E$.\footnote{It seems that $R$ is exactly the centre of $E$, see  \cite[Rem. 3.4.12]{Pan}; but we don't need this property.}

\begin{lemma}\label{lemma-tensor-P}
(i) $P$ is a projective envelope of $\widehat{\otimes}_{i=1}^r\pi_i^{\vee}$  in $\mathfrak{C}(\mathscr{G})$ and $\End_{\mathfrak{C}(\mathscr{G})}(P)\cong E$.

(ii) $R$  is a Noetherian complete local  $\F$-algebra, Cohen-Macaulay of Krull dimension $3r$.   Moreover, $E$ is finite free over $R$. 

(iii)  $\F\otimes_EP$ (resp. $\F\otimes_{R}P$) has  finite length in $\mathfrak{C}(\mathscr{G})$ and $\delta_{\Lambda}(\F\otimes_EP)=\delta_{\Lambda}(\F\otimes_RP)=r$. 
\end{lemma}
\begin{proof}
(i) It is proved in \cite[Lem. B.8]{GN}.

(ii) It follows from Theorem \ref{theorem-Paskunas}.

(iii) It follows  from Proposition \ref{proposition-residue=finitelength} and \cite[Lem. A.11]{GN}.  
\end{proof}

\begin{lemma}\label{lemma-pi-r-finitelength}
If $M\in\mathfrak{C}(\mathscr{G})$ has finite length, then $\delta_{\Lambda}(M)\leq r$ and $\dim_{\F}M_{\mathscr{T}_1(p^{\bf n})}\ll \prod_{i=1}^rn_i$.
\end{lemma}

\begin{proof}
We may assume $M$ is irreducible, so that $M \cong \widehat{\otimes}_{i=1}^r\pi_{i}^{\vee}$ with each $\pi_i$ irreducible. The result then follows from Lemma \ref{lemma-tensor-pi}(ii) and   Theorem \ref{theorem-A}.
\end{proof}

\subsubsection{Coadmissibility} Similar to Lemma \ref{lemma-tensor-pi}(i), an irreducible representation of $\mathscr{K}$ is of the form $\sigma=\otimes_{i=1}^r\sigma_i$ with each $\sigma_i\in \Rep_{\F}(K)$ irreducible. We have the obvious notion of a \emph{Serre weight} for $\pi=\otimes_{i=1}^r\pi_i$ as in \S\ref{subsection-Serreweight}; denote by $\mathscr{D}(\pi)$ the set of Serre weights of $\pi$. Clearly, $\sigma\in\mathscr{D}(\pi)$ if and only if $\sigma_i\in\mathscr{D}(\pi_i)$ for each $i$. The following lemma is a direct generalisation of Lemma \ref{lemma-D(pi)neqD(pi')} and Proposition \ref{proposition-HT}.

\begin{lemma}\label{lemma-Serre-r}
(i) If $\pi\neq \pi'$ are two objects in a block $\mathfrak{B}$, then $\mathscr{D}(\pi)\cap \mathscr{D}(\pi')=\emptyset$.

(ii) Let $\sigma\in\Rep_{\F}(\mathscr{K})$ be irreducible. Whenever  non-zero, $\Hom_{\mathscr{K}}^{\rm cont}(P,\sigma^{\vee})^{\vee}$ is a cyclic $E$-module and if $J_{\sigma}$ denotes its annihilator then $E/J_{\sigma}\cong\F[[S_1,...,S_r]]$. 
If $\sigma=\otimes_i\sigma_i$ and $\sigma'=\otimes_i\sigma_i'$ are two irreducible $\mathscr{K}$-representations such that $\sigma_i=\sigma_i'$ whenever $\pi_i$ is supersingular (i.e. $\mathfrak{B}_i$ is of type (I)), then  \[\Hom_{\mathscr{K}}^{\rm cont}(P,\sigma^{\vee})^{\vee}\cong \Hom_{\mathscr{K}}^{\rm cont}(P,\sigma'^{\vee})^{\vee}\] as $E$-modules when they are both non-zero. 

(iii) Let $\widetilde{\sigma}=\oplus_{\sigma}\sigma$ where the sum is taken over all irreducible $\sigma$ such that $\Hom_{\mathscr{K}}^{\rm cont}(P,\sigma^{\vee})\neq0$.  
Then $\Hom_{\mathscr{K}}^{\rm cont}(P,\widetilde{\sigma}^{\vee})^{\vee}$ is a Cohen-Macaulay $R$-module of Krull dimension $r$. 
\end{lemma}
\begin{proof}
The universal property of the completed tensor product, see the proof of  \cite[Lem. B.8]{GN},  gives 
\[\Hom_{\mathscr{K}}^{\rm cont}(\widehat{\otimes}_{i=1}^r\pi_i^{\vee},\widehat{\otimes}_{i=1}^{r}\sigma_i^{\vee})\cong \widehat{\otimes}_{i=1}^r\Hom_{K}^{\rm cont}(\pi_i^{\vee},\sigma_i^{\vee}).\]
This proves (i) using Lemma \ref{lemma-D(pi)neqD(pi')}. (ii) is proved in a similar way using Proposition \ref{proposition-HT}(i). 

(iii) One checks that $\widetilde{\sigma}=\otimes_{i=1}^r\widetilde{\sigma}_i$, where $\widetilde{\sigma}_i$ is the direct sum of all irreducible $\sigma_i$ such that $\Hom_K^{\rm cont}(P_{\pi_i^{\vee}},\sigma_i^{\vee})\neq0$. Hence 
\[\Hom_{\mathscr{K}}^{\rm cont}(P,\widetilde{\sigma}^{\vee})^{\vee}\cong \widehat{\otimes}_{i=1}^r\Hom_K^{\rm cont}(P_{\pi_i^{\vee}},\widetilde{\sigma}_i^{\vee})^{\vee}\]
as $R$-modules. The result follows from Proposition \ref{proposition-HT}(ii) which says that each component $\Hom_K^{\rm cont}(P_{\pi_i^{\vee}},\widetilde{\sigma}_i^{\vee})^{\vee}$ is a Cohen-Macaulay $R_{\pi_i^{\vee}}$-module of Krull dimension $1$.
\end{proof}

If $M\in\mathfrak{C}^{\mathfrak{B}}$, let $\mathrm{m}(M):=\Hom_{\CG}(P,M)$ which is a compact right $E$-module. If $\tau\in\Rep_{\F}(\mathscr{K})$  is of finite length, let \[P(\tau):=\Hom_{\mathscr{K}}^{\rm cont}(P,\tau^{\vee})^{\vee}\] which is a finitely generated left $E$-module.
 \begin{proposition}\label{proposition-M-to-m(M)-adm}
Let $M\in\mathfrak{C}(\mathscr{G})$ be a coadmissible quotient of $P$. Then
 $\mathrm{m}(M)\otimes_EP$ is coadmissible.
\end{proposition}
\begin{proof}
The  proof is similar to the proof of Proposition \ref{proposition-M-to-m(M)}(i).
Let $\Ker$ be the kernel of the natural morphism 
\[\mathrm{ev}: \mathrm{m}(M){\otimes}_EP\twoheadrightarrow M,\]
which is surjective by \cite[Lem. 2.10]{Pa13}.
Using \cite[Lem. 2.9]{Pa13} and the projectivity of $P$, we have $\Hom_{\mathfrak{C}(\mathscr{G})}(P,\Ker)=0$, that is $\pi^{\vee}$ does not occur in $\Ker$.  

We need to show that $\Hom_{\mathscr{K}}^{\rm cont}(\mathrm{m}(M){\otimes}_EP,\sigma^{\vee})$ is finite dimensional for any irreducible $\sigma\in\Rep_{\F}(\mathscr{K})$. By \cite[Prop. 2.4]{Pa15}, we have an isomorphism of finitely generated $E$-modules 
\[\Hom_{\mathscr{K}}^{\rm cont}(\mathrm{m}(M){\otimes}_EP,\sigma^{\vee})^{\vee}\cong \mathrm{m}(M){\otimes}_EP(\sigma).\]
Hence, it suffices to consider those $\sigma$ such that $P(\sigma)\neq 0$, or  equivalently $\Hom_{K}^{\rm cont}(P_{\pi_{i}^{\vee}},\sigma_i^{\vee})\neq 0$ for all $i$. Note that this implies automatically $\sigma_i\in\mathscr{D}(\pi_i)$ if $\pi_i$ is supersingular, i.e. $\mathfrak{B}_i$ is of type (I). 
We choose another weight $\sigma'=\otimes_{i=1}^r\sigma_i'$ as follows: if $\mathfrak{B}_i$ is of type (I), let $\sigma_i':=\sigma_i$; otherwise, let $\sigma_i'$ be any weight in $\mathscr{D}(\pi_i)$. Then the assumption of Lemma \ref{lemma-Serre-r}(ii) is satisfied and we obtain an isomorphism of $E$-modules
$P(\sigma)\cong P(\sigma')$. 
Moreover,  we have $\sigma'\in\mathscr{D}(\pi)$ by construction, hence  $\Hom_{\mathscr{K}}^{\rm cont}(\Ker,\sigma'^{\vee})=0$ by Lemma \ref{lemma-Serre-r}(i) because $\pi^{\vee}$ does not occur as a subquotient in $\Ker$. 
As in the proof of Proposition \ref{proposition-M-to-m(M)}(i), we obtain isomorphisms
\[\Hom_{\mathscr{K}}^{\rm cont}(\mathrm{m}(M){\otimes}_EP,\sigma^{\vee})^{\vee}\cong \Hom_{\mathscr{K}}^{\rm cont}(\mathrm{m}(M){\otimes}_EP,\sigma'^{\vee})^{\vee}\cong \Hom_{\mathscr{K}}^{\rm cont}(M,\sigma'^{\vee})^{\vee}\]
from which the result follows.
\end{proof}
 
 \begin{proposition}\label{proposition-EtoR-adm}
 Let $\mathrm{m}$ be a finitely generated right $E$-module such that $\mathrm{m}\otimes_EP$ is coadmissible. Then $\mathrm{m}\otimes_RP$ is also coadmissible.
 \end{proposition}
\begin{proof}
 As in the proof of Proposition \ref{proposition-M-to-m(M)-adm}, we need to show that $\mathrm{m}\otimes_RP(\sigma)$ is finite dimensional for any irreducible $\sigma\in\Rep_{\F}(\mathscr{K})$. Let $I_{\sigma}\subset R$ (resp. $J_{\sigma}\subset E$) be the annihilator of $P(\sigma)$. Then we have an isomorphism $E/J_{\sigma}\cong P(\sigma)$, see Lemma \ref{lemma-Serre-r}(ii). Since $P(\sigma)$ is finitely generated over $R$, a standard argument  in commutative algebra  shows that the Krull dimension of  $\mathrm{m}\otimes_RP(\sigma)$ is equal to that of $\mathrm{m}\otimes_RR/I_{\sigma}$.   Letting $J_{\sigma}':=I_{\sigma} E$ which is a two-sided ideal of $E$ contained in $J_{\sigma}$, there is an isomorphism 
\[\mathrm{m}\otimes_RR/I_{\sigma}\cong \mathrm{m}\otimes_EE/J_{\sigma}',\]
so that we are left to compare $E/J_{\sigma}'$ and $E/J_{\sigma}$.
Since $\mathrm{m}\otimes_EE/J_{\sigma}$ is finite dimensional over $\F$ by assumption, it suffices to show that $J_{\sigma}^n\subset J_{\sigma}'$ for some  $n\geq 1$, because then 
\[\delta_{\Lambda}(\mathrm{m}\otimes_EE/J_{\sigma}^n)\geq \delta_{\Lambda}(\mathrm{m}\otimes_EE/J_{\sigma}')\geq \delta_{\Lambda}(\mathrm{m}\otimes_EE/J_{\sigma}),\]
and  consequently \[\delta_{\Lambda}(\mathrm{m}\otimes_EE/J_{\sigma})=\delta_{\Lambda}(\mathrm{m}\otimes_EE/J_{\sigma}')\]
because $E$ is Noetherian. Recall that $E$ and $R$ only differ at the indices $i$ where $\mathfrak{B}_i$ is of type (III), in which case $J_{\sigma_i}^4\subset J_{\sigma_i}'$ by Proposition \ref{proposition-III-image}(iii). The result follows.
\end{proof}

\subsubsection{Dimension formula} \label{subsection-dimformula}
We introduce another ring which lies between $R$ and $E$: let
\[E':=\big(\widehat{\otimes}_{i\notin \mathrm{IV}}E_{\pi_i^{\vee}}\big)\widehat{\otimes}\big(\widehat{\otimes}_{i\in \mathrm{IV}}R_{\pi_i^{\vee}}'\big)\]
 where the subscript $i\in \mathrm{IV}$ (resp. $i\notin\mathrm{IV}$) indicates that  $\mathfrak{B}_i$ is (resp. not)  of type (IV), and  $R_{\pi_i^{\vee}}'$ is one of the  subrings of $E_{\pi_i^{\vee}}$ constructed in Proposition  \ref{proposition-flatness} (e.g. the one constructed in Remark \ref{remark-flatness}). 
By Theorem \ref{theorem-Paskunas} and Proposition \ref{proposition-flatness}, $E$ is finite free over $E'$ and $P$ is flat over $E'$.
\begin{lemma}\label{lemma-compare-GKdimension}
Let $\mathrm{m}$ be a  finitely generated right $E$-module and assume that $\mathrm{m}\otimes_EP$ is coadmissible. Then $\mathrm{m}\otimes_{E'}P$ is also coadmissible and $\delta_{\Lambda}(\mathrm{m}\otimes_{E'}P)=\delta_{\Lambda}(\mathrm{m}\otimes_{E}P)$.
\end{lemma}
\begin{proof}
The first assertion can be deduced from Proposition \ref{proposition-EtoR-adm}; but the proof below gives another proof.  

We claim that there exists an exact sequence of $(E,E)$-bimodules for some $n\geq 1$
\begin{equation}\label{equation-E-otimes-E}E^{\oplus n}\ra E\otimes_{E'}E\ra E\ra0,\end{equation}
where the second morphism sends $x\otimes y$ to $xy$. In fact, by definition of $E'$, we only need to construct such an exact sequence for $E$ (resp. $E'$) replaced by $\widehat{\otimes}_{i\in \mathrm{IV}}E_{\pi_i^{\vee}}$ (resp. $\widehat{\otimes}_{i\in \mathrm{IV}}R_{\pi_i^{\vee}}'$), and \eqref{equation-E-otimes-E} can be obtained by tensoring it with $\widehat{\otimes}_{i\notin \mathrm{IV}}E_{\pi_i^{\vee}}$.
But since $\widehat{\otimes}_{i\in \mathrm{IV}}E_{\pi_i^{\vee}}$ and $\widehat{\otimes}_{i\in \mathrm{IV}}R_{\pi_i^{\vee}}'$  are commutative  Noetherian rings,  the claim is standard.

We have a natural isomorphism $\mathrm{m}\otimes_{E'}P\cong \mathrm{m}\otimes_E(E\otimes_{E'}E)\otimes_EP$, which together with \eqref{equation-E-otimes-E} gives an exact sequence
\[(\mathrm{m}\otimes_EP)^{\oplus n}\ra \mathrm{m}\otimes_{E'}P\ra \mathrm{m}\otimes_EP\ra0.\]
The result follows from this.
\end{proof}

We prove the following dimension formula.   
\begin{proposition}\label{proposition-dim-equality}
Let $\mathrm{m}$ be a non-zero finitely generated (right) $E$-module and assume that $\mathrm{m}\otimes_EP$ is coadmissible. Then we have an equality 
\[\delta_{\Lambda}(\mathrm{m}\otimes_EP)=\dim_{R}\mathrm{m}+r,\]
where $\dim_R\mathrm{m}$ denotes the Krull dimension of $\mathrm{m}$ as  an $R$-module.
\end{proposition}
\begin{proof} 
Let $d:=\dim_{R}\mathrm{m}$. We first prove (for possibly zero $\mathrm{m}$) \begin{equation}\label{equation-<}\delta_{\Lambda}(\mathrm{m}\otimes_{E}P)\leq d+r.\end{equation}
 Indeed,  we may find  a system of parameters  (in the maximal ideal of $R$) for $\mathrm{m}$, say $a_1,...,a_{d}$, so that $\dim_{\F}\mathrm{m}/(a_1,...,a_d)$ is finite. Thus, $\mathrm{m}/(a_1,...,a_d)\otimes_{E} P$ has finite length in $\CG$ by Lemma \ref{lemma-tensor-P}(iii), and has canonical dimension $\leq r$ by Lemma \ref{lemma-pi-r-finitelength}.   As $\mathrm{m}/(a_1,...,a_d)\otimes_{E} P\cong (\mathrm{m}\otimes_{E}P)/(a_1,...,a_d)$,  
 we deduce \eqref{equation-<} from Proposition \ref{proposition-GKdim}.

Assuming $\mathrm{m}$ non-zero, we prove the inequality $\delta_{\Lambda}(\mathrm{m}\otimes_{E}P)\geq d+r$ by induction on $d$. By Lemma \ref{lemma-compare-GKdimension}, 
we are reduced to prove
\[\delta_{\Lambda}(\mathrm{m}\otimes_{E'}P)\geq\dim_{R}\mathrm{m}+r.\] 
The advantage to work with $E'$ is that $P$ is flat over $E'$.
 If $d=0$ (but $\mathrm{m}$ is non-zero), then $\mathrm{m}$ is finite dimensional over $\F$, and the result is a consequence of Lemma \ref{lemma-tensor-P}(iii) and Lemma \ref{lemma-compare-GKdimension}.
Assume $d\geq 1$ and the statement is true for all (non-zero) $E$-modules $\mathrm{m}'$ with $\dim_{R}\mathrm{m}'\leq d-1$. Since $\mathrm{m}$ is finitely generated over $R$, we may choose $x\in R$ such that $\dim_{R}(\mathrm{m}/x\mathrm{m})=d-1$; this implies that $x$ is not nilpotent on $\mathrm{m}$ and $\dim_{R}(x^k\mathrm{m}/x^{k+1}\mathrm{m})=d-1$  for any $k\geq 0$ (otherwise we would have $\dim_R(x^k\mathrm{m})\leq d-1$, hence $\dim_R\mathrm{m}\leq d-1$). Moreover,  since $x^k\mathrm{m}$ is non-zero, $x^k\mathrm{m}/x^{k+1}\mathrm{m}$ is also non-zero by Nakayama's lemma.  The inductive hypothesis then implies  \[\delta_{\Lambda}\big((x^k\mathrm{m}/x^{k+1}\mathrm{m})\otimes_{E'}P\big)\geq (d-1)+r.\]
Since $P$ is flat over $E'$, we have an isomorphism 
\[(x^k\mathrm{m}/x^{k+1}\mathrm{m})\otimes_{E'}P\cong x^k(\mathrm{m}\otimes_{E'}P)/x^{k+1}(\mathrm{m}\otimes_{E'}P),\] 
 and we conclude by Proposition \ref{proposition-GKdim}(ii) applied to $M=\mathrm{m}\otimes_{E'}P$.  
 \end{proof}
 
 \begin{corollary}
 Let $\mathrm{m}$ be a  finitely generated right $E$-module and assume that $\mathrm{m}\otimes_EP$ is coadmissible. Then  $\delta_{\Lambda}(\mathrm{m}\otimes_{R}P)=\delta_{\Lambda}(\mathrm{m}\otimes_{E}P)$.
 \end{corollary}
 \begin{proof}
 We  know that $\mathrm{m}\otimes_RP$ is coadmissible by Proposition \ref{proposition-EtoR-adm} and it is clear that $\delta_{\Lambda}(\mathrm{m}\otimes_RP)\geq\delta_{\Lambda}(\mathrm{m}\otimes_EP)$.  
 The first part of the proof of Proposition \ref{proposition-dim-equality} still works if we replace $E$ by $R$, showing 
 \[\delta_{\Lambda}(\mathrm{m}\otimes_RP)\leq \dim_R\mathrm{m}+r.\]
The result then follows from Proposition \ref{proposition-dim-equality}.
 \end{proof}
 
 \begin{corollary}
 Let $M$ be a coadmissible quotient of $P$. Then $\dim_R \mathrm{m}(M)\leq 2r$.
 \end{corollary}
 \begin{proof}
 By Proposition \ref{proposition-M-to-m(M)-adm}, $\mathrm{m}(M)\otimes_EP$ is also coadmissible, so  $\delta_{\Lambda}(\mathrm{m}(M)\otimes_EP)\leq \dim\Lambda=3r$. The result then follows from Proposition \ref{proposition-dim-equality}. 
 \end{proof}
 \begin{proposition}\label{proposition-BP-replaced}
 Let $M$ be a coadmissible quotient of $P$. There exists  a sequence $f_1,...,f_r\in \mathrm{Ann}_R(M)$ which is  both $R$-regular and $P$-regular, such that  $P/(f_1,...,f_r)$ is  finite free over $\Lambda$.
 \end{proposition} 
\begin{proof}
Since $M$ is coadmissible, so is $\mathrm{m}(M)\otimes_RP$ by Propositions \ref{proposition-M-to-m(M)-adm} and \ref{proposition-EtoR-adm}. Moreover, as in \emph{loc. cit.}, this implies that $\mathrm{m}(M)\otimes_RP(\widetilde{\sigma})$ 
is  finite dimensional over $\F$, where $\widetilde{\sigma}=\oplus_{\sigma}\sigma$ is the sum of all irreducible $\sigma$ such that $P(\sigma)\neq0$.  Since both $\mathrm{m}(M)$ and $P(\widetilde{\sigma})$ are finitely generated $R$-modules, we deduce that
$R/\big(\mathrm{Ann}_R(\mathrm{m}(M))+\mathrm{Ann}_R(P(\widetilde{\sigma}))\big)$
is an Artinian ring.
Writing $\fa=\mathrm{Ann}_R(\mathrm{m}(M))$  (which coincides with $\mathrm{Ann}_R(M)$ by Lemma \ref{lemma-annihilator} below),  then $R/\fa\otimes_RP(\widetilde{\sigma})$ is also finite dimensional over $\F$. Since $P(\widetilde{\sigma})$ is a Cohen-Macaulay $R$-module of Krull dimension $r$,  we can find a regular sequence $f_1,...,f_r$ in $\fa$ for $P(\widetilde{\sigma})$, see \cite[Thm. 2.1.2(b)]{BH-CM}.  

 As a generalisation of \cite{EP}, it is proved in Corollary \ref{corollary-EP} below that $P|_{\mathscr{K}}$ remains projective in $\mathfrak{C}(\mathscr{K})$. Applying repeatedly Proposition \ref{proposition-projective}, we obtain that  $f_1,...,f_r$ is  a regular sequence for $P$ and $P/(f_1,...,f_r)$ is again projective in $\mathfrak{C}(\mathscr{K})$. Moreover, $\Hom_{\mathscr{K}}^{\rm cont}(P/(f_1,...,f_r),\widetilde{\sigma}^{\vee})$ is finite dimensional over $\F$, so  $P/(f_1,...,f_r)$ is coadmissible by the choice of $\widetilde{\sigma}$. 
 
We are left to check that $f_1,...,f_r$ is an $R$-regular sequence. Since the sequence is $P$-regular, it suffices to prove that $R/(f_1,...,f_i)$ acts faithfully on $P/(f_1,...,f_i)$ for all $1\leq i\leq r$. For this, it suffices to prove that there is a natural isomorphism 
 \begin{equation}\label{equation-End}
 E/(f_1,...,f_i)\cong \End_{\CG}\big(P/(f_1,...,f_i)\big).\end{equation}
In fact, since $f_1,...,f_r$ lie in  the centre of $E$, the proof of \cite[Lem. 7.11]{HP} shows that   any morphism $P\ra P/(f_1,...,f_{i-1})$ factors through $P/(f_1,...,f_{i-1})\ra P/(f_1,...,f_{i-1})$. Since $P$ is projective, the exact sequence 
\[0\ra P/(f_1,...,f_{i-1})\overset{f_i}{\ra} P/(f_1,...,f_{i-1})\ra P/(f_1,...,f_i)\ra0\]
induces an isomorphism $\End_{\CG}\big(P/(f_1,...,f_{i-1})\big)/f_i\cong \End_{\CG}\big(P/(f_1,...,f_i)\big)$. 
An obvious induction then gives \eqref{equation-End}.  
\end{proof}

\subsubsection{Torsion $\Lambda$-modules}

 If $\eta\in\Rep_{\F}(T)$ is a smooth character, recall from \S\ref{subsection-PS} that $M_{\eta^{\vee}}=(\Ind_B^G\mathrm{Inj}_{T}\eta)^{\vee}$ and $E_{\eta^{\vee}}=\End_{\mathfrak{C}(G)}(M_{\eta^{\vee}})$.
 
 \begin{definition}\label{definition-P-r}
Let $\Sigma$ be a subset of $\{1,...,r\}$.  Given a smooth character $\eta_i\in\Rep_{\F}(T)$ for each $i\in \Sigma$ and an  irreducible $\pi_i\in\Rep_{\F}(G)$ for each $i\notin \Sigma$, we define $P(\underline{\eta},\underline{\pi},\Sigma)\in\CG$ by \[P(\underline{\eta},\underline{\pi},\Sigma):=\big(\widehat{\otimes}_{i\in \Sigma}M_{\eta_i^{\vee}}\big)\widehat{\otimes}\big(\widehat{\otimes}_{i\notin \Sigma}P_{\pi_i^{\vee}}).\]
\end{definition}

Denote by $E_{\Sigma}$ the endomorphism ring $\End_{\CG}(P(\underline{\eta},\underline{\pi},\Sigma))$; then \[E_{\Sigma}\cong \big(\widehat{\otimes}_{i\in \Sigma}E_{\eta_i^{\vee}}\big) \widehat{\otimes} \big(\widehat{\otimes}_{i\notin \Sigma}E_{\pi_i^{\vee}}\big)\]
is a quotient of $E=\widehat{\otimes}_{i=1}^rE_{\pi_i^{\vee}}$, where $\pi_i:=\rsoc_G\pi_{\eta_i}$ if $i\in\Sigma$. We employ  the rings constructed in Proposition \ref{proposition-flatness} once again by setting \begin{equation}\label{equation-define-R_N}R_{\Sigma}:=\big(\widehat{\otimes}_{i\in\Sigma}E_{\eta_i^{\vee}}\big)\widehat{\otimes}\big(\widehat{\otimes}_{i\notin\Sigma, i\notin\mathrm{IV}}R_{\pi_i^{\vee}}\big)\widehat{\otimes}\big(\widehat{\otimes}_{i\notin\Sigma,i\in \mathrm{IV}}R_{\pi_i^{\vee}}'\big),\end{equation}
where the notation is as in \S\ref{subsection-dimformula}. 
Then $R_{\Sigma}$ is a (commutative) regular local ring of Krull dimension $3r-|\Sigma|$. Clearly,  $R_{\Sigma}$ is contained in (the centre of) $E_{\Sigma}$ and $E_{\Sigma}$ is finite free over $R_{\Sigma}$.

Let $\sigma\in\Rep_{\F}(\mathscr{K})$ be  an  irreducible representation with  $P(\sigma)\neq0$. Corollary \ref{corollary-P=Meta} implies an isomorphism
\[ P(\sigma)\cong \Hom_{\mathscr{K}}^{\rm cont}(P(\underline{\eta},\underline{\pi},\Sigma),\sigma^{\vee})^{\vee}, \]
so $P(\sigma)$ can be viewed as a module over $E_{\Sigma}$ and  $R_{\Sigma}$.

 \begin{lemma}\label{lemma-annihilator-PS}
With the above notation, $\mathrm{Ann}_{R_{\Sigma}}(P(\sigma))$ can be generated by $2r-|\Sigma|$ elements.
\end{lemma}
\begin{proof}
To simplify and uniform the notation, we write $R_i$ for the ring at index $i$ in \eqref{equation-define-R_N} so that $R_{\Sigma}=\widehat{\otimes}_{i=1}^rR_i$; similarly write $E_{\Sigma}=\widehat{\otimes}_{i=1}^rE_i$.  

By Lemma \ref{lemma-Serre-r}(ii) together with Corollary \ref{corollary-P=Meta},  $P(\sigma)$ is a cyclic $E_{\Sigma}$-module such that \[E_{\Sigma}/\mathrm{Ann}_{E_{\Sigma}}(P(\sigma))\cong\widehat{\otimes}_{i=1}^r\F[[S]]\cong\F[[S_1,...,S_r]].\] For each $i$, the composite morphism $R_i\ra E_i\ra \F[[S]]$ is either  surjective or equal to $\F[[S^2]]$; this is an easy  check if $i\in\Sigma$ or $i\notin \mathrm{III}$, and follows from   Proposition \ref{proposition-III-image}(ii) if $i\in\mathrm{III}$ (and  $i\notin\Sigma$). In all, the image of $R_\Sigma\hookrightarrow E_{\Sigma}\twoheadrightarrow \widehat{\otimes}_{i=1}^r\F[[S]]$ is  a regular local ring of Krull dimension $r$. The result then follows from  \cite[Thm. 21.2(ii)]{Mat}.
\end{proof}

\begin{lemma}\label{lemma-general-PS-torsion}
Let $M$ be a coadmissible quotient of $P(\underline{\eta},\underline{\pi},\Sigma)$ and assume $\Sigma$ is non-empty. Then $M$ is a torsion $\Lambda$-module. In fact, $\delta_{\Lambda}(M)\leq 3r-|\Sigma|$.
\end{lemma}
\begin{proof}
Note that the action of $E$ on $\mathrm{m}(M)$ factors through $E_{\Sigma}$ by \cite[Prop. 7.1(iii)]{Pa13}, hence $\mathrm{m}(M)$ can be viewed as an $R_{\Sigma}$-module.  Since $M$ is coadmissible, $\mathrm{m}(M)\otimes_{R_{\Sigma}}P(\sigma)$ is finite dimensional, where $\sigma$ is as before. As a consequence, \[R_{\Sigma}/\big(\mathrm{Ann}_{R_{\Sigma}}(\mathrm{m}(M))+\mathrm{Ann}_{R_{\Sigma}}(P(\sigma))\big)\] is an Artinian ring. As in Proposition  \ref{proposition-BP-replaced}, we can find a sequence $f_1,...,f_r\in \mathrm{Ann}_{R_{\Sigma}}(\mathrm{m}(M))$ which  is regular for $P(\sigma)$.  
By Lemma \ref{lemma-annihilator-PS}, $\mathrm{Ann}_{R_{\Sigma}}(P(\sigma))$ is generated by $2r-|\Sigma|$ elements, say $g_1,...,g_{2r-|\Sigma|}$. Since $R_{\Sigma}$ is Cohen-Macaulay of Krull dimension $3r-|\Sigma|$, the sequence $f_1,...,f_r,g_1,...,g_{2r-|\Sigma|}$  is $R_{\Sigma}$-regular. As a consequence, 
$\dim R_{\Sigma}/(f_1,...,f_r)=2r-|\Sigma|$.  But, $\mathrm{m}(M)$ is annihilated by  $(f_1,...,f_r)$, hence
\[\dim \mathrm{m}(M)\leq \dim R_{\Sigma}/(f_1,...,f_r)=2r-|\Sigma|.\]
We conclude by Proposition \ref{proposition-dim-equality}.  
\end{proof}

  \begin{proposition}\label{proposition-M-to-m(M)-torsion}
Let $M\in\mathfrak{C}(\mathscr{G})$ be a coadmissible quotient of $P$. 
If $M$ is a torsion $\Lambda$-module, then so is $\mathrm{m}(M)\otimes_EP$.
\end{proposition}
\begin{proof}
As in the proof of Proposition \ref{proposition-M-to-m(M)}(ii), it is enough to show the following:\medskip

\textbf{Claim}: For $1\leq i\leq r$ and  $\pi_i'\in \mathfrak{B}_i$ distinct with $\pi_i$, let $Q_i'$  be the maximal quotient of $P_{\pi_i'^{\vee}}$ none of whose irreducible subquotients is isomorphic to   $\pi_i^{\vee}$. If $M$ is  a coadmissible quotient of \[Q_i'\widehat{\otimes}\big(\widehat{\otimes}_{j\neq i}P_{\pi_j'^{\vee}}\big),\]
then $M$ is a torsion $\Lambda$-module. 
\medskip

 As in Proposition \ref{proposition-M-to-m(M)}(ii),  we know that $\mathfrak{B}_i$ can only be of type (II) or (IV), and we may assume that $M$ is a  coadmissible quotient of $M_{\eta_i^{\vee}}\widehat{\otimes}\big(\widehat{\otimes}_{j\neq i}P_{\pi_j'^{\vee}}\big)$. The claim is then a special case of  Lemma \ref{lemma-general-PS-torsion}.  
\end{proof}

 \begin{corollary}\label{corollary-Krull-m}
Let $M\in \CG$ be a coadmissible quotient of $P$. 

(i) $\mathrm{m}(M)$ has Krull dimension $\leq 2r$. 

(ii)  $M$ is a torsion $\Lambda$-module if and only if $\dim_R\mathrm{m}(M)\leq 2r-1$. 
\end{corollary}
\begin{proof}
(i) Since $M$ is coadmissible, so is $\mathrm{m}(M)\otimes_EP$ by Proposition \ref{proposition-M-to-m(M)-adm}. Since we always have $\delta_{\Lambda}(M)\leq \delta_{\Lambda}(\mathrm{m}(M)\otimes_EP)\leq 3r$, the result follows from  Proposition \ref{proposition-dim-equality}. 

(ii)  $M$ is torsion if and only if $\mathrm{m}(M)\otimes_EP$ is torsion by Proposition \ref{proposition-M-to-m(M)-torsion}, if and only if $\delta_{\Lambda}(\mathrm{m}(M)\otimes_EP)\leq 3r-1$, if and only if $\dim_R\mathrm{m}(M)\leq 2r-1$ by Proposition \ref{proposition-dim-equality}.   
\end{proof}

\subsubsection{Breuil-Pa\v{s}k\=unas construction}
We generalize the construction of Breuil-Pa\v{s}k\={u}nas \cite[\S9]{BP} to our setting.

\begin{proposition}\label{proposition-BP-adapted}
Let $M\in\mathfrak{C}(\mathscr{G})$ be  coadmissible and $\widetilde{\sigma}$ be the $\mathscr{K}$-cosocle of $M$.   Then there exists a surjection in $\CG$:
\[\overline{P}\twoheadrightarrow M\]
where $\overline{P}|_{\mathscr{K}}$ is isomorphic to a projective envelope of $\widetilde{\sigma}$ (with central character).  In particular, $\overline{P}$ is finite free as a $\Lambda$-module.
 
\end{proposition} 
\begin{proof}
The proof is given in Appendix, Theorem \ref{theorem-BP-appendix}.
\end{proof}  

The following result generalizes \cite[Cor. 3.8]{EP}.

\begin{corollary}\label{corollary-EP}
If $\widetilde{P}\in \mathfrak{C}(\mathscr{G})$ is projective, then $\widetilde{P}$ is also projective in $\mathfrak{C}(\mathscr{K})$.  
\end{corollary}
\begin{proof}
The proof is identical to that of \cite[Cor. 3.8]{EP}, using Proposition \ref{proposition-BP-adapted} in place of \cite[Thm. 3.4]{EP}. 
\end{proof}

\subsubsection{Finite free modules}
If $M$ is a quotient of $P=P_{\pi^{\vee}}$, then $R$ acts (from left) on $M$ and (from right) on $\mathrm{m}(M):=\Hom_{\mathfrak{C}(\mathscr{G})}(P,M)$.  We make explicit these actions. Let $\phi\in R$ and view it as an endomorphism $\phi: P\ra P$. It induces an endomorphism $\overline{\phi}:M\ra M$ (because $R$ is contained in the Bernstein center of $\mathfrak{C}(\mathscr{G})^{\mathfrak{B}}$), and for any $\theta\in \mathrm{m}(M)$ the following diagram is commutative:
\begin{equation}\label{equation-action-commute}\xymatrix{P\ar^{\phi}[r]\ar_{\theta}[d]&P\ar^{\theta}[d]\\
M\ar^{\overline{\phi}}[r]&M.}\end{equation}
The action of $R$ on $M$ is given by $(\phi,m)\mapsto \overline{\phi}(m)$, and the action on $\mathrm{m}(M)$ is given by
\begin{equation}\label{equation-action-m(M)}(\theta,\phi)\mapsto \theta\circ\phi=\bar{\phi}\circ\theta.\end{equation}
We have the following simple lemma. 
\begin{lemma}\label{lemma-annihilator}
We have $\mathrm{Ann}_R(M)=\mathrm{Ann}_R(\mathrm{m}(M))$.
\end{lemma}
\begin{proof} 
Given $\phi\in R$, we have the following equivalences
\[\begin{array}{rll} \phi\in \mathrm{Ann}_R(M) &\Longleftrightarrow&\overline{\phi}=0\\
&\overset{(*)}{\Longleftrightarrow}& \overline{\phi}\circ\theta=0, \ \  \forall\ \theta\in\mathrm{m}(M)\\
&\overset{\eqref{equation-action-commute}}{\Longleftrightarrow}& \theta\circ \phi=0, \ \ \forall\ \theta\in\mathrm{m}(M) \\&\overset{\eqref{equation-action-m(M)}}{\Longleftrightarrow}& \phi\in \mathrm{Ann}_R(\mathrm{m}(M)),\end{array}\]
where ($*$) holds  because there exists  at least one $\theta$ which is surjective, e.g. the natural quotient map $P\twoheadrightarrow M$.
\end{proof}

Let $\overline{P}$ be a quotient of $P$ which is finite free as a $\Lambda$-module. Set   \[\overline{R}:=R/\mathrm{Ann}_R(\mathrm{m}(\overline{P}))=R/\Ann_R(\overline{P}).\] 
 
\begin{proposition}\label{proposition-Krull-Rbar}
(i) $\overline{R}$ has Krull dimension $2r$. There exists a subring $A\subset \overline{R}$ which is formally smooth of dimension $2r$ such that $\overline{R}$ is finite over $A$. 

(ii) $\overline{P}$ is flat over $A$. 
\end{proposition}
\begin{proof}
(i) The first assertion is a direct consequence of Corollary \ref{corollary-Krull-m}.
 The second one is proved as in \cite[Cor. 4.2]{Pa18} by applying Cohen's structure theorem for Noetherian complete local rings (see \cite[Thm. 29.4(iii)]{Mat}).

 
(ii)  Because $\overline{P}$ is  Cohen-Macaulay (being free over $\Lambda$) with $\delta_{\Lambda}(\overline{P})=3r$ and 
\[\delta_{\Lambda}(\F\otimes_{A}\overline{P})=\delta_{\Lambda}(\F\otimes_{\overline{R}}\overline{P})=r,\]  
the result follows from the ``miracle flatness'' criterion, see \cite[Prop. A.30]{GN}.
 \end{proof}

\subsection{A lemma}  \label{subsection-keylemma}
In this subsection, we prove a lemma which can be viewed as an analogue of Proposition \ref{proposition-GKdim}. These two results will allow us to relate the canonical dimension of a coadmissible module $M\in\CG$ and the $\F$-dimension of $\Tn$-coinvariants of $M$. It is where the quantity $\kappa(\bn)=\max_i\{n_i\}$ comes.
\begin{lemma}\label{lemma-keystep-r}
Let $M$ be a finitely generated $\Lambda$-module and $\phi\in\End_{\Lambda}(M)$. Assume that $\bigcap_{k\geq 1}\phi^k(M)=0$.
\begin{enumerate}
\item[(i)] If $\phi$ is nilpotent, then $\dim_{\F}M_{\mathscr{T}_1(p^{\bf n})}\sim \dim_{\F}\big(M/\phi(M)\big)_{\mathscr{T}_1(p^{\bf n})}$.
\item[(ii)] If $\phi$ is not nilpotent, then for some sufficiently large $k_0$ we have 
\[\dim_{\F}M_{\mathscr{T}_1(p^{\bf n})}\ll  \dim_{\F}\big(M/\phi(M)\big)_{\mathscr{T}_1(p^{\bf n})}+p^{\kappa({\bf n})} \dim_{\F}\big(\phi^{k_0}(M)/\phi^{k_0+1}(M)\big)_{\mathscr{T}_1(p^{\bf n})}\] 
where $\kappa(\bn):=\max_i\{n_i\}$. 
\end{enumerate}
In any case, we have $\dim_{\F}M_{\mathscr{T}_1(p^{\bf n})}\ll p^{\kappa({\bf n})}\dim_{\F}\big(M/\phi(M)\big)_{\mathscr{T}_1(p^{\bf n})}$.
\end{lemma}


\begin{proof}
(i) We trivially have $\dim_{\F}M_{\mathscr{T}_1(p^{\bf n})}\geq \dim_{\F}\big(M/\phi(M)\big)_{\mathscr{T}_1(p^{\bf n})}$. 
If $\phi$ is nilpotent, then $M$ admits a finite filtration by $\phi^k(M)$, for $k\leq k_0$ where $k_0\gg1 $ is such that $\phi^{k_0}=0$. For any $k\geq 1$, $\phi$ induces a surjective morphism 
$M/\phi(M)\twoheadrightarrow \phi^{k}(M)/\phi^{k+1}(M)$, hence $\dim_{\F}M_{\mathscr{T}_1(p^{\bf n})}\leq k_0\cdot \dim_{\F}\big(M/\phi(M)\big)_{\mathscr{T}_1(p^{\bf n})}$, giving the result.

(ii) By Lemma \ref{lemma-phi}, there exists $k_0\gg0$ such that $\phi: \phi^{k_0}(M)\ra \phi^{k_0}(M)$ is injective. Using the short exact sequence $0\ra \phi^{k_0}(M)\ra M\ra M/\phi^{k_0}(M)\ra0$ and applying (i) to $M/\phi^{k_0}(M)$, we are reduced to prove
\[\dim_{\F}\phi^{k_0}(M)_{\mathscr{T}_1(p^{\mathbf{n}})}\ll p^{\kappa(\mathbf{n})}\dim_{\F}\big(\phi^{k_0}(M)/\phi^{k_0+1}(M)\big)_{\mathscr{T}_1(p^{\mathbf{n}})}. \]
That is, by replacing $M$ by $\phi^{k_0}(M)$, we may assume $\phi$ is injective in the rest of the proof.
 
Set $Q:=M/\phi(M)$ so that we have a short exact sequence $0\ra M\overset{\phi}{\ra}M\ra Q\ra0$.  Let $J$ denote the maximal ideal of $\Lambda$.  By Remark \ref{remark-condition-phi},   we may choose $k_1\gg 1$ such that $\phi^{k_1}(M)\subset JM$. Replacing $\phi$ by $\phi^{k_1}$ and $Q$ by $M/\phi^{k_1}(M)$, we may assume $\phi(M)\subset JM$. Since $\phi$ is a $\Lambda$-morphism,  we obtain inductively
\begin{equation}\label{equation-phi-k}\phi^k(J^sM)\subset J^{k+s}M,\ \ \forall k,s\geq 1.\end{equation}
Letting $Q_k:=M/\phi^k(M)$, the short exact sequence  
$0\ra M\overset{\phi^{k}}{\ra} M\ra Q_{k}\ra0$ then  
 induces by modulo $J^{k+1}$: 
\[M/JM\overset{\phi^k}{\ra} M/J^{k+1}M\ra Q_k/J^{k+1}Q_k\ra0,\]
where we have used the fact  $\phi^k(JM)\subset J^{k+1}M$ by \eqref{equation-phi-k}.
If $I$ is a right ideal of $\Lambda$  containing $J^{k+1}$ (and contained in $J$) then, by tensoring  the above sequence with   $(\Lambda/I)\otimes_{\Lambda}$, we obtain an exact sequence of $\F$-vector spaces:
\[M/JM\ra M/IM\ra Q_k/IQ_k\ra0.\]
Since $Q_k$ is a successive extension of $Q$ ($k$ times), we obtain the following inequality with $c_0:=\dim_{\F}M/JM$:
\begin{equation}\label{equation-key-inequality}\dim_{\F}M/IM\leq \dim_{\F}Q_k/IQ_k+c_0\leq k\dim_{\F}Q/IQ+c_0.\end{equation}

We specialize the above inequality to our situation.  Let $I_{\bn}$ denote the \emph{right} ideal of $\Lambda$ generated by the maximal ideal of $\F[[\mathscr{T}_1(p^{\bn})/\mathscr{Z}_1]]$. Then Lemma \ref{lemma-Iwasawa-ideal} below shows that  $J^{3rp^{\kappa(\bn)-1}}$ is contained in $I_{\bn}$ because $I_{\bn}$ clearly  contains $J_{\bn}\Lambda$. 
Applying \eqref{equation-key-inequality} to $I=I_n$, we obtain 
\[\dim_{\F}M_{\mathscr{T}_1(p^{\bn})}=\dim_{\F}M/I_{\bn}M\leq (3rp^{\kappa(\bn)-1}-1)\cdot\dim_{\F}Q/I_{\bn}Q+c_0,\]
giving the result.
\end{proof}

\begin{lemma}\label{lemma-Iwasawa-ideal}
Let $J$ be the maximal ideal of $\Lambda$ and $J_{\bn}$ be the maximal ideal of $\F[[\mathscr{K}_{\bn}/(\mathscr{K}_{\bn}\cap \mathscr{Z}_1)]]$ (viewed as a sub-algebra of $\Lambda$). Then the right ideal $J_{\bn}\Lambda$ is a two-sided ideal of $\Lambda$ and satisfies $J^{3rp^{\kappa(\bn)-1}}\subset J_{\bn}\Lambda$.
\end{lemma}
\begin{proof}
The first assertion  follows from that $\mathscr{K}_{\bn}$ is a normal subgroup of $\mathscr{K}_1$.

We are left to prove the inclusion $J^{3rp^{\kappa(\bn)-1}}\subset J_{\bn}\Lambda$. We first consider the case $r=1$. Then $\Lambda=\F[[K_1/Z_1]]$ is topologically generated by three elements, say $z_1, z_2, z_3$, such that every element of $\Lambda$ can be uniquely expressed as a sum over multi-indices $\alpha=(\alpha_1,\alpha_2,\alpha_3)\in\N^3$:
\[x=\sum_{\alpha}\lambda_{\alpha}z^{\alpha},\ \ \ z^{\alpha}=z_1^{\alpha_1}z_2^{\alpha_2}z_3^{\alpha_3}.\]
Moreover, $z^{\alpha}z^{\beta}=z^{\alpha+\beta}$ up to terms of degree $>|\alpha|+|\beta|$, which we refer to as the almost commutativity of $\Lambda$; see \cite[Thm. 10]{Mar-Annals}. 
The ideal $J$ is simply topologically spanned by the set of monomials $z^{\alpha}$ with $|\alpha|>0$. Similarly, $\F[[K_n/(K_n\cap Z_1)]]$ is topologically generated by $z_1^{p^{n-1}}, z_2^{p^{n-1}},z_3^{p^{n-1}}$, and $J_{n}$ is topologically spanned by the set of monomials $z^{p^{n-1}\cdot\alpha}$ with $|\alpha|>0$. 
Hence, the ideal $J_n\Lambda$ is topologically spanned by monomials $z^{\alpha}$  with at least one of $\alpha_j$ greater than  or equal to $p^{n-1}$ (cf. the proof of \cite[Lem. 12]{Mar-Annals}).  By the almost commutativity, we deduce that $J^{3p^{n-1}}\subset J_n\Lambda$.    
For general $r$, the proof is identical noting that $J$ is  topologically generated by $3r$ elements  $\{z_{i,j}: 1\leq i\leq r, 1\leq j\leq 3\}$, and $J_{\bn}$ is topologically generated by $z_{i,j}^{p^{n_i-1}}$, hence 
\[J^{\sum_{i=1}^r3p^{n_i-1}}\subset J_{\bn}\Lambda\]
which in particular implies the result.
\end{proof} 

\begin{remark}
In the proof of Lemma \ref{lemma-keystep-r}, it is crucial that we are working with $T_1(p^n)$ instead of $K_1(p^{2n})$ (this group is defined in \eqref{equation-define-K1(pn)} below), although they are (up to finite order) conjugate to each other in $\GL_2(\Q_p)$. We have learnt this trick of ``averaging'' from \cite{Mar-Annals} (used in a different manner  there).
\end{remark}

\subsection{Main result}\label{subsection-results}
In this subsection, we prove the following theorem.

\begin{theorem}\label{theorem-higher-r}
Let $M\in\mathfrak{C}^{\rm fg,tor}(\mathscr{G})$. Then for any $i\geq 0$  
\begin{equation}\label{equation-main-dim}\dim_{\F}H_i(\mathscr{T}_{1}(p^{\bn})/\mathscr{Z}_1,M)\ll  \kappa(\bn)^rp^{(2r-1)\kappa(\bn)},\end{equation}
where $\kappa(\bn):=\max_i\{n_i\}$.
\end{theorem}

\begin{lemma}\label{lemma-firstreduction}
In Theorem \ref{theorem-higher-r}, we may assume 
that $M$ has an irreducible $\mathscr{G}$-cosocle (hence indecomposable). 
\end{lemma}
\begin{proof}
Let $S$ be the $\mathscr{G}$-cosocle of $M$. Since $M$ is coadmissible, $S$ decomposes as a finite direct sum $\oplus_{i=1}^sS_i$ with each $S_i$  irreducible. For each $i$, let $P_{S_i}$ be a  projective envelope of $S_i$ in $\CG$. 
The projection $M\twoheadrightarrow S_1$ extends to  a $\CG$-equivariant morphism $\alpha_1:P_{S_1}\ra M$. It is clear that $\mathrm{coker}(\alpha_1)$ has $\mathscr{G}$-cosocle isomorphic to $\oplus_{i=2}^sS_i$ and $\im(\alpha_1)$ has $\mathscr{G}$-cosocle $S_1$. Continuing this with $\mathrm{coker}(\alpha_1)$, we get a finite filtration of $M$ such that each graded piece, say $\mathrm{gr}^i(M)$, has an irreducible $\mathscr{G}$-cosocle. Since $M$ is  torsion as a $\Lambda$-module if and only if each $\mathrm{gr}^i(M)$ is, we are reduced to prove \eqref{equation-main-dim} for all $\mathrm{gr}^i(M)$.
\end{proof}

Let $M$ be a quotient of $P_{\pi^{\vee}}$ for some  irreducible $\pi\in\CG$. Let $P=P_{\pi^{\vee}}$, $E$, $R$ be as before.

\begin{definition}\label{definition-elementary}
We say $M\in \CGtor$ is  \emph{elementary}\footnote{The notation is motivated by the corresponding one in commutative ring theory; see \cite[\S11.6]{Iv}.} if there exists a short exact sequence in $\CG$:
\begin{equation}\label{equation-define-elementary}0\ra \overline{P}\overset{a}{\ra} \overline{P}\ra M\ra0,\end{equation}
where  $\overline{P}\in \CG$ is a quotient of $P$ and is  finite free as a $\Lambda$-module, and $a\in \overline{R}:=R/\mathrm{Ann}(\mathrm{m}(\overline{P}))$. 
\end{definition}

\begin{lemma}\label{lemma-elementary}
Theorem \ref{theorem-higher-r} is true if $M$ is an elementary quotient of $P$.
\end{lemma}
\begin{proof}
Let $\overline{P}$, $a\in \overline{R}$ be as in \eqref{equation-define-elementary}. 
Since $\overline{P}$ is a free $\Lambda$-module, taking homology of \eqref{equation-define-elementary}  we obtain
\[\dim_{\F}H_0(\mathscr{T}_{1}(p^{\bn})/\mathscr{Z}_1,M)=\dim_{\F}H_1(\mathscr{T}_{1}(p^{\bn})/\mathscr{Z}_1,M)\]
\[H_i(\mathscr{T}_{1}(p^{\bn})/\mathscr{Z}_1,M)=0,\ \ \forall i\geq 2.\] 
Hence, it suffices to prove \eqref{equation-main-dim}  when $i=0$. Since $a$ is $\overline{P}$-regular and $\overline{R}$ acts faithfully on $\overline{P}$ (by Lemma \ref{lemma-annihilator}), $a$ is also $\overline{R}$-regular. Since $\overline{R}$ has Krull dimension $2r$ by Proposition \ref{proposition-Krull-Rbar}, we may extend $a$ to a system of parameters of $\overline{R}$, say $a_1=a,a_2,...,a_{2r}$. Then $\overline{R}/(a_1,...,a_{2r})$ is finite dimensional over $\F$ and hence $M/(a_2,...,a_{2r})$ has finite length in $\CG$. By Lemma \ref{lemma-pi-r-finitelength}, we deduce   that 
\[\dim_{\F}H_0(\mathscr{T}_{1}(p^{\bn})/\mathscr{Z}_1,M/(a_2,...,a_{2r}))\ll \kappa(\mathbf{n})^r,\]
and we conclude by repeatedly applying  Lemma \ref{lemma-keystep-r} to $M/(a_2,...,a_{2r-1})$, $\cdots$, $M$.
\end{proof}
\begin{remark}
Let $M$ be an elementary quotient of $P$ and let $\overline{P}, a\in\overline{R}$ be as in \eqref{equation-define-elementary}. Moreover, we assume $\mathbf{n}=(n,...,n)$ is parallel. Since $\overline{P}$ is finite free over $\Lambda$, we have \[\dim_{\F}\overline{P}_{\mathscr{T}_1(p^{\mathbf{n}})}\sim [\mathscr{K}_1:\mathscr{T}_1(p^{\mathbf{n}})]\sim p^{2rn}. \] 
From Lemma \ref{lemma-keystep-r} we deduce that
\[\dim_{\F}M_{\mathscr{T}_1(p^{\mathbf{n}})}\gg p^{(2r-1)n}.\]
This shows that the upper bound \eqref{equation-main-dim} is almost optimal.
\end{remark} 

\begin{proof}[Proof of Theorem \ref{theorem-higher-r}]
For each $i\geq 0$, denote by $(\mathrm{C}_i)$ the inequality: 
\[\dim_{\F}H_i(\mathscr{T}_1(p^{\mathbf{n}})/\mathscr{Z}_1,M)\ll \kappa(\mathbf{n})^rp^{(2r-1)\kappa(\mathbf{n})}. \]
We will prove $(\mathrm{C}_i)$ for any $M\in \CGtor$ by induction on $i$. 

First  prove $(\mathrm{C}_0)$. By Lemma \ref{lemma-firstreduction}, we may ssume  $M\in \CGtor$ is a quotient of $P=P_{\pi^{\vee}}$ for some irreducible $\pi\in \CG$.  By Proposition \ref{proposition-BP-adapted}, we may find $\overline{P}\in\CG$, which is finite free as a $\Lambda$-module and has the same $\mathscr{K}$-cosocle as $M$, such that $M$ is a quotient of $\overline{P}$.  In particular, $\overline{P}$ has the same $\mathscr{G}$-cosocle as $M$ and we may view $\overline{P}$ as a quotient of $P$.\footnote{Alternatively, we may apply Proposition \ref{proposition-BP-replaced} to obtain such an object $\overline{P}$ (but without cosocle condition).} 
Set $\overline{R}=R/\mathrm{Ann}(\mathrm{m}(\overline{P}))$. 
By Proposition \ref{proposition-Krull-Rbar},  $\overline{R}$ has Krull dimension $2r$ and is finite over a formally smooth subring $A\cong \F[[y_1,...,y_{2r}]]$. We view $\mathrm{m}(M)$ as an $A$-module. Since $\mathrm{m}(M)$ has Krull dimension $<2r$ by Corollary \ref{corollary-Krull-m}, there exists a non-zero $a\in A$ which annihilates $\mathrm{m}(M)$. In particular, we obtain a surjection 
\begin{equation}
\label{equation-surj-to-M}\overline{P}/a\overline{P}\twoheadrightarrow M.\end{equation}
 Proposition \ref{proposition-Krull-Rbar}(ii) shows that $\overline{P}$ is flat over $A$, hence $a$ is $\overline{P}$-regular and $\overline{P}/a\overline{P}$ is an elementary module.  By Lemma \ref{lemma-elementary}, $(\mathrm{C}_0)$ holds for $\overline{P}/a\overline{P}$, hence also holds for $M$.
  
Now we assume $(\mathrm{C}_i)$ holds for \emph{any} object in $\CGtor$, and prove $(\mathrm{C}_{i+1})$ for (the fixed) $M$.   
Let $M'$ be the  kernel of \eqref{equation-surj-to-M}. Taking homology we obtain an exact sequence
\[H_{i+1}(\mathscr{T}_1(p^{\mathbf{n}})/\mathscr{Z}_1,\overline{P}/a\overline{P})\ra H_{i+1}(\mathscr{T}_1(p^{\mathbf{n}})/\mathscr{Z}_1,M)\ra H_i(\mathscr{T}_1(p^{\mathbf{n}})/\mathscr{Z}_1,M').\]
Since $M'\in\CGtor$, $(\mathrm{C}_{i})$ holds for $M'$ by inductive hypothesis and $(\mathrm{C}_{i+1})$ holds for $\overline{P}/a\overline{P}$ by Lemma \ref{lemma-elementary},  we obtain that
 $(\mathrm{C}_{i+1})$ holds for $M$. 
\end{proof}

\subsection{Change of groups}\label{subsection-changegroups}
We keep the notation in the previous subsection. For $n\geq 1$, let 
\begin{equation}\label{equation-define-K1(pn)}K_1(p^n):=K_1\cap K_0(p^n)=\matr{1+p\Z_p}{p\Z_p}{p^n\Z_p}{1+p\Z_p}.\end{equation} 
These groups are closely related  to  the  $T_1(p^n)$ in the sense that letting $D=\smatr{1}00{p^{\lfloor n/2\rfloor}}$ and $n'=\lfloor n/2\rfloor+1$, we have 
\begin{equation}\label{equation-conjugate}D^{-1}K_1(p^n)D< T_1(p^{n'}),\ \ |T_1(p^{n'}):D^{-1}K_1(p^n)D|\leq p,\ \ |n'-n/2|\leq 1.\end{equation}

On the other hand, using (essentially) the fact  that the restriction of $\F[K_1/K_1(p^n)]$ to $\smatr{1}0{p\Z_p}1$ is uniserial, which follows from the Iwahori decomposition for $K_1$ and the fact that $\smatr{1}0{p\Z_p}1$ is pro-cyclic, Marshall proved the following interesting result (\cite[Cor. 14]{Mar-Annals}).
\begin{lemma}\label{lemma-Marshall-filtration}
Let $L\subset \F[K_1/K_1(p^n)]$
be a submodule of dimension $d$, and let the base $p$ expansion of $d$ be written as
$d=\sum_{i=1}^lp^{\alpha(i)}$, 
where $\{\alpha(i)\}$ is a non-increasing sequence of non-negative integers. Then there exists a filtration $0=L_0\subset \cdots \subset L_{l}=L$ of $L$ by submodules $L_i$ such that $L_i/L_{i-1}\cong \F[K_1/K_1(p^{\alpha(i)+1})]$. 
\end{lemma}
If $\bn=(n_1,...,n_r)\in(\Z_{\geq 1})^r$, let 
\[\mathscr{K}_1(p^{\bn})=\prod_{i=1}^rK_1(p^{n_i}).\]

\begin{theorem}\label{theorem-finalform}
Let $M\in\CGtor$ and let $L$ be any sub-representation of $\F[\mathscr{K}_1/\mathscr{K}_1(p^{\bf n})]$ which factorizes as $\otimes_{i=1}^r L_i$ with  $L_i\subset \F[K_1/K_1(p^{n_i})]$. 
Then for $i\geq 0$ we have
\[\dim_{\F} H_i(\mathscr{K}_1/\mathscr{Z}_1,M\otimes L)\ll \kappa(\bn)^{2r}p^{(r-\frac{1}{2})\kappa(\bn)}.\] 
In particular, if $\bn=(n,...,n)$ is parallel, then 
\[\dim_{\F} H_i(\mathscr{K}_1/\mathscr{Z}_1,M\otimes L)\ll n^{2r}p^{(r-\frac{1}{2})n}.\]
\end{theorem}
\begin{proof}
The proof goes as  that of \cite[Lem. 19]{Mar-Annals}. For convenience of the reader, we briefly explain it. First, if $L=\F[\mathscr{K}_1/\mathscr{K}_1(p^{\bf m})]$ for some ${\bf m}\in(\Z_{\geq 1})^r$, we apply Shapiro's lemma to obtain
\begin{equation}\label{equation-global-H_i}H_i(\mathscr{K}_1/\mathscr{Z}_1, M\otimes L)\cong H_i(\mathscr{K}_1(p^{\bf m})/\mathscr{Z}_1, M).\end{equation}
Using a suitable diagonal element of $\mathscr{G}$, precisely 
\[D=\bigg(\matr{1}00{p^{{\lfloor m_1\rfloor}/2}},\dots, \matr{1}00{p^{\lfloor m_r/2\rfloor}}\bigg)\]
we obtain by \eqref{equation-conjugate} that for some ${\bf m}'$: 
\[D^{-1}\mathscr{K}_1(p^{\bf m})D\leq \mathscr{T}_1(p^{{\bf m}'}),\ \  \ [\mathscr{T}_1(p^{{\bf m}'}):D^{-1}\mathscr{K}_1(p^{\bf m})D]\leq p^{r},\ \ \ |m_i'-m_i/2|\leq 1.\] 
Since $M$ carries a compatible action of $\mathscr{G}$, we have natural isomorphisms
\[H_i(\mathscr{K}_1(p^{\bf m})/\mathscr{Z}_1,M)\cong H_i(D^{-1}\mathscr{K}_1(p^{\bf m})D/\mathscr{Z}_1,M).\]
Hence we deduce from \cite[Lem. 20]{Mar-Annals} that
\[\dim_{\F}H_i(\mathscr{K}_1(p^{\bf m})/\mathscr{Z}_1,M)\leq p^r\dim_{\F}H_i(\mathscr{T}_1(p^{{\bf m}'})/\mathscr{Z}_1,M)\]
and the result follows from Theorem \ref{theorem-higher-r}.
 
For general $L$, Lemma \ref{lemma-Marshall-filtration} provides a finite filtration of $L$
\[0=L_0\subset L_1\subset \cdots   \subset L \]
such that every quotient $L_i/L_{i-1}$ is isomorphic to $\F[\mathscr{K}_{1}/\mathscr{K}_1(p^{\bf m})]$ for some ${\bf m}\leq {\bf n}$ and each isomorphism class of quotient occurs at most  $p^r$ times. We then deduce from the first case that 
\[\begin{array}{rll}
\dim_{\F}H_i(\mathscr{K}_1/\mathscr{Z}_1, M\otimes L)&\leq &p^r\sum\limits_{{\bf m}\leq {\bf n}}\dim_{\F}H_i(\mathscr{K}_1(p^{\bf m})/\mathscr{Z}_1, M)\\
&\ll & \sum\limits_{{\bf m}\leq{\bn}} \kappa({\bf m})^r p^{(r-\frac{1}{2})\kappa({\bf m})}\\
&\ll &   \kappa(\bn)^{2r}p^{(r-\frac{1}{2})\kappa({\bf n})}.
\end{array}\]
Here we have used the fact that the cardinality of the set $\{{\bf m}: {\bf m}\leq \bn\}$ is $\prod_{i=1}^rn_i$, hence bounded by $\kappa(\bn)^r$.
\end{proof}

\subsection{$\GL_2(\Q_p)$ vs $\SL_2(\Q_p)$} \label{subsection-SL2}
For the application in Section \ref{section-global}, we need to consider smooth admissible $\F$-representations of  $\mathscr{G}'=\prod_{i=1}^r\SL_2(\Q_p)$ and their Pontryagin duals. The results above translate to this situation. To explain this, we give the proof of the following analog of Theorem \ref{theorem-higher-r}.  If $H$ is a subgroup of $\mathscr{G}$, we denote by $H'$ the intersection $H\cap \mathscr{G}'$. 
 
\begin{theorem}\label{theorem-SL2}
Let $M'\in\mathfrak{C}^{\rm fg,tor}(\mathscr{G}')$. Then for all $i\geq 0$, 
\[\dim_{\F}H_i(\mathscr{T}_1(p^{\mathbf{n}})'/\mathscr{Z}_1',M')\ll \kappa({\mathbf{n}})^rp^{(2r-1)\kappa(\mathbf{n})}.\]
\end{theorem}
 
\begin{proof}
After twisting we may assume the central character of $M'$ is trivial. Then we may extend $M'$ to be an object $M^+$ in $\mathfrak{C}^{\rm fg,tor}(\mathscr{G}^+)$ by letting $\mathscr{Z}$ act trivially, where $\mathscr{G}^+:=\mathscr{G}'\mathscr{Z}$.  Note that $\mathscr{T}_1(p^{\mathbf{n}})$ is contained in $\mathscr{G}^+$.

It is easy to see that  $\mathscr{G}^+$ is a normal subgroup of $\mathscr{G}$ of finite index. Let $M:=\Ind_{\mathscr{G}^+}^{\mathscr{G}}M^+$, which is an object in $\mathfrak{C}^{\rm fg,tor}(\mathscr{G})$. We can apply Theorem \ref{theorem-higher-r} to $M$ and obtain 
\[\dim_{\F}H_i(\mathscr{T}_1(p^{\mathbf{n}})/\mathscr{Z}_1,M))\ll \kappa(\mathbf{n})^{r}p^{(2r-1)\kappa(\mathbf{n})}.\]
Since $M^+$ is a direct summand of $M|_{\mathscr{G}^+}$ by Mackey's theorem, we obtain 
\[\dim_{\F}H_i(\mathscr{T}_1(p^{\mathbf{n}})/\mathscr{Z}_1,M^+))\ll \kappa(\mathbf{n})^{r}p^{(2r-1)\kappa(\mathbf{n})}.\]
Since $\mathscr{T}_1(p^{\mathbf{n}})/\mathscr{Z}_1\cong \mathscr{T}_1(p^{\mathbf{n}})'/\mathscr{Z}_1'$, restricting $M^+$ to $\mathscr{G}'$ gives the result.  
\end{proof}  
 
\section{Application} \label{section-global}

Let $F$ be a number field of degree $r$, and $r_1$ (resp. $2r_2$) be the number of real (resp. non-real) places. 
Let $F_{\infty}=F\otimes_{\Q}\R$, so that $\SL_2(F_{\infty})=\SL_2(\R)^{r_1}\times\SL_2(\C)^{r_2}$.  Let $K_{\infty}$ be the standard maximal compact subgroup of $\SL_2(F_{\infty})$.

Let $\{\sigma_1,...,\sigma_r\}$ be the set of complex embeddings of $F$ and let ${\bf d}=(d_1,...,d_r)\in(\Z_{\geq 1})^r$ be an $r$-tuple indexed by the $\sigma_i$ such that $d_i=d_j$ when $\sigma_i$ and $\sigma_j$ are complex conjugate to each other.
Let $W_{\bf d}$ be the representation of $\SL_2(F_{\infty})$ obtained by forming the tensor product 
\[\big(\bigotimes_{\sigma_i\ {\rm real}}\Sym^{d_i}\C^2\big)\bigotimes \big(\bigotimes_{\{\sigma_i,\sigma_j\}\ {\rm complex}}\Sym^{d_i}\C^2\otimes \overline{\Sym}^{d_j}\C^2\big).\]
If $K_f\subset \SL_2(\mathbb{A}_f)$ is a compact open  subgroup
we write
\[Y(K_f):=\SL_2(F)\backslash \SL_2(\mathbb{A})/K_fK_{\infty},\]  and still use $W_{\bf d}$ to denote the local system  on $Y(K_f)$ attached to $W_{\bf d}$.

\begin{theorem}\label{theorem-global-A}
 If $F$ is not totally real and $K_f\subset \SL_2(\mathbb{A}_f)$ is a compact open subgroup, then   
\[\dim_{\C} H_i(Y(K_f),W_{\bf d})\ll_{\epsilon}\kappa({\bf d})^{r-1/2+\epsilon}\]
where $\kappa({\bf d})=\max_i\{d_i\}$.
\end{theorem}



\begin{proof}
The proof follows closely the one presented in \cite[\S5]{Mar-Annals}.  We content ourselves with briefly explaining the main ingredients.  Below we abuse the notation by letting the same letters to denote subgroups of $\SL_2$ obtained by intersection from $\GL_2$. Write $Y=Y(K_f)$ in the proof.
\begin{enumerate}

\item Choose a rational prime $p\geq 5$ which splits completely in $F$.  By \cite[Lem. 18]{Mar-Annals}, there exists a $p$-adic local system $V_{\bf d}$ defined over $\cO=W(\F)$, such that
\[\dim_{\C}H_i(Y,W_{\bf d})=\dim_{\cO[1/p]}H_i(Y,V_{\bf d}).\]
For this we need to choose a bijection between the set of complex embeddings $F\hookrightarrow \C$ and $p$-adic embeddings  $F\hookrightarrow\bQp$, as done in \cite[Lem. 18]{Mar-Annals}. 

 \item As is explained in \cite[\S5]{Mar-Annals}, by passing to an open subgroup we may assume that $K_f$ has  the form $\prod_{v}K_{f,v}$, with   $K_{f,v}=K_1$ ($\subset \SL_2(\Z_p)$) for all $v|p$. To achieve this, we could first choose $p$ such that $K_{f,v}=\SL_2(\Z_p)$, then pass to an open subgroup with $K_{f,v}=K_1$.

\item 
  Emerton's theory of completed  homology gives a bound (\cite[\S5, (34),(35)]{Mar-Annals})
\[\dim_{\cO[1/p]} H_q(Y,V_{\bf d})\leq \sum_{i+j=q}\dim_{\cO[1/p]} H_i(\mathscr{K}_1/\mathscr{Z}_1, \widetilde{H}_{j,\Q_p}\otimes V_{\bf d})\]  
where $\widetilde{H}_j$ is the $j$-th completed homology of Emerton with (trivial) coefficients in $\cO$, and $\widetilde{H}_{j,\Q_p}=\widetilde{H}_j\otimes_{\Z_p}\Q_p$. Note that $\widetilde{H}_j$ is a coadmissible module over $\cO[[\mathscr{K}_1]]$ and carries a natural compatible action of $\prod_{i=1}^r\SL_2(\Q_p)$.

\item Let $\bn=(n_1,...,n_r)$ where $n_i$ is the smallest integer such that $p^{n_i-1}-1\geq d_i$ (resp. $p^{n_i-1}-1\geq d_i/2$) if $\sigma_i$ is real (resp. complex). By \cite[Lem. 17]{Mar-Annals} we may choose lattice $\mathscr{V}_{d_i}\subset V_{d_i}$ such that $\mathscr{V}_{d_i}/p\subset \F[[K_1/K_1(p^{n_i})]]$. Let $L_{\bf d}$ be the reduction mod $p$ of $\otimes_{i=1}^r \mathscr{V}_{d_i}$.

\item Let $M_j$ be the reduction modulo $p$ of the image of $\widetilde{H}_j\ra \widetilde{H}_{j,\Q_p}$. We then have
\[\dim_{\cO[1/p]}H_i(\mathscr{K}_1/\mathscr{Z}_1,\widetilde{H}_{j,\Q_p}\otimes V_{\bf d})\leq \dim_{\F}H_i(\mathscr{K}_1/\mathscr{Z}_1, M_j\otimes L_{\bf d}).\]

\item Because $\SL_2(\C)$ does not admit discrete series, the assumption that $F$ is not totally real implies that  $\widetilde{H}_{j,\Q_p}$ is a torsion   $\cO[[\mathscr{K}_1]]\otimes_{\Z_p}\Q_p$-module, see \cite[Thm. 3.4]{CE-Annals}. So by Lemma \ref{lemma-p-free}, $M_j$ is a torsion $\Lambda$-module. Therefore our Theorem \ref{theorem-finalform} applies, via Theorem \ref{theorem-SL2}, and shows that 
\[\dim_{\F}H_i(\mathscr{K}_1,M_j\otimes L_{\bf d})\ll \kappa(\bn)^{2r}p^{(r-\frac{1}{2})\kappa(\bn)}\ll_{\epsilon} \kappa({\bf d})^{r-\frac{1}{2}+\epsilon}.\]

\end{enumerate}
\end{proof}

\begin{remark}\label{remark-global}
Our Theorem \ref{theorem-global-A} only gives interesting bound when all the $d_i$ tend to infinity at a parallel rate, while \cite[Thm. 1]{Mar-Annals} allows a subset of the weights $d_i$ to be fixed. 
Nonetheless,  this already includes the most interesting cases: for example,  when $F$ is imaginary quadratic, we do have $d_1=d_2$.  
\end{remark}

Next we deduce Theorem \ref{theorem-intro-A} in the introduction. We change slightly the notation. Let $Z_{\infty}$ be the centre of $\GL_2(F_{\infty})$, $K_f$ be a compact open subgroup of $\GL_2(\mathbb{A}_f)$ and let 
\[X=\GL_2(F)\backslash \GL_2(\mathbb{A})/K_fZ_{\infty}.\]  If ${\bf d}=(d_1,...,d_{r_1+r_2})$ is an $(r_1+r_2)$-tuple of positive even integers, let $S_{\bf d}(K_f)$ denote the space of cusp forms on $X$ which are of cohomological type with weight ${\bf d}$.  Then using the Eichler-Shimura isomorphism, see \cite[\S2.1]{Mar-Annals}, Theorem \ref{theorem-global-A} can be restated as follows. 

\begin{theorem}\label{theorem-global-B}
If $F$ is not totally real then for any fixed $K_f$ and ${\bf d}=(d_1,...,d_{r_1+r_2})$ as above, we have
\[\dim_{\C} S_{\bf d}(K_f)\ll_{\epsilon} \kappa({\bf d})^{r-1/2+\epsilon}.\]
\end{theorem} 
In particular, when ${\bf d}=(d,...,d)$ is parallel,  we obtain 
\[\dim_{\C} S_{\bf d}(K_f)\ll_{\epsilon} d^{r-1/2+\epsilon}\]
which strengthens Corollary 2 of \cite{Mar-Annals} by a power $d^{1/6}$.

\section{Appendix: A generalization of Breuil-Pa\v{s}k\={u}nas' construction} \label{section-appendix}
In this appendix, we generalize a construction of Breuil and Pa\v{s}k\=unas  in \cite{BP} for $\GL_2(F)$ to a finite product of $\GL_2(F)$, where $F$ is a local field  with finite residue field $k$ of characteristic $p$. Let $\cO$ be  the ring of integers in $F$ with $\varpi$ a fixed uniformizer. We  assume $p>2 $ for simplicity.

In \cite[\S9]{BP} (which is based on \cite{Pa04}), Breuil and Pa\v{s}k\={u}nas have proven the following theorem, see \cite[Cor. 9.11]{BP}.
\begin{theorem}\label{theorem-Breuil-Paskunas}
Let $\pi$ be an admissible representation of $\GL_2(F)$ such that $\smatr{\varpi}00{\varpi}$ acts trivially on $\pi$ and $\tilde{\sigma}:=\rsoc_{\GL_2(\cO)}\pi$. Then there exists an injection $\pi\hookrightarrow \Omega$ where $\Omega$ is a smooth representation of $\GL_2(F)$ such that $\Omega|_{\GL_2(\cO)}\cong \rInj_{\GL_2(\cO)}\tilde{\sigma}$ (an injective envelope of $\tilde{\sigma}$ in the category of smooth $\F$-representations of $\GL_2(\cO)$).
\end{theorem}

Let $r\geq 1$ be an integer. Denote \[\mathscr{G}=\prod_{i=1}^r\GL_2(F),\ \ \mathscr{K}=\prod_{i=1}^r\GL_2(\cO),\ \ \mathscr{Z}_{\varpi}=\prod_{i=1}^r \varpi^{\Z}\mathrm{Id}.\] 
The main result of this appendix is the following.
\begin{theorem}\label{theorem-BP-appendix}
Let $\pi$ be an admissible representation of $\mathscr{G}$ such that $\mathscr{Z}_{\varpi}$ acts trivially on $\pi$ and $\tilde{\sigma}:=\rsoc_{\mathscr{K}}\pi$. Then there exists an injection $\pi\hookrightarrow \Omega$ where $\Omega$ is a representation of $\mathscr{G}$ such that $\Omega|_{\mathscr{K}}\cong \rInj_{\mathscr{K}}\tilde{\sigma}$.  

If moreover $\pi$ admits a central character $\eta$, then we may require $\Omega$ to be an injective envelope of $\tilde{\sigma}$ in the category of smooth representations of $\mathscr{K}$ with the central character $\eta|_{\mathscr{K}}$.
\end{theorem}

The proof of Theorem \ref{theorem-BP-appendix} is an easy generalization of the original proof in \cite[\S9]{BP}. We only indicate the changes needed; for this we keep mostly the notation there.  We define the following subgroups of $\GL_2(F)$ (where $\p=\varpi\cO$ and $[\cdot]$ means Teichm\"uller lift):
\[I_1=\matr{1+\p}{\cO}{\p}{1+\p}, \ \ \ I=\matr{\cO^{\times}}{\cO}{\p}{\cO^{\times}},\]
\[H=\bigg\{\matr{[\lambda]}00{[\mu]},\lambda,\mu\in k^{\times}\bigg\}\subset I.\]
Let $\mathscr{I}_1,\mathscr{I},\mathscr{H}$ be respectively a product of $r$ copies of $I_1$, $I$, $H$, viewed as subgroups of $\mathscr{G}$. Then $\mathscr{H}$ has order prime to $p$ and provides a section for $\mathscr{I}\twoheadrightarrow \mathscr{I}/\mathscr{I}_1$. 
Let $\mathscr{R}_1$ denote the normalizer of $\mathscr{I}$ in $\mathscr{G}$, which as a group is generated by $\mathscr{I}$ and the  elements $\{t_i, 1\leq i\leq r\}$, where $t_i\in \mathscr{G}$ takes $\smatr{0}1{\varpi}0$ at the index $i$ and $\smatr{1}001$ at other indices.
Note that $t_it_j=t_jt_i$ and $t_i^2\in \mathscr{Z}_{\varpi}$. In other words, 
\[\Delta:=\langle t_i,1\leq i\leq r\rangle/\mathscr{Z}_{\varpi}\cong \prod_{i=1}^r\Z/2\Z.\]
It is clear that $\Delta$ normalizes $\mathscr{H}$ and there is an isomorphism of groups 
\begin{equation}\label{equation-appendix-isom}
 \mathscr{R}_1/\mathscr{I}_1\mathscr{Z}_{\varpi}\cong \mathscr{H}\rtimes \Delta.\end{equation} 
 

\begin{lemma}\label{lemma-appendix-(S)}
Let $\tau$ be a smooth admissible representation of $\mathscr{R}_1$ on which $\mathscr{Z}_{\varpi}$ acts trivially. Let  $\iota:\tau|_{\mathscr{I}}\hookrightarrow \rInj_{\mathscr{I}}(\tau|_{\mathscr{I}})$ be an injective envelope of $\tau|_{\mathscr{I}}$, then there exists an action of $\mathscr{R}_1$ on $\rInj_{\mathscr{I}} (\tau|_{\mathscr{I}})$ such that $\iota$ is $\mathscr{R}_1$-equivariant.
\end{lemma}
\begin{proof}
The proof is identical to that of \cite[Lem. 9.5]{BP}, using the (generalized) property (S)  defined in \cite[Def. 9.1]{BP} which holds as $p>2$ (see \cite[Prop. 9.2]{BP}). 
\end{proof}
Any $\F$-representation $V$ of $\mathscr{H}$ is semi-simple and decomposes as
$\oplus_{\chi}V_{\chi}$ where $\chi:\mathscr{H}\ra \F^{\times}$ runs over all the characters and $V_{\chi}$ denotes the $\chi$-isotypic subspace. For a character $\chi:\mathscr{H}\ra \F^{\times}$ and $t\in \mathscr{R}_1$, we let $\chi^t$ denotes the conjugate character:
\[\chi^t(g):=\chi(t^{-1}gt);\] 
this induces an action of $\Delta$ on the set of characters $\{\chi:\mathscr{H}\ra \F^{\times}\}$. We write $\langle \Delta.\chi\rangle$ for the $\Delta$-orbit generated by 
$\chi$, that is,  the set of characters (without multiplicities) $\{\chi^{t}, t\in\Delta\}$.

\begin{lemma}\label{lemma-appendix-action}
Let $V$ be a finite dimensional $\F$-representation of $\mathscr{H}$  such that:
\begin{equation}\label{equation-appendix-equality}\dim_{\F}V_{\chi}=\dim_{\F} V_{\chi^{t}},\ \ \forall t\in\Delta.\end{equation}
Then the action of $\mathscr{H}$ on $V$ can be extended to an action of $\mathscr{H}\rtimes\Delta$. 
\end{lemma}
\begin{proof}
We will construct an action of $\Delta$ on $V$ by constructing  inductively actions of $\Delta_s:=\prod_{i=1}^s\Z/2\Z$ (so $\Delta_r=\Delta$). The case $s=1$ is easy, see the proof of \cite[Lem. 9.6]{BP}.  Assume the action of $\Delta_{s-1}$ has been constructed. To construct the action of $\Delta_s$ amounts to defining an action of $t_s$ which has order $2$ and commutes with the given one of $\Delta_{s-1}$. It is clear that $V$ decomposes as a direct sum of subspaces  each of which has the form $\oplus_{\chi'\in \langle\Delta_{s}.\chi\rangle}V_{\chi'}$ (i.e. with respect to the action of $\Delta_s$), so it suffices to define the action of $t_s$ on each summand. Fixing a character $\chi$, we have two cases:

\begin{enumerate}
\item[--] If   $\chi^{t_s}=\chi$, then $\langle \Delta_{s}.\chi\rangle=\langle\Delta_{s-1}.\chi\rangle$, and we let $t_s$ acts trivially on $\oplus_{\chi'\in\langle\Delta_{s}.\chi\rangle}V_{\chi'}$.
\item[--] If $\chi^{t_s}\neq \chi$, then we have a disjoint union
\[\langle \Delta_s.\chi\rangle=\langle \Delta_{s-1}.\chi\rangle \cup \langle \Delta_{s-1}.\chi^{t_s}\rangle.\]
We choose an (arbitrary) $\F$-linear isomorphism $\phi_{\chi,\chi^{t_s}}: V_{\chi}\simto V_{\chi^{t_s}}$ and set $\phi_{\chi^{t_s},\chi}:=\phi_{\chi,\chi^{t_s}}^{-1}:V_{\chi^{t_s}}\simto V_{\chi}$. For any $t\in \Delta_{s-1}$, we consider
\[\xymatrix{V_{\chi}\ar^{\phi}_{\sim}[d]&V_{\chi^{t}}\ar@{-->}[d]\ar^{\sim}_{t^{-1}}[l]\\ 
V_{\chi^{t_s}}\ar_{\sim}^{t}[r]&V_{\chi^{t\cdot t_s}}}\]
and define  $\phi_{\chi^{t},\chi^{t\cdot t_s}}: V_{\chi^{t}}\simto V_{\chi^{t\cdot t_s}}$ to be   the composition
$t\circ\phi\circ t^{-1}$, resp.  $\phi_{\chi^{t\cdot t_s},\chi^{t}}:=\phi_{\chi^{t},\chi^{t\cdot t_s}}^{-1}:V_{\chi^{t\cdot t_s}}\simto V_{\chi^{t}}$. Clearly, putting them together uniquely determines a (compatible) action of $t_s$, hence of $\Delta_s$.
\end{enumerate}
 This finishes the proof by induction.
\end{proof}
\begin{lemma}\label{lemma-appendix-dim}
Let $\sigma$ be an irreducible $\F$-representation of $\mathscr{K}$ and $\rInj_{\mathscr{K}}\sigma$ an injective envelope of $\sigma$. Then $V:=(\rInj_{\mathscr{K}}\sigma)^{\mathscr{I}_1}$ satisfies the condition \eqref{equation-appendix-equality}.
\end{lemma}
\begin{proof}
Any irreducible $\sigma\in\Rep_{\F}(\mathscr{K})$ has the form $\otimes_{i=1}^r\sigma_i$ with each $\sigma_i\in \Rep_{\F}(K)$. We have 
\[(\rInj_{\mathscr{K}}\sigma)^{\mathscr{I}_1}=\big(\rInj_{\prod_{i=1}^r\GL_2(k)}\sigma\big)^{\mathscr{I}_1}\cong\big(\otimes_{i=1}^r\rInj_{\GL_2(k)}\sigma_i\big)^{\mathscr{I}_1}=\otimes_{i=1}^r(\rInj_{\GL_2(k)}\sigma_i)^{I_1};\]
here we have used the isomorphism $\rInj_{\prod_{i=1}^r\GL_2(k)}\sigma\cong \otimes_{i=1}^r\rInj_{\GL_2(k)}\sigma_i$. The result then follows from a similar result in the case $r=1$, see the proof of \cite[Lem. 9.6]{BP}.
\end{proof}

\begin{proof}[Proof of Theorem \ref{theorem-BP-appendix}] Since $\pi$ is admissible, we may take an injective envelope $\iota:\pi\hookrightarrow \Omega$. We will define an action of $\mathscr{R}_1$ on $\Omega$ which extends the given action of $\mathscr{I}$ on $\Omega$ and such that $\iota$ is $\mathscr{R}_1$-equivariant.  

  By  Lemma \ref{lemma-appendix-dim}, $V:=\Omega^{\mathscr{I}_1}$ satisfies the condition \eqref{equation-appendix-equality}. On the other hand, since $\pi$ carries an action of $\mathscr{G}$, $W:=\pi^{\mathscr{I}_1}$ also satisfies \eqref{equation-appendix-equality}.  So we may decompose $\mathscr{I}$-equivariantly $V$ as $W\oplus W'$ with $W'$ satisfying \eqref{equation-appendix-equality}, hence a decomposition: \[\Omega|_{\mathscr{I}}=\rInj_{\mathscr{I}}W\oplus\rInj_{\mathscr{I}}W'\]
  such that $\pi\subset \rInj_{\mathscr{I}}W$. Lemma \ref{lemma-appendix-action} allows us to define an action of $\Delta$, hence an action of $\mathscr{R}_1$ on $W'$  via \eqref{equation-appendix-isom}.  Then Lemma \ref{lemma-appendix-(S)} allows to extend the action of $\mathscr{R}_1$ on $\pi$ (resp. on $W'$) to the whole $\rInj_{\mathscr{I}}W$ (resp. $\rInj_{\mathscr{I}}W'$).  Putting them together, we obtain an action of $\mathscr{R}_1$ on $\Omega$ which makes $\iota$ to be $\mathscr{R}_1$-equivariant. Finally using  the ``amalgame'' structure of $\GL_2(F)$  which generalizes to $\mathscr{G}$, the two actions of $\mathscr{K}$ and $\mathscr{R}_1$ on $\Omega$ glue to an action of $\mathscr{G}$ on $\Omega$ (such that $\mathscr{Z}_{\varpi}$ acts trivially), as in \cite[Cor. 5.5.5]{Pa04}.  Note that in \cite{Pa04}, Cor. 5.5.5 is proved by passing to diagrams, but this can be circumvented  because we can simply write down the gluing action of $\mathscr{G}$ using the ones of $\mathscr{K}$ and $\mathscr{R}_1$.

The last assertion is clear, by taking the sub-space of  $\Omega_{\eta}\subset \Omega$ on which the centre of $\mathscr{G}$ acts via $\eta$.
\end{proof}

\bigskip


\bigskip

\bigskip

 \end{document}